\documentclass[12pt]{amsart}
\usepackage{ amsmath, amsthm, amsfonts, amssymb, color}
 \usepackage{mathrsfs}
\usepackage{amsfonts, amsmath}
 \usepackage{amsmath,amstext,amsthm,amssymb,amsxtra}
 \usepackage{txfonts} 
 \usepackage{amscd}
 \usepackage[colorlinks, citecolor=blue,pagebackref,hypertexnames=false]{hyperref}
 \allowdisplaybreaks
 \usepackage{pgf,tikz}
 \usepackage{multirow}
 \usepackage{diagbox} 

 \usepackage{tikz}
 \usepackage{cite}
\usepackage{xcolor}

\usetikzlibrary{decorations.pathreplacing}

 \usepackage[total={18cm,23.5cm}, left=1.5cm, right=1.5cm]{geometry}

\newcommand{\CC}{{\mathbb C^n}}

\newtheorem{thm}{Theorem}[section]
\newtheorem{cor}[thm]{Corollary}
\newtheorem{lem}[thm]{Lemma}
\newtheorem{prop}[thm]{Proposition}

\newtheorem{re}[thm]{Remark}
\newtheorem{defn}[thm]{Definition}

\numberwithin{equation}{section}

\newtheorem*{cor*}{Corollary}

\allowdisplaybreaks

\title[Weighted Fock Spaces]
{Bergman Kernels and Hankel Operators on Weighted Fock Spaces in Several Complex Variables}

\author{Guijun Liu,  Xiaofeng Wang and Zhicheng Zeng }

\address{Guijun Liu, Xiaofeng Wang and Zhicheng Zeng,
School of Mathematics and Information Science,
Guangzhou University, Guangzhou 510006, China}

\email{lgj19981998@163.com (G. Liu);
wxf@gzhu.edu.cn (X. Wang);
743706929@qq.com (Z. Zeng)}

\date{\today}


\makeatletter
\@namedef{subjclassname@2020}{\textup{2020} Mathematics Subject Classification}
\makeatother

\begin{document}

\begin{abstract}
We introduce a class of plurisubharmonic weights on $\CC$ and develop
a theory of the associated weighted Fock spaces $F_\varphi^p$. The
curvature assumptions are imposed on a $\mathcal C^2$-regularization at
bounded distance from the original weight. Thus the original weight need
not be smooth or strictly plurisubharmonic, and, even when it is of class
$\mathcal C^2$, its complex Hessian eigenvalues need not be uniformly
comparable. This provides a counterpart of
Christ's doubling theory in several complex variables. We establish a global upper estimate with decay away from the diagonal and a uniform lower estimate near the diagonal for the weighted Bergman kernel. These
estimates yield $L^p$ bounds for the kernel functions, boundedness of
the Bergman projection, and duality and complex interpolation for the
associated Fock spaces. For a certain class of weights beyond the uniformly controlled
curvature setting, we construct an integral solution operator for the
$\bar\partial$ equation and establish scale-adapted weighted $L^p$
estimates for $1\leq p\leq\infty$. As an application, for all $1\leq p,q<\infty$, we characterize the boundedness and compactness of Hankel operators from
$F_\varphi^p$ to $L_\varphi^q$ with possibly unbounded symbols in
terms of weighted $\mathrm{IDA}$ spaces.
\end{abstract}

\maketitle

\section{Introduction}
\subsection{The problem and the framework.}
Let $\varphi$ be a plurisubharmonic function on $\CC$. For
$0<p<\infty$, let $L_\varphi^p$ be the space of measurable functions
$f$ on $\CC$ such that
\[
 \|f\|_{L_\varphi^p}:=\left(\int_{\CC}|f(z)|^p e^{-p\varphi(z)}\,dV(z)\right)^{1/p}<\infty,
\]
where $dV$ denotes Lebesgue volume measure. For $p=\infty$, set
\[
 \|f\|_{L_\varphi^\infty}:=\operatorname*{ess\,sup}_{z\in\CC}|f(z)|e^{-\varphi(z)}<\infty.
\]
The associated weighted Fock space is $F_\varphi^p:=L_\varphi^p\cap H(\CC)$, where $H(\CC)$ denotes the space of entire functions. When $0<p<1$, the preceding expression is
understood as a quasinorm. The Gaussian weights
$\varphi(z)=\alpha|z|^2/2$, $\alpha>0$, give the classical Fock
spaces. 

The curvature of the weight plays a central role in the analysis of
weighted Fock spaces. In one complex dimension, the curvature of a
subharmonic weight \(\varphi\) is represented by its Riesz measure
\(\mu:=\Delta\varphi\). If $\varphi$ is nonharmonic and $\mu$ is doubling, its local size can be described by a scale function $\rho$ satisfying $\mu(D(z,\rho(z)))=1$ for all $z\in\mathbb C$. This scale determines the natural metric and enters the estimates for the Bergman kernel, the Bergman metric, and the $\bar\partial$ equation. This point of view originates in the work of Christ \cite{C91} and was developed further in \cite{Mm03,Mo09,CO11}. A decisive feature of this theory is that every such weight admits a smooth subharmonic regularization $\widetilde\varphi$ satisfying
\begin{align}\label{smooth}
\|\varphi-\widetilde\varphi\|_{L^\infty(\mathbb C)}<\infty,
\qquad
C_0^{-1}\rho^{-2}\leq\Delta\widetilde\varphi
\leq C_0\rho^{-2},
\end{align}
where $C_0\geq 1$, and $\Delta\widetilde\varphi$ remains doubling; see \cite[Theorem~14]{Mm03}. Since weights at bounded distance define the
same weighted Fock spaces with equivalent quasinorms, the
regularization in \eqref{smooth} provides a smooth geometric model
for the original space.

In several complex variables, for a strictly plurisubharmonic weight $\varphi$ of class $\mathcal C^2$, the natural counterpart of the Riesz measure is the complex Monge--Amp\`ere measure $\mu_\varphi:=(i\partial\bar\partial\varphi)^n.$ A natural problem is whether a doubling condition on $\mu_\varphi$ yields results analogous to those in one complex dimension. Christ observed that this problem presents several difficulties when $n>1$.
In particular, a doubling condition alone does not suffice, and the main obstacle is the passage from a scalar equation to a system of equations. Indeed, if $\lambda_1,\ldots,\lambda_n$ are the eigenvalues of the complex Hessian of $\varphi$, then the density of $\mu_\varphi$ is a fixed positive multiple of $\lambda_1\cdots\lambda_n$. Consequently, doubling of $\mu_\varphi$ controls neither the individual eigenvalues nor the ratio $\lambda_{\max}/\lambda_{\min}$. Thus the Monge--Amp\`ere measure
alone does not determine a scalar scale that normalizes the curvature
in every complex direction. This is the scalar versus system
obstruction emphasized in \cite{C91}.

Motivated by these observations, we formulate a counterpart of
Christ's doubling theory by adopting an analogue of \eqref{smooth}
in several complex variables as our basic hypothesis; see
Definition~\ref{def:regularizable} below. The pointwise
curvature assumptions are imposed on a $\mathcal C^2$ representative at
bounded distance, whose complex Hessian eigenvalues are uniformly
comparable at each point, while the original weight may be nonsmooth
or not strictly plurisubharmonic. Even when the original weight is of class
$\mathcal C^2$, the eigenvalues of its complex Hessian need not be
uniformly comparable. Within this framework, we develop a systematic theory of weighted Fock spaces in several complex variables.

\subsection{Regularizable weights and Bergman kernels}

For $z\in\CC$ and $r>0$, let
$B(z,r):=\{w\in\CC:|w-z|<r\}.$
Let $\mathcal S$ be the class of positive Borel measurable functions
$\rho$ on $\CC$ for which there exists $c_0\geq1$ such that
\begin{align}\label{rho123}
c_0^{-1}\rho(w)
\leq \rho(z)
\leq c_0\rho(w),
\qquad z\in B(w,\rho(w)).
\end{align}
For $\rho\in\mathcal S$ and $z\in\CC$, define
$\omega_\rho(z):=\rho(z)^{-2}i\partial\bar\partial|z|^2.$

\begin{defn}\label{def:regularizable}
Let $\rho\in\mathcal S$. A plurisubharmonic function $\varphi$ on
$\CC$ is called $\rho$-regularizable if there exists a strictly
plurisubharmonic function
$\widetilde\varphi\in\mathcal C^2(\CC)$ such that
\begin{align}\label{regu}
\|\varphi-\widetilde\varphi\|_{L^\infty(\CC)}
<\infty
\quad\text{and}\quad
m\omega_\rho
\leq i\partial\bar\partial\widetilde\varphi
\leq M\omega_\rho
\end{align}
for some constants $0<m\leq M<\infty$. We call
$\widetilde\varphi$ a $\rho$-regularization  of $\varphi$.
\end{defn}

Let $\mathcal W^*(\CC)$ be the class of all plurisubharmonic
functions $\varphi$ on $\CC$ for which there exist
$\rho\in\mathcal S$ and a $\rho$-regularization
$\widetilde\varphi$ such that the complex Monge--Amp\`ere measure
$\mu_{\widetilde\varphi}:=(i\partial\bar\partial\widetilde\varphi)^n$
is doubling. The curvature bounds in
Definition~\ref{def:regularizable}, together with the local
comparability of $\rho$ in \eqref{rho123}, imply that
\[
C^{-1}
\leq
\mu_{\widetilde\varphi}\bigl(B(z,\rho(z))\bigr)
\leq C,
\qquad z\in\CC,
\]
for some constant $C>0$. Thus $\rho$ describes the local scale at
which the Monge--Amp\`ere mass of $\widetilde\varphi$ is uniformly
comparable to one. This is the analogue in several complex variables
of the unit mass normalization of the Riesz measure in Christ's
theory. When $n=1$, the regularization in \eqref{smooth} shows that
$\mathcal W^*(\mathbb C)$ contains the classical class of
subharmonic weights with doubling Riesz measure. Basic examples include the radial weights
$\varphi(z)=|z|^\alpha$ for $\alpha>0$. When $n\geq2$, another example
is
\begin{align}\label{noncom}
\varphi(z)
=
|z|^2+4\sum_{j=1}^n\cos(\operatorname{Re}z_j),
\qquad z\in\CC.
\end{align}
Its complex Hessian is diagonal with
eigenvalues $1-\cos(\operatorname{Re}z_j), 1\leq j\leq n,$
which are not uniformly comparable.

Throughout the paper, for each $\varphi\in\mathcal W^*(\CC)$, we fix
a function $\rho\in\mathcal S$ and a $\rho$-regularization
$\widetilde\varphi$ such that $\mu_{\widetilde\varphi}$ is doubling.
Let
\[
K_\varphi:\CC\times\CC\longrightarrow\mathbb C
\]
be the weighted Bergman kernel associated with $\varphi$, namely, the
reproducing kernel of $F_\varphi^2$. 

Our first objective is to establish a global upper estimate and a
uniform lower estimate near the diagonal for \(K_\varphi\) when
\(\varphi\in\mathcal W^*(\CC)\). Although
\(F_\varphi^2\) and \(F_{\widetilde\varphi}^2\) coincide with
equivalent norms, their reproducing kernels correspond to different
inner products. Thus norm equivalence alone does not transfer the
required pointwise estimates from \(K_{\widetilde\varphi}\) to
\(K_\varphi\). We address this difficulty by developing new arguments
for the corresponding kernel estimates.

An extensive literature is devoted to upper estimates for weighted
Bergman kernels when the weight is a $\mathcal C^2$
plurisubharmonic function. Delin \cite{D98} obtained a global upper
estimate for $\mathcal C^2$ strictly plurisubharmonic weights. The
constant in his estimate depends on the reciprocal of a local
positive lower bound for $\lambda_{\min}$, the smallest eigenvalue
of the complex Hessian. Consequently, the estimate does not apply
when $\lambda_{\min}=0$, as happens at the origin for
$\varphi(z)=|z|^\alpha$ with $\alpha>2$. Under uniform upper and lower curvature bounds, Lindholm \cite{L01} obtained exponential estimates away from the diagonal together with
a lower estimate on the diagonal. These curvature bounds correspond to the constant scale $\rho\equiv1$, whereas our framework allows the local scale to vary with $z$. Later, Dall'Ara \cite{DA15} used coercivity of the weighted Kohn Laplacian to obtain upper estimates for Bergman kernels associated with $\mathcal C^2$ plurisubharmonic weights that need not be strictly plurisubharmonic. In the specialization of his admissible weight framework where the complex Hessian eigenvalues are uniformly comparable and $\Delta\varphi\in RH_\infty$, the decay is governed by a scalar local scale $\rho$ defined through $\Delta\varphi$.
It follows from \cite[Proposition~12]{DA15} that this scale $\rho$ is
uniformly bounded above. In our framework, the local scale need not be bounded above. For
example, if $\varphi(z)=(1+|z|^2)^{\alpha/2}, 0<\alpha<2,$ then one may take $\rho(z)=(1+|z|)^{1-\alpha/2},$ and hence $\rho(z)\to\infty$ as $|z|\to\infty$.
 More recently, Phung \cite{Phung24} weakened the assumptions on
$\Delta\varphi$ while retaining the uniform comparability of the
complex Hessian eigenvalues. We observe that the argument used to
obtain the kernel estimate in \cite{Phung24} cannot be applied to the
weight in \eqref{noncom}. Indeed, its original complex Hessian
eigenvalues are not uniformly comparable, while $|z|^2$ is a
regularization at bounded distance. This illustrates the additional flexibility
gained by imposing curvature assumptions on the regularization rather than
on the original weight.

These developments laid important foundations for the theory in
higher dimensions and provided essential guidance for the present
work. However, to the best of our knowledge, among the classes
discussed above, a uniform lower estimate near the diagonal has previously
been established only under uniform upper and lower curvature bounds, as in
\cite{Sv12}. Without such bounds, the
preceding extensions provide global upper estimates but no
corresponding uniform lower estimate.

To state our result, for $z,w\in\CC$, write $K_z^\varphi(w):=K_\varphi(w,z).$

\begin{thm}\label{main}
Let $\varphi\in\mathcal W^*(\CC)$ with $\rho\in\mathcal S$. Then there exist positive
constants $C_1,C_2,\varepsilon$, and $r$ such that
\begin{align}\label{main1}
|K_z^\varphi(w)|
\leq
C_1
\frac{e^{\varphi(z)+\varphi(w)}}
{\rho(z)^n\rho(w)^n}
e^{-\varepsilon d_\varphi(z,w)},
\qquad z,w\in\CC,
\end{align}
and
\begin{align}\label{main2}
|K_z^\varphi(w)|
\geq
C_2
\frac{e^{\varphi(z)+\varphi(w)}}
{\rho(z)^n\rho(w)^n},
\qquad
d_\varphi(z,w)<r.
\end{align}
Here $d_\varphi$ denotes the Riemannian distance induced by the K\"ahler metric
$i\partial\bar\partial\widetilde\varphi$.
\end{thm}
The scale geometry in Theorem~\ref{main} is intrinsic in the sense that
it agrees, up to uniform equivalence, with the Bergman geometry. Let $\varrho$ denote the distance induced by the Bergman metric, and let $d_\rho$ denote the length distance induced by $\omega_\rho$. By \eqref{regu}, $d_{\varphi}$ and $d_\rho$ are uniformly equivalent, while Lemma~\ref{bergmetric} shows that there exists a constant $C\geq1$ such that
\[
C^{-1}d_\varphi(z,w)
\leq \varrho(z,w)
\leq Cd_\varphi(z,w),
\qquad z,w\in\CC.
\]
 Consequently, $d_\rho$, $d_{\varphi}$, and
$\varrho$ are mutually uniformly equivalent. In particular, when
$\rho\equiv1$, all three distances are comparable to the Euclidean
distance.

The upper and lower estimates in Theorem~\ref{main} are obtained by
different methods. To prove \eqref{main1}, we first combine a weighted
mean value inequality with a weighted H\"ormander $L^2$ estimate to
obtain the required estimate for the regularized weight
$\widetilde\varphi$. We then relate $K_\varphi$ to
$K_{\widetilde\varphi}$ through the invertible Toeplitz operator
$T_b$, where $b=e^{-2(\varphi-\widetilde\varphi)}.$
Schur estimates and a Neumann series argument give the required
localization of $T_b^{-1}$; see Proposition~\ref{Gue}. This Toeplitz
transfer yields the pointwise estimate for the original weight
$\varphi$, which may be nonsmooth. To prove \eqref{main2}, we solve a one-point interpolation problem and construct peak functions with uniformly controlled $L^2$ norms;
see Lemma~\ref{peak}. This construction is inspired by the method of
\cite{Sv12} for $\rho\equiv1$. When the local scale $\rho$ is allowed
to tend to zero or infinity, however, the construction faces
substantial difficulties caused by the degeneration or expansion of
the local geometry. Its implementation in our setting is therefore
essentially different from the $\rho\equiv1$ argument in \cite{Sv12}.

\subsection{The \texorpdfstring{$\bar\partial$}{d-bar} equation and Hankel
operators}
Let \(P_\varphi\) denote the Bergman projection from \(L_\varphi^2\) onto \(F_\varphi^2\). For a suitable symbol \(f\), the corresponding Hankel operator is defined by
$$
H_f:=(I-P_\varphi)M_f,
$$
where \(M_f\) denotes multiplication by \(f\). The boundedness and compactness of Hankel operators with general, possibly unbounded, symbols are classical problems in operator theory on analytic function spaces. The use of integral distance to holomorphic functions, abbreviated IDA, goes back to Luecking's work on Bergman spaces \cite{L92}. Hu and Virtanen applied this method to Hankel operators on Fock spaces;
see \cite{HV22,HV23}. In particular,  Hu and Virtanen \cite{HV23} characterized the
boundedness and compactness of Hankel operators
$H_f:F_\varphi^p\to L_\varphi^q$ for general symbols and weights
$\varphi\in\mathcal W_0$. This class consists of convex \(\mathcal C^2\) weights satisfying
$mI_{2n}\leq D_{\mathbb R}^2\varphi\leq MI_{2n}$
for some \(0<m\leq M<\infty\), where \(D_{\mathbb R}^2\varphi\) denotes the real Hessian of \(\varphi\) and \(I_{2n}\) is the \(2n\times2n\) identity matrix. Related characterizations for radial Fock-type weights $\Psi$ were obtained in \cite{ZHW} in the Hilbert space setting
$H_f:F_\Psi^2\rightarrow L_\Psi^2.$ To the best of our knowledge, beyond the classes treated there, corresponding boundedness and compactness characterizations from
$F^p_\varphi$ to $L^q_\varphi$ for weights with nonuniform curvature
in several complex variables have not previously been available.

Our next objective is to extend the results of \cite{HV23} to weights whose curvature need not be uniformly bounded. For this purpose, we work with the subclass \(\mathcal W_{\mathrm{gc}}^*(\CC)\) of \(\mathcal W^*(\CC)\), defined precisely in Section~\ref{canonical}. Informally, a weight belongs to this class if it admits a \(\rho\)-regularization \(\widetilde\varphi\) for which there exists \(F\in H(\CC)\) such that
\begin{align}\label{RALhess}
m\rho^{-2}I_{2n}
\leq
D_{\mathbb R}^2
\bigl(\widetilde\varphi-\operatorname{Re}F\bigr)
\leq
M\rho^{-2}I_{2n}
\end{align}
for some \(0<m\leq M<\infty\). When \(\rho\equiv1\), this class contains
$\mathcal W_0+\operatorname{Re}H(\CC),$ whose elements need not be convex. More generally, it contains weights of the form
$$\varphi(z)=|z|^\alpha+\operatorname{Re}F(z),
\qquad \alpha>1,\quad F\in H(\CC).
$$

For each \(1\leq p<\infty\), let \(\mathcal Q_p\) be the symbol class
defined in \eqref{Qp}. If \(f\in\mathcal Q_p\), then the Hankel
operator \(H_f\) is densely defined on \(F_\varphi^p\). In particular,
every bounded measurable function belongs to \(\mathcal Q_p\). Using
the weighted IDA spaces \(\mathrm{IDA}_r^{s,q,\alpha}\),
\(\mathrm{BDA}_r^{q,\alpha}\), and \(\mathrm{VDA}_r^{q,\alpha}\)
defined in Section~\ref{canonical}, we obtain the following
characterization.
   
 \begin{thm}\label{Hfpq}
Let $\varphi\in \mathcal W^*_{\mathrm{gc}}(\mathbb C^n)$. \\
 \textup{(A)} If $1\leq p\leq q<\infty$, $f\in\mathcal Q_p$, and $\gamma=n(1/q-1/p)$, then $ H_f : F_\varphi^p \to L_\varphi^q $ is bounded if and only if \( f \in \mathrm{BDA}_r^{q, 2\gamma} \), and \( H_f : F_\varphi^p \to L_\varphi^q \) is compact if and only if \( f \in \mathrm{VDA}_r^{q, 2\gamma} \)  for some (or any) $r>0$. Moreover, there exists a constant $C>0$ such that
\[
C^{-1}\|f\|_{\mathrm{BDA}_r^{q, 2\gamma }}\leq\|H_f\|_{F_\varphi^p \to L_\varphi^q} \leq C \|f\|_{\mathrm{BDA}_r^{q, 2\gamma }}.
\]
 \textup{(B)} If $1\leq q<p<\infty$, $f\in\mathcal Q_q$, and $s=\dfrac{pq}{p-q}$, then $ H_f : F_\varphi^p \to L_\varphi^q $ is bounded if and only if it is compact, and these conditions are equivalent to $f\in\mathrm{IDA}_{r}^{s, q, 0} $ for some (or any) $r>0$. Moreover, there exists a constant $C>0$ such that 
 \[
C^{-1}\|f\|_{\mathrm{IDA}_{r}^{s, q, 0}}\leq\|H_f\|_{F_{\varphi}^p \to L_{\varphi}^q}\leq C\|f\|_{\mathrm{IDA}_{r}^{s, q, 0}}.
\]
\end{thm}
 
We first note that, when $n\geq2$, the characterization 
in Theorem~\ref{Hfpq} is new even in the Hilbert space case
\[
H_f:F_{|z|^\alpha}^2\longrightarrow L_{|z|^\alpha}^2,
\qquad 2<\alpha<4.
\]
These weights are not covered by \cite{ZHW}, which, among other
hypotheses on the radial profile, assumes
$\Psi'>0$, $\Psi''\geq0$, and $\Psi'''\geq0$. For $|z|^\alpha$ with
$2<\alpha<4$, one has $\Psi(t)=2t^{\alpha/2}$ under our
normalization, and hence $\Psi'''(t)<0$.  In the special case \(\rho\equiv1\), Theorem~\ref{Hfpq} extends the
results of Hu and Virtanen \cite{HV23} from \(\mathcal W_0\) to
\(\mathcal W_0+\operatorname{Re}H(\CC)\) and provides a partial
answer to their conjecture. A noteworthy consequence is an extension
of the Berger--Coburn phenomenon. This phenomenon was established for
the classical Fock spaces in \cite{BC87,HV21} and for weights in
\(\mathcal W_0\) in \cite{HV23}. Our result extends it to
\(\mathcal W_0+\operatorname{Re}H(\CC)\), as follows.

\begin{cor}\label{intro-berger-coburn}
Let $\varphi\in\mathcal W_0+\operatorname{Re}H(\CC)$,
$f\in L^\infty(\CC)$, and $1\leq p,q<\infty$. Then
\[
H_f:F_\varphi^p\longrightarrow L_\varphi^q
\text{ is compact}
\quad\Longleftrightarrow\quad
H_{\overline f}:F_\varphi^p\longrightarrow L_\varphi^q
\text{ is compact}.
\]
\end{cor}

Previous work for other classes of weights mainly considered the
simultaneous boundedness or compactness of $H_f$ and
$H_{\overline f}$, as well as Hankel operators with conjugate
holomorphic symbols, through mean oscillation and gradient conditions;
see, for example,
\cite{CH21,BC94,W05,HW18,LWH25,A86,LW24,SY}.
In Section~\ref{canonical}, Theorem~\ref{Hfpq} yields the corresponding
characterizations in our setting; see
Theorems~\ref{ffpq}, \ref{ffqp}, and \ref{fpq1}.

The proof of Theorem~\ref{Hfpq} faces two difficulties that do not
arise in the uniformly convex setting with $\rho\equiv1$ considered in
\cite{HV23}. First, the Fock spaces associated with different
exponents need not be nested. For example, let
$\varphi(z)=|z|^3$. Then
$\rho(z)\simeq(1+|z|)^{-1/2}$, and for every
$1<p<q<\infty$,
$$
F_\varphi^p\not\subset F_\varphi^q
\qquad\text{and}\qquad
F_\varphi^q\not\subset F_\varphi^p.
$$
See Remark~\ref{FpFq} for further discussion. Consequently, the cases
$p\leq q$ and $q<p$ must be treated separately. In particular, when
$p\leq q$, one can no longer use the inclusion available for
$\rho\equiv1$ to reduce part of the argument to the case $p=q$.
Moreover, the local quantities defining the IDA spaces must be
normalized by the varying scale $\rho$. The second difficulty concerns the construction of a solution operator for the $\bar\partial$ equation. For a sufficiently regular symbol $f$
and a holomorphic function $g$, the identity
$\bar\partial(H_fg)=g\bar\partial f$ suggests studying Hankel
operators through this equation. For $\varphi\in\mathcal W_0$, Hu and
Virtanen \cite{HV23} obtained weighted $L^p$ estimates for the
Berndtsson--Andersson operator $A_\varphi$, in the notation of
\cite[p.~92]{BA82}. Their argument relies essentially on the uniform convexity of $\varphi$, and hence does not apply directly to weights in $\mathcal W_{\mathrm{gc}}^*(\CC)$. Nevertheless, their use of the Berndtsson--Andersson operator provides the starting point for our construction. Indeed, for $\varphi\in\mathcal W_{\mathrm{gc}}^*(\CC)$, we choose an entire function $F$ and a $\rho$-regularization $\widetilde\varphi$ satisfying \eqref{RALhess}. We then introduce a new operator $T_\varphi$ obtained by conjugating
the Berndtsson--Andersson operator. For a $(0,1)$-form $\omega$, define
\[
T_{\varphi}(\omega):=e^F A_{\widetilde\varphi-\operatorname{Re}F}(e^{-F}\omega).
\]
We prove that
\[
\|T_{\varphi}(\omega)\|_{L_\varphi^p}\lesssim\|\rho\omega\|_{L_\varphi^p},\qquad 1\leq p\leq\infty,
\]
Under the hypotheses of Lemma~\ref{solve1}\textup{(B)}, we also have
$\bar\partial T_{\varphi}(g\bar\partial f)=g\bar\partial f.$
Thus $T_{\varphi}(g\bar\partial f)$ provides the particular solution required in the analysis of $H_f$ and leads to the subsequent
IDA characterizations.

The paper is organized as follows. Section~\ref{Pre} develops the
geometry induced by the scale function $\rho$, including consequences
of the doubling condition, covering lemmas, metric comparison
properties, and weighted mean value inequalities. In
Section~\ref{esti}, we establish global upper estimates and uniform
lower estimates near the diagonal for the Bergman kernel, prove
the equivalence between the Bergman distance and the distance
$d_\rho$ determined by the scale function, and derive weighted
$L^p$ estimates for the Bergman kernel. Section~\ref{sec:proj}
is devoted to the boundedness of the Bergman projection, duality, and
complex interpolation. Finally, Section~\ref{canonical} establishes weighted \(L^p\)
estimates for the \(\bar\partial\) equation, develops the required
IDA and Carleson measure characterizations, and applies them to
Hankel operators.

Throughout the paper, the notation $A\lesssim B$ means that $A\leq CB$ for a positive constant $C$ independent of the points, functions, and operators under consideration. We write $A\simeq B$ if both $A\lesssim B$ and $B\lesssim A$ hold. Unless explicitly stated otherwise, implicit constants may depend on the dimension and on the structural constants in the definitions of $\mathcal S$ and $\mathcal W^*(\CC)$.

\section{Preliminaries}\label{Pre}
In this section, we establish several foundational results that will
be used throughout the paper. These include geometric properties of
the scale function $\rho$ associated with weights
$\varphi\in\mathcal W^*(\CC)$, covering lemmas for metric balls
determined by $\rho$, and weighted mean value inequalities. We begin
by recalling the definition of a doubling measure.

\begin{defn}
A nonnegative Borel measure $\mu$ on $\CC$ is called doubling if there
exists a constant $C>0$ such that
\[
\mu(B(z,2r))\leq C\mu(B(z,r))
\]
for all $z\in\CC$ and $r>0$. We denote by $C_\mu$ the least constant
for which this inequality holds.
\end{defn}

The following result is an important and useful estimate for doubling measures.

\begin{lem}\label{estimate:doubling}
	Let $\mu$ be a nonatomic doubling measure on $\mathbb C^n$ with
$\operatorname{supp}\mu=\mathbb C^n$. Then there are constants $C>1$ and $0<\delta_1,\delta_2<\infty$, which only depend on $C_{\mu}$, such that
	 if $B$ and $B'$ are open balls of radii $r$ and $r'$, respectively, such that
	 $B'\subset B$ and $r'<r$, then 
	 $$
	 C^{-1}(r'/r)^{\delta_1}\mu(B)\le\mu(B')\le C(r'/r)^{\delta_2}\mu(B).               
	 $$
\end{lem}
\begin{proof}
The first inequality follows from \cite[(4.16)]{H01}.
The second inequality is the reverse doubling estimate for nonatomic
doubling measures with full support; see \cite[Lemma~2.1]{C91}.
\end{proof}

The measure $\mu_{\widetilde\varphi}$ used below is nonatomic and has
full support, since (1.3) makes it locally equivalent to the positive
density $\rho^{-2n}dV$.
We denote by $d_{\varphi}$ the Riemannian distance induced by the
K\"ahler metric
\[g_{\widetilde\varphi}(z,\xi):=\sum_{j,k=1}^n\frac{\partial^2\widetilde\varphi(z)}{\partial z_j\partial\overline z_k}\xi_j\overline{\xi}_k.
\]
Since
$i\partial\bar\partial\widetilde\varphi\simeq \omega_\rho,$
we have
$ d_{\varphi}(z,w)\simeq d_\rho(z,w),$
where
\[d_\rho(z,w):=\inf_\gamma\int_0^1\frac{|\gamma'(t)|}{\rho(\gamma(t))}\,dt,\]
and the infimum is taken over all piecewise $C^1$ curves joining $z$ and $w$. 
        
This auxiliary distance $d_\rho$ can be further estimated in terms of the Euclidean distance as follows. For convenience we write $B^r(w)=B(w,r\rho(w))$ and
$B(w)=B^1(w)$ for $w\in\CC$ and $r>0$. 
\begin{lem}\label{dzw}
Let $\varphi\in \mathcal W^*(\CC)$ with $\rho\in \mathcal{S}$. Then there exist constants $C>0$ and $0<\alpha\leq1\leq\beta$ such that for all $z,w\in \mathbb C^n$,
$$
C^{-1}\min\{t,t^\alpha\}
\leq d_\rho(z,w)
\leq C\max\{t,t^\beta\},\quad t=\frac{|z-w|}{\rho(w)}.
$$
\end{lem}

\begin{proof}
For $t<1$, the conclusion follows immediately from the local comparability of $\rho$. Therefore, it remains to consider the case $t\geq 1$.

Set 
$$\mu(B(z,r)):=\int_{B(z,r)}(i\partial\bar\partial\widetilde\varphi)^n.$$
Then $\mu$ is a doubling measure on $\CC$. By the definition of $\rho$,
\begin{equation}\label{eqa1}
	\mu(B(z))=\int_{B(z)}(i\partial\bar\partial\widetilde\varphi)^n\simeq \int_{B(z)}(\omega_\rho)^n\simeq 1,\quad z\in\CC.
\end{equation}
Set $B_1 = B(z, |z - w|/4)$. Then by Lemma \ref{estimate:doubling} and \eqref{eqa1},
\begin{align}\label{5zwa}
\mu(B_1) &\simeq \mu(B(z, 2|z - w|))
\geq \mu(B(w, |z - w|))
\gtrsim \mu(B(w))\left[ |z - w| / \rho(w) \right]^{\delta_2}\simeq \left[ |z - w| / \rho(w) \right]^{\delta_2},
\end{align}
for $0<\delta_2 <\infty$ where $\delta_2$ is from Lemma \ref{estimate:doubling}. For any $x \in \overline{B}_1$, by Lemma \ref{estimate:doubling} and \eqref{eqa1},
\begin{align}\label{5zwb}
\mu(B_1) &\simeq \mu(B(x, |z - w|)) 
\lesssim \left( |z - w| / \rho(x) \right)^{M_1} ,
\end{align}
where $M_1=\delta_1$ if $|z-w|>\rho(x)$ and $M_1=\delta_2$ if $|z-w|\leq \rho(x)$. Thus, combining \eqref{5zwa} and \eqref{5zwb}, we obtain
$$
\left[ |z - w| / \rho(x) \right]^{M_1} \gtrsim \left[ |z - w| / \rho(w) \right]^{\delta_2},
$$
which implies that
\begin{align}\label{rho B}
\rho(x) \lesssim |z - w| \left[ |z - w| / \rho(w) \right]^{-\alpha},\quad x\in \overline{B}_1, 
\end{align}
where $\alpha=\min\{\delta_2/\delta_1,1\}$. By \eqref{rho B}, for any piecewise $C^1$ curve $\gamma : [0,1] \to \mathbb{C}^n$ with $\gamma(0) = z$ and $\gamma(1) = w$, 
$$\int_{0}^{1}\frac{|
	\gamma^{\prime}(t)|}{\rho(\gamma(t))}dt\geq \int_{\{t:\gamma(t)\in \overline{B}_1\}}\frac{|
	\gamma^{\prime}(t)|}{\rho(\gamma(t))}dt
\geq\inf_{x\in\overline{B}_1}\rho(x)^{-1}\int_{\{t:\gamma(t)\in \overline{B}_1\}}|
\gamma^{\prime}(t)|dt\gtrsim \left( \frac{|z - w|}{\rho(w)} \right)^{\alpha}. $$
It follows that
$$
d_\rho(z,w) \ge C^{-1}\Big(\frac{|z-w|}{\rho(w)}\Big)^\alpha .
$$

Set $B_2 = B(z, |z - w|)$. Then by Lemma \ref{estimate:doubling} and \eqref{eqa1},
\begin{align}\label{5zw}
	\mu(B_2) &\leq \mu(B(w, 2|z - w|))
	\lesssim \mu(B(w, |z - w|))
	\lesssim \mu(B(w))\left[ |z - w| / \rho(w) \right]^{\delta_1}\simeq \left[ |z - w| / \rho(w) \right]^{\delta_1}.
\end{align}
For any $x \in \overline{B}_2$, by Lemma \ref{estimate:doubling} and \eqref{eqa1},
\begin{align}\label{5zwc}
	\mu(B_2) &\simeq \mu(B(x, |z - w|)) 
	\gtrsim \left( |z - w| / \rho(x) \right)^{M_2} ,
\end{align}
where $M_2=\delta_2$ if $|z-w|>\rho(x)$ and $M_2=\delta_1$ if $|z-w|\leq \rho(x)$. Thus, combining \eqref{5zw} and \eqref{5zwc}, we obtain
$$
\left[ |z - w| / \rho(x) \right]^{M_2} \lesssim \left[ |z - w| / \rho(w) \right]^{\delta_1},
$$
which implies that
\begin{align}\label{rho Ba}
	\rho(x) \gtrsim |z - w| \left[ |z - w| / \rho(w) \right]^{-\beta},\quad x\in \overline{B}_2 ,
\end{align}
where $\beta=\max\{\delta_1/\delta_2,1\}$. Applying this and taking $\gamma(t) = z + t(w - z)$, we deduce that
$$d_\rho(z,w)\leq |z-w|\int_{0}^{1}\frac{dt}{\rho(\gamma(t))}\lesssim  \left(\frac{|z - w|}{\rho(w)} \right)^{\beta}.$$
Note that $\beta=1/\alpha\geq \alpha$. This completes the proof.
\end{proof}
\begin{lem}\label{rhoB}
Let $\varphi\in \mathcal W^*(\CC)$ with $\rho\in \mathcal{S}$. Then the following statements hold.\\
\textup{(A)} There exists a constant $C_r>0$ depending only on $r$ such that 
\begin{align}\label{rhozw}
C_r^{-1}\rho(w)\leq \rho(z)\leq C_r \rho(w)
\end{align}
for any $r>0$ and $z\in B^r(w)$. \\ 
\textup{(B)} Given $r>0$, there exist $m_1, m_2 > 0$, depending on $r$, such that
\begin{align}\label{Bzwz}
 B^r(z) \subseteq B^{m_1r}(w), \qquad B^r(w) \subseteq B^{m_2r}(z), \quad \text{whenever}\, z \in B^r(w).
\end{align}
\textup{(C)} There exist $c_1(r),c_2(r)>0$ such that 
\begin{align}\label{drhoB}
B_{\rho}(z, c_1(r)) \subseteq B^r(z) \subseteq B_{\rho}(z, c_2(r))
\end{align}
 for any $z \in \CC$ and $r>0$, where $B_\rho(z,r):=\{w:\,d_\rho(z,w)<r\}$.
\end{lem}
\begin{proof}
	In view of \eqref{rho B} with $x=z$, we have
	\begin{align}\label{z-w}
	\frac{\rho(z)}{\rho(w)}\lesssim \left(\frac{|z - w|}{\rho(w)} \right)^{1-\alpha},\quad \textup{whenever}\, |z-w|\ge \rho(w),
	\end{align}
	where $\alpha=\min\{\delta_2/\delta_1,1\}\leq 1.$ In particular, $\rho(z)\leq C_{1}(r) \rho(w)$ whenever $z\in B^r(w)\backslash B(w)$. Applying \eqref{rho Ba} with $x=z$, we have
	\begin{align}\label{w-z}
	\frac{\rho(z)}{\rho(w)}\gtrsim \left(\frac{|z - w|}{\rho(w)} \right)^{1-\beta}, \quad \textup{whenever}\, |z-w|\ge \rho(w),
	\end{align}
	where $\beta=\max\{\delta_1/\delta_2,1\}\geq 1.$ In particular, $\rho(z)\geq C_{2}(r) \rho(w)$ whenever $z\in B^r(w)\backslash B(w)$. Since $\rho\in \mathcal{S}$, we also have $\rho(z)\simeq \rho(w)$. This completes the proof of (A).
	
	For any $\xi \in B^r(z)$, by (A), we have $|\xi-w|\leq |\xi -z|+|z-w|< r\rho(z)+r\rho(w)\leq r(C_r+1)\rho(w)$. Thus, $B^r(z) \subseteq B^{m_1r}(w)$ with $m_1=C_r+1$. For any $\xi \in B^r(w)$, by (A), we have $|\xi-z|\leq |\xi -w|+|z-w|< r\rho(w)+r\rho(w)\leq 2rC_r\rho(z)$. Thus, $B^r(w) \subseteq B^{m_2r}(z)$ with $m_2=2C_r$. This proves (B).
	
	Choose $c_1(r)<C^{-1}\min\{r,r^\alpha\}$. If
$\xi \in B_{\rho}(z,c_1(r))$, then the lower estimate in
Lemma~\ref{dzw}, applied with center $z$, gives $|\xi-z|<r\rho(z)$. Hence
$B_{\rho}(z,c_1(r))\subseteq B^r(z)$. Conversely, if $\xi\in B^r(z)$,
Lemma~\ref{dzw} gives $d_\rho(z,\xi)\leq C\max\{r,r^\beta\}$.
Choosing $c_2(r)>C\max\{r,r^\beta\}$ gives
$B^r(z)\subseteq B_\rho(z,c_2(r))$. This proves \textup{(C)}.
\end{proof}

The following lemma is of the same type as the covering lemma in \cite[Theorem 2.1]{MP95} and is consistent with the formulation in \cite[Proposition 7]{DA15}.

\begin{lem}\label{cover}
Let $\varphi\in\mathcal W^*(\CC)$ with $\rho\in\mathcal S$. For
every $r>0$, there exist a constant $m=m(r)\in(0,1)$ and a sequence
$\{a_k\}_{k=1}^{\infty}\subset\CC$ such that

\smallskip

\noindent
\textup{(A)} $\CC=\bigcup_{k=1}^{\infty}B^r(a_k)$;

\noindent
\textup{(B)} $B^{mr}(a_k)\cap B^{mr}(a_j)=\varnothing$ whenever $k\neq j$;

\noindent
\textup{(C)} For every fixed $M\geq1$,
$\{B^{Mr}(a_k)\}_{k=1}^{\infty}$ is a covering of $\CC$ of finite
multiplicity.
\end{lem}
 A sequence $\{a_k\}_{k=1}^{\infty}$ satisfying
{\rm(A)}--{\rm(C)} is called an $r$-lattice. More precisely, for every
fixed $M\geq1$, there exists a constant $N=N(r,M)>0$, depending only
on $r$, $M$, the structural constants of $\rho$, and the dimension
$n$, such that
\begin{equation}\label{NNN}
1\leq
\sum_{k=1}^{\infty}
\chi_{B^{Mr}(a_k)}(z)
\leq N,
\qquad z\in\mathbb C^n,
\end{equation}
where $\chi_E$ denotes the characteristic function of
$E\subset\mathbb C^n$.

Using Lemmas~\ref{dzw} and~\ref{rhoB}, we obtain the following integral estimate, which will play a crucial role in establishing the $L^p_\varphi$-norm bounds for the Bergman kernel.
\begin{lem}\label{z-wk2n}
Let $\varphi\in \mathcal W^*(\CC)$ with $\rho\in \mathcal{S}$. Suppose $\varepsilon>0$, $k\ge0$ and $l\in\mathbb{R}$.\\
  \textup{(A)} There exists a constant $C>0$ such that 
  $$
\int_{\CC}|z-w|^k\rho(w)^le^{-\varepsilon d_\rho(z,w)}\,dV(w)\leq C \rho(z)^{k+2n+l},\quad z\in\CC.
$$
 \textup{(B)} 
For every $r\ge1$, there exists a constant $C_{\varepsilon,k,l}(r)>0$ such that 
$$
\int_{\CC\backslash B^r(z)}|z-w|^k\rho(w)^le^{-\varepsilon d_\rho(z,w)}dV(w)
\le C_{\varepsilon,k,l}(r)\,\rho(z)^{k+2n+l},\quad z\in\CC.
$$
Moreover, $ C_{\varepsilon,k,l}(r)\to 0$, as $r\to\infty$, for any fixed  $\varepsilon>0$, $k\ge0$ and $l\in\mathbb{R}$.
\end{lem}
\begin{proof}
By \eqref{rhozw} and Lemma \ref{dzw}, we deduce that there exists a constant $C>0$ such that
 \begin{align}\label{Bz2n}
\int_{B(z)}|z-w|^k\rho(w)^le^{-\varepsilon d_\rho(z,w)}\,dV(w)\leq C \rho(z)^{k+2n+l}
\end{align}
for every $z\in \CC$.

Let $r\ge 1$. By Lemma~\ref{dzw}, there exist constants $C_1>0$ and $C_2(k,l)\in\mathbb R$ such that for every $z\in \CC$,
\begin{align*}
\int_{\CC\backslash B^r(z)}|z-w|^k\rho(w)^le^{-\varepsilon d_\rho(z,w)}\,dV(w)&\leq \int_{\CC\backslash B^r(z)}|z-w|^k\rho(w)^le^{-\varepsilon C_1\left(\frac{|z-w|}{\rho(z)}\right)^\alpha}\,dV(w)\\
& \leq\int_{\CC\backslash B^r(z)}|z-w|^k\rho(w)^l\int_{\varepsilon C_1\left(\frac{|z-w|}{\rho(z)}\right)^\alpha}^{\infty}e^{-s}\,ds\,dV(w)\\
& \leq \int_{C_1 \varepsilon r^\alpha}^{\infty}e^{-s}\int_{B^{\left(\frac{s}{C_1 \varepsilon}\right)^{1/\alpha}}(z)}|z-w|^k\rho(w)^l\,dV(w)\,ds\\
&\lesssim \rho(z)^{k+2n+l}\int_{C_1 \varepsilon r^\alpha}^{\infty} e^{-s}s^{C_2(k,l)}\,ds\\
&= C_{\varepsilon,k,l}(r)\rho(z)^{k+2n+l},
\end{align*}
where
$$ C_{\varepsilon,k,l}(r):=\int_{C_1 \varepsilon r^\alpha}^{\infty} e^{-s}s^{C_2(k,l)}\,ds.$$
Hence {\rm(B)} holds. The conclusion {\rm(A)} is obtained by specializing the second estimate to the case $r=1$ and invoking \eqref{Bz2n}. This completes the proof.
\end{proof}

 Note that for any $\varphi \in \mathcal W^*(\CC)$,  there exist a $\mathcal{C}^{2}$ strictly plurisubharmonic function $\widetilde\varphi$ and a constant $D>0$ such that
\begin{align}\label{varphi}
\|\varphi-\widetilde{\varphi}\|_{L^\infty(\CC)}\leq D,
\end{align}
which means that $F^p_\varphi=F^p_{\widetilde\varphi}$ for any $0<p\leq\infty$. Throughout this paper, for simplicity of notation, we write
\[
\partial_j:=\frac{\partial}{\partial z_j},
\qquad
\bar{\partial}_j:=\frac{\partial}{\partial \bar z_j},
\qquad 1\le j\le n,
\]
and define the complex gradient operator $ \nabla:=(\partial_1,\ldots,\partial_n).$
\begin{lem}\label{var=psi}
Let $\varphi \in \mathcal W^*(\CC)$ with $\rho \in \mathcal S$. Then, for every $w \in \mathbb C^n$ and $r>0$, there exist a $\mathcal{C}^{2}$ function $\psi: B^r(w) \to \mathbb R$ and a constant $C_r>0$, depending only on $r$, such that
$$
i\partial\bar\partial \psi
=
i\partial\bar\partial\widetilde\varphi
\quad \textup{on}\, B^r(w),
\quad \text{and} \quad
\left\||\psi|+\rho|\nabla\psi|\right\|_{L^\infty(B^r(w))}
\leq C_r,
$$
where $\widetilde\varphi$ is a $\rho$-regularization of $\varphi$.
\end{lem}

\begin{proof}
For $z\in \overline{B(0,2)}$, we define
$\tau_w(z):=w+(z_1r\rho(w),\cdots,z_nr\rho(w))$,
$\widetilde\varphi_w(z):=\widetilde\varphi(\tau_w(z))$,
and $\omega_0(z):=|z|^2$ for $z\in\overline{B(0,2)}$.
Since $|\tau_w(z)-w|\leq 2r\rho(w)$, by Lemma \ref{rhoB},
$\rho(\tau_w(z))\simeq\rho(w)$.
Thus
$$
i\partial\bar\partial\widetilde\varphi_w(z)
=
[r\rho(w)]^2\cdot
i\partial\bar\partial\widetilde\varphi(\tau_w(z))
\simeq
i\partial\bar\partial\omega_0,
\quad \textup{on}\,\overline{B(0,2)}.
$$
By \cite[Lemma~3]{D98}, the continuous closed form $i\partial\bar\partial\widetilde\varphi_w$ admits a uniformly bounded potential $\psi_w$ on $B(0,3/2)$. Since $\psi_w-\widetilde\varphi_w$ is distributionally pluriharmonic, Weyl's lemma gives $\psi_w\in\mathcal C^2(B(0,3/2))$. The interior gradient estimate, applied to $\Delta\psi_w=\Delta\widetilde\varphi_w$, therefore yields
$$
i\partial\bar\partial\psi_w
=
i\partial\bar\partial\widetilde\varphi_w
\quad\textup{on}\, B(0,1)\quad\text{and}\quad
\left\|
|\psi_w|+|\nabla\psi_w|
\right\|_{L^\infty(B(0,1))}
\leq C.
$$
For $\xi\in B^r(w)$, we define
$\psi(\xi)=\psi_w(\tau_w^{-1}(\xi))$.
Write $\xi=\tau_w(z)$ for some $z\in B(0,1)$. Then
$$
\begin{aligned}
i\partial\bar\partial\psi(\xi)
&=
[r\rho(w)]^{-2}
\cdot
i\partial\bar\partial
\psi_w(\tau_w^{-1}(\xi))
\\
&=
[r\rho(w)]^{-2}
\cdot
i\partial\bar\partial
\widetilde\varphi_w(\tau_w^{-1}(\xi))
\\
&=
i\partial\bar\partial\widetilde\varphi(\xi).
\end{aligned}
$$
Hence
$$
i\partial\bar\partial\psi
=
i\partial\bar\partial\widetilde\varphi
\quad\textup{on}\,B^r(w).
$$
Meanwhile, by Lemma \ref{rhoB},
$\rho(\tau_w(z))\simeq\rho(w)\simeq\rho(\xi)$.
Thus,
$$
\begin{aligned}
|\psi(\xi)|+\rho(\xi)|\nabla\psi(\xi)|
&=
|\psi_w(z)|
+
\rho(\xi)[r\rho(w)]^{-1}
|\nabla\psi_w(z)|
\leq C_r .
\end{aligned}
$$
The proof is complete.
\end{proof}

The following lemma compares the values of $\varphi$ on a ball with its value at the center.

\begin{lem}\label{fx}
Let $\varphi \in \mathcal W^*(\CC)$ with $\rho \in \mathcal S$. Then, for each $r>0$, there exist a pluriharmonic function $f_x : B^r(x) \to \mathbb R$ and a constant $C_r>0$, depending only on $r$ but not on $x$, such that

\smallskip

\noindent
\textup{(A)} $f_x(x) = 0$;

\noindent
\textup{(B)} 
$|\varphi(y)-\varphi(x)-f_x(y)|
\le C_r$ for all $y\in B^r(x)$;

\noindent
\textup{(C)} 
$|\nabla\widetilde\varphi(y)-\nabla f_x(y)|
\le C_r\rho(x)^{-1}$ for all $y\in B^r(x)$.
\end{lem}

\begin{proof}
Fix $x\in\mathbb C^n$ and define
$$
f_x:=\widetilde\varphi-\widetilde\varphi(x)+\psi(x)-\psi,
$$
where $\psi$ is defined as in Lemma~\ref{var=psi}. 
Then, for any $r>0$, $f_x$ is pluriharmonic on $B^r(x)$ and satisfies
$
f_x(x)=0.
$
Moreover, by \eqref{varphi}, Lemmas~\ref{var=psi} and \ref{rhoB}, we see that there exists a
constant $C_r>0$ such that for all $y\in B^r(x)$,
$$
|\varphi(y)-\varphi(x)-f_x(y)|\lesssim 1+|\widetilde\varphi(y)-\widetilde\varphi(x)-f_x(y)|
=
1+|\psi(y)-\psi(x)|
\le C_r,
$$
and
$$
|\nabla\widetilde\varphi(y)-\nabla f_x(y)|
=
|\nabla\psi(y)|
\le C_r\rho(x)^{-1}.
$$
This completes the proof.
\end{proof}

The following lemma provides pointwise and gradient estimates for holomorphic functions in the weighted space in terms of their local mean values.
\begin{lem}\label{meaninequality}
Let $\varphi\in\mathcal W^*(\CC)$ with $\rho\in\mathcal S$ and $0<p<\infty$.
For every $r>0$ there exists a constant $C_{r,p}>0$ such that the
following statements hold for $h\in H(\mathbb C^n)$.

\smallskip

\noindent
\textup{(A)} For every $x\in\mathbb C^n$
\[
|h(x)|^p e^{-p\varphi(x)}
\le
\frac{C_{r,p}}{|B^r(x)|}
\int_{B^r(x)}|h(y)|^p e^{-p\varphi(y)}\,dV(y).
\]

\noindent
\textup{(B)} For every $x\in\mathbb C^n$ such that $h(x)\ne0$,
\begin{align*}
\left|\nabla\bigl(|h|e^{-\widetilde\varphi}\bigr)(x)\right|
\le
\frac{C_{r,p}}{\rho(x)}
\left(
\frac{1}{|B^r(x)|}
\int_{B^r(x)}|h(y)|^p e^{-p\varphi(y)}\,dV(y)
\right)^{1/p}.
\end{align*}

\noindent
\textup{(C)} For every $x\in\mathbb C^n$ such that
$h(x)=0$,
\begin{equation}\label{zero-derivative-bound}
\operatorname{Lip}
\bigl(|h|e^{-\widetilde\varphi}\bigr)(x)
\le
\frac{C_{r,p}}{\rho(x)}
\left(
\frac{1}{|B^r(x)|}
\int_{B^r(x)}|h(y)|^p e^{-p\varphi(y)}\,dV(y)
\right)^{1/p},
\end{equation}
where \(\operatorname{Lip}G(x)\) denotes the upper pointwise
Lipschitz constant of a locally Lipschitz function \(G\) at \(x\),
defined by
\[
\operatorname{Lip}G(x)
:=
\limsup_{\substack{y\to x\\y\neq x}}
\frac{|G(y)-G(x)|}{|y-x|}.
\]
\end{lem}

\begin{proof}
Fix $x\in\mathbb C^n$.  By Lemma~\ref{fx}, there is a
pluriharmonic function $f_x$ on $B^r(x)$ satisfying
Lemma~\ref{fx}(A)--(C).  Since $B^r(x)$ is simply connected, we may
choose a holomorphic function $F_x$ on $B^r(x)$ such that
\[
\operatorname{Re}F_x=f_x,
\qquad
F_x(x)=0,
\qquad
\partial_jF_x=2\partial_jf_x,
\quad 1\le j\le n.
\]
The function
\[
y\longmapsto
|h(y)e^{-F_x(y)}|^p e^{-p\varphi(x)}
\]
is plurisubharmonic on $B^r(x)$ for every $p>0$.  The mean value
inequality and Lemma~\ref{fx}(B) therefore give
\begin{align*}
|h(x)|^p e^{-p\varphi(x)}
&\le
\frac{1}{|B^r(x)|}
\int_{B^r(x)}
|h(y)|^p e^{-p\varphi(x)-pf_x(y)}\,dV(y)
\\
&\le
\frac{C_r}{|B^r(x)|}
\int_{B^r(x)}
|h(y)|^p e^{-p\varphi(y)}\,dV(y).
\end{align*}
This proves (A).

We next prove (B).  Assume that $h(x)\ne0$ and put
$u_x:=he^{-F_x}$.  Direct differentiation gives
\[
\left|\nabla\bigl(|h|e^{-\widetilde\varphi}\bigr)(x)\right|
=
\frac{e^{-\widetilde\varphi(x)}}{2}
\left(
\sum_{j=1}^n
\left|\partial_jh(x)-2h(x)\partial_j\widetilde\varphi(x)\right|^2
\right)^{1/2}.
\]
The Cauchy estimates,
followed by the submean value inequality for $|u_x|^p$, imply
\begin{align*}
\left(\sum_{j=1}^n|\partial_ju_x(x)|^2\right)^{1/2}
&\le
\frac{C_r}{\rho(x)}
\left(
\frac{1}{|B^r(x)|}
\int_{B^r(x)}|u_x(y)|^p\,dV(y)
\right)^{1/p}
\\
&\le
\frac{C_r e^{\varphi(x)}}{\rho(x)}
\left(
\frac{1}{|B^r(x)|}
\int_{B^r(x)}
|h(y)|^p e^{-p\varphi(y)}\,dV(y)
\right)^{1/p},
\end{align*}
where the second inequality follows from Lemma~\ref{fx}(B).
At the point $x$,
\begin{align*}
\partial_jh-2h\partial_j\widetilde\varphi
&=
\partial_j(he^{-F_x})
+2h\bigl(\partial_jf_x-\partial_j\widetilde\varphi\bigr).
\end{align*}
Hence Lemma~\ref{fx}(C), part (A), and
$\|\varphi-\widetilde\varphi\|_{L^\infty}<\infty$ yield
\begin{align*}
\left|\nabla\bigl(|h|e^{-\widetilde\varphi}\bigr)(x)\right|
&\lesssim
e^{-\widetilde\varphi(x)}
\left(
\left(\sum_{j=1}^n|\partial_ju_x(x)|^2\right)^{1/2}
+\frac{|h(x)|}{\rho(x)}
\right)
\\
&\lesssim
\frac{C_r}{\rho(x)}
\left(
\frac{1}{|B^r(x)|}
\int_{B^r(x)}
|h(y)|^p e^{-p\varphi(y)}\,dV(y)
\right)^{1/p}.
\end{align*}
This proves (B).

We now prove (C). Since \(h\) is holomorphic and
\(\widetilde\varphi\in C^2\), the function
\(|h|e^{-\widetilde\varphi}\) is locally Lipschitz. Moreover, since
\(h(x)=0\),
\[
h(x+v)
=
\sum_{j=1}^n\partial_jh(x)v_j+o(|v|),
\qquad v\to0.
\]
It follows that
\begin{align*}
\operatorname{Lip}
\bigl(|h|e^{-\widetilde\varphi}\bigr)(x)
&=
\limsup_{y\to x}
\frac{|h(y)|e^{-\widetilde\varphi(y)}}{|y-x|}
=
e^{-\widetilde\varphi(x)}
\sup_{|v|=1}
\left|
\sum_{j=1}^n\partial_jh(x)v_j
\right|
=
e^{-\widetilde\varphi(x)}
\left(
\sum_{j=1}^n|\partial_jh(x)|^2
\right)^{1/2}.
\end{align*}
Set \(u_x=he^{-F_x}\) as above. Since \(h(x)=0\) and \(F_x(x)=0\),
\[
\partial_ju_x(x)=\partial_jh(x),
\qquad 1\leq j\leq n.
\]
The preceding Cauchy estimate, which does not require
\(h(x)\neq0\), therefore gives
\[
\left(
\sum_{j=1}^n|\partial_jh(x)|^2
\right)^{1/2}
\leq
\frac{C_r e^{\varphi(x)}}{\rho(x)}
\left(
\frac{1}{|B^r(x)|}
\int_{B^r(x)}
|h(y)|^p e^{-p\varphi(y)}\,dV(y)
\right)^{1/p}.
\]
Multiplying by \(e^{-\widetilde\varphi(x)}\) and using
\(\|\varphi-\widetilde\varphi\|_{L^\infty}<\infty\) proves (C).
\end{proof}

\section{Estimates for the Bergman Kernel and the Bergman Metric}\label{esti}

This section is devoted to pointwise estimates for the Bergman kernel, the equivalence of the Bergman metric, and $L^p$ estimates for the Bergman kernel.

Recall that for $v \in C^2(\mathbb{C}^n)$, the $\bar{\partial}$ operator is given by
$$
\bar{\partial} v = \sum_{j=1}^{n} \frac{\partial v}{\partial \bar{z}_j}\, d\bar{z}_j.
$$
Let $\varphi \in \mathcal W^*(\CC)$ with $\rho \in \mathcal S$. For a $(0,1)$-form $u=\sum_{j=1}^{n}u_jd\overline{z}_j$, we define
$$
|u|^2_{i\partial\overline{\partial}\widetilde\varphi}
:=
\sum_{1\leq j,k\leq n}
a^{jk}u_j\overline{u}_k,
$$
where $(a^{jk})_{n\times n}$ is the inverse of the Hermitian matrix
$$
(a_{jk})_{n\times n}
=
\left(
\frac{\partial^2\widetilde\varphi}
{\partial z_j\partial\overline{z}_k}
\right)_{j,k=1}^n .
$$
Since $i\partial\overline{\partial}\widetilde\varphi \simeq \omega_\rho$, it follows that
$(a_{jk})_{n\times n}\simeq \rho^{-2}I_n,$
where $I_n$ is the $n\times n$ identity matrix. Then
$(a^{jk})_{n\times n} \simeq\rho^2I_n.$
A direct calculation shows that
\begin{equation}\label{ade1}
|u|^2_{i\partial\overline{\partial}\widetilde\varphi}
\simeq
|u|^2\rho^2,
\end{equation}
which will be used throughout the paper.

We recall the following estimate for the $\overline{\partial}$-equation established by Delin \cite{D98}. This result plays a crucial role in deriving the estimates for the Bergman kernel.
\begin{thm}\label{df=theta}
Assume that \( \theta \) is a closed \((0, 1)\)-form on a pseudoconvex domain \( \Omega \subset \mathbb{C}^n \) and that \( \varphi \) is a \( C^2 \) strictly plurisubharmonic function. Assume $\int_\Omega |\theta|_{i\partial\bar\partial\varphi}^2e^{-2\varphi}\,dV<\infty$. Let $\omega$ be a positive $\mathcal C^1$ weight on $\Omega$ satisfying
\[
|\partial \omega|_{i\partial\bar{\partial}\varphi} \leq \varepsilon \omega \quad \text{for some } \varepsilon \in (0, \sqrt{2}).
\]  
Then the \( L^2_{\varphi}(\Omega) \)-minimal solution \( f \) to the \( \bar{\partial} \)-equation  
\[
\bar{\partial}f = \theta
\]  
satisfies  
\[
\int_{\Omega} |f|^2 e^{-2\varphi} \omega \, dV \leq \frac{2}{( \sqrt{2}-\varepsilon )^2} \int_{\Omega} |\theta|_{i\partial\bar{\partial}\varphi}^2 e^{-2\varphi} \omega \, dV.
\]  
Here, \( |\theta|_{i\partial\bar{\partial}\varphi} \) denotes the norm of \( \theta \) in the Kähler metric induced by the potential function \( \varphi \).
\end{thm}

\begin{prop}[Global upper estimate]\label{Gue}
Let $\varphi \in \mathcal W^*(\CC)$ with $\rho \in \mathcal S$. Then there exist constants $C>0$ and $\varepsilon>0$ such that
\begin{equation}
    |K_z^\varphi(w)|
    \le
    C\,
    \frac{e^{\varphi(z)+\varphi(w)}}
         {\rho(z)^n\rho(w)^n}
    e^{-\varepsilon d_\rho(z,w)},
    \qquad z,w\in\CC.
    \label{eq:general-upper}
\end{equation}
\end{prop}

\begin{proof}
We divide the proof into two steps.

\medskip
\noindent
\textbf{Step 1: the smooth weight $\widetilde{\varphi}$.}
 First, we prove an upper bound near the diagonal.  Applying Lemma \ref{meaninequality} to $K_z^{\widetilde{\varphi}}$ for $z\in\CC$, we obtain
\begin{align*}
    |K_z^{\widetilde{\varphi}}(z)|^2e^{-2\widetilde{\varphi}(z)}&\leq C\frac{1}{|B(z)|}\int_{B(z)}|K_z^{\widetilde{\varphi}}(y)|^2e^{-2\widetilde{\varphi}(y)}
    \,dV(y)\\
    &\leq C\frac{1}{\rho(z)^{2n}}\int_{\CC}|K_z^{\widetilde{\varphi}}(y)|^2
    e^{-2\widetilde{\varphi}(y)}\,dV(y)=C\frac{|K_z^{\widetilde{\varphi}}(z)|}
    {\rho(z)^{2n}},
\end{align*}
which implies that 
\begin{equation}\label{ade10}
	|K_z^{\widetilde{\varphi}}(z)|\leq C\frac{e^{2\widetilde{\varphi}(z)}}
	{\rho(z)^{2n}}.
\end{equation}
This, together with Cauchy--Schwarz inequality further implies that
$$
|K_z^{\widetilde{\varphi}}(w)| \leq \sqrt{K_z^{\widetilde{\varphi}}(z)K_w^{\widetilde{\varphi}}(w)} \leq C \frac{e^{\widetilde{\varphi}(z)+\widetilde{\varphi}(w)}}{\rho(z)^n\rho(w)^n}. 
$$
In particular, for any fixed $R$, we have
$$
|K_z^{\widetilde{\varphi}}(w)| \lesssim \frac{e^{\widetilde{\varphi}(z)+\widetilde{\varphi}(w)}}{\rho(z)^n\rho(w)^n} e^{-d_{\rho}(z,w)},\quad d_{\rho}(z,w) \leq R. 
$$

We now prove the off-diagonal upper bound when  $d_{\rho}(z,w) > R$. By \eqref{drhoB}, there exists $r>0$ such that
\begin{equation}\label{ade6}
	B^r(z) \cap B^r(w) = \emptyset.
\end{equation}
Fix a smooth cut-off function $\eta$ on $\mathbb{C}^n$ such that
\begin{align}\label{eta}
0 \leq \eta \leq 1, \quad
\eta \equiv 1 \ \text{on } B(0,1/2), \quad
\operatorname{supp}\eta \subset B(0,1), \quad
|\bar{\partial}\eta|^2 \leq C\eta .
\end{align}
 Set $\chi_w(\cdot) = \eta\left( \frac{\cdot - w}{r \rho(w)} \right)$. Then
$$
0 \leq \chi_w \leq 1, \quad \text{supp}\chi_w \subset B^r(w), \quad \chi_w|_{B^{r/2}(w)} \equiv 1, \quad |\bar{\partial}\chi_w|_{i\partial\bar{\partial}\widetilde{\varphi}} \lesssim 1.
$$
Applying Lemma \ref{meaninequality}, we obtain
\begin{equation}\label{ade7}
\begin{aligned}
|K_z^{\widetilde{\varphi}}(w)|^2 e^{-2\widetilde{\varphi}(w)} &\leq C \frac{1}{\rho(w)^{2n}} \int_{B^{r/2}(w)} |K_z^{\widetilde{\varphi}}(\xi)|^2 e^{-2\widetilde{\varphi}(\xi)}\, dV(\xi) \\
&\leq C \frac{1}{\rho(w)^{2n}} \int_{\CC} \chi_w(\xi) K_z^{\widetilde{\varphi}}(\xi) \overline{K_z^{\widetilde{\varphi}}(\xi)} e^{-2\widetilde{\varphi}(\xi)}\, dV(\xi) \\
&= C \frac{1}{\rho(w)^{2n}} P_{\widetilde{\varphi}}(\chi_w K_z^{\widetilde{\varphi}})(z),
\end{aligned}
\end{equation}
where $P_{\widetilde{\varphi}}$ is the orthogonal projection from $L_{\widetilde{\varphi}}^2$ onto $F_{\widetilde{\varphi}}^2$. Let
$$
v = (I - P_{\widetilde{\varphi}})(\chi_w K_z^{\widetilde{\varphi}}),
$$
then the function $v$ is the
$L^2_{\widetilde{\varphi}}$-minimal solution of this equation $\overline{\partial}v = K_z^{\widetilde{\varphi}} \overline{\partial}\chi_w$.
By \eqref{ade6}, we know that $\chi_w|_{B^r(z)} \equiv 0$. Thus, $|v(\xi)|=|P_{\widetilde\varphi}(\chi_w K_z^{\widetilde{\varphi}})(\xi)|$ for $\xi\in B^r(z)$. Since $P_{\widetilde\varphi}(\chi_w K_z^{\widetilde{\varphi}})\in H(\CC)$, we apply Lemma \ref{meaninequality} to get
\begin{align}\label{Pv}
|P_{\widetilde\varphi}(\chi_w K_z^{\widetilde{\varphi}})(z)|^2 e^{-2\widetilde{\varphi}(z)} 
&\lesssim \frac{1}{\rho(z)^{2n}} \int_{B^r(z)} |v(\xi)|^2 e^{-2\widetilde{\varphi}(\xi)} dV(\xi).
\end{align}
Choose a smooth Riemannian metric $g_0$ uniformly equivalent to
$i\partial\bar\partial\widetilde\varphi$, and denote its distance by
$d_0$. Lemma~\ref{dzw} gives
$d_0\simeq d_\varphi\simeq d_\rho$ and shows that $g_0$ is complete.
By the Greene--Wu approximation theorem \cite{GW79}, for each $z\in\CC$ there exists a smooth real-valued function $g_z$ such that
\[
|g_z(\xi)-d_0(\xi,z)|\leq1,
\qquad
|g_z(\xi)-g_z(w)|\leq2d_0(\xi,w).
\]
Consequently, there exist constants $c,C>0$ such that
\begin{equation}\label{ade5}
c d_\rho(\xi,z)-1\leq g_z(\xi)
\leq C d_\rho(\xi,z)+1,
\qquad
|g_z(\xi)-g_z(w)|\leq C d_\rho(\xi,w).
\end{equation}
Let $\omega=e^{-\varepsilon g_z}$. By \eqref{ade5}, $\omega(\xi)$ is bounded from below whenever $\xi\in B^r(z)$. From this, \eqref{ade7}, \eqref{Pv} and \eqref{varphi}, we deduce that
\begin{equation}\label{ade8}
	\begin{aligned}
		|K_z^{\widetilde{\varphi}}(w) e^{-\widetilde{\varphi}(w)}|^2 &\lesssim \frac{e^{\widetilde{\varphi}(z)}}{\rho^{2n}(w)\rho^{n}(z)} \left(\int_{B^r(z)} \omega(\xi)|v(\xi)|^2 e^{-2\widetilde{\varphi}(\xi)} dV(\xi)\right)^{1/2}\\
		&\lesssim \frac{e^{\widetilde{\varphi}(z)}}{\rho^{2n}(w)\rho^{n}(z)} \left(\int_{\CC} |v(\xi)|^2 e^{-2\widetilde{\varphi}(\xi)}\omega(\xi) \, dV(\xi)\right)^{1/2}.
	\end{aligned}
\end{equation}
Let $\Delta x=(\Delta x_1,0,\ldots,0)$. Combining \eqref{ade5} with Lemma~\ref{dzw}, the partial derivative \(\frac{\partial g_z}{\partial x_1}\) satisfies
\begin{align*}
\left|\frac{\partial g_z}{\partial x_1}(\xi)\right|& = \lim_{\Delta x_1 \to 0} \frac{|g_z(\xi + \Delta x) - g_z(\xi)| d_\rho(\xi, \xi + \Delta x)}{d_\rho(\xi, \xi + \Delta x) \cdot |\Delta x|}
\leq C \lim_{\Delta x_1 \to 0} \frac{d_\rho(\xi, \xi + \Delta x)}{|\Delta x|} \lesssim \frac{1}{\rho(\xi)}.
\end{align*}
 Similarly, for $1\leq i\leq 2n$ we have
$$\left|\frac{\partial g_z}{\partial x_i}(\xi)\right|\lesssim \frac{1}{\rho(\xi)},$$
which implies that
$$|\partial g_z|_{i\partial\bar{\partial}\widetilde{\varphi}}^2\simeq |\partial g_z|^2\rho^2\lesssim 1.$$
Moreover, we have $\partial \omega=-\varepsilon e^{-\varepsilon g_z}\partial g_z$. Then, we choose $\varepsilon>0$ so that 
$$|\partial \omega|_{i\partial\bar{\partial}\widetilde{\varphi}}=\varepsilon \omega|\partial g_z|_{i\partial\bar{\partial}\widetilde{\varphi}}\leq \varepsilon_1\omega$$
for some $\varepsilon_1\in (0, \sqrt{2})$.  From this, we may invoke Theorem \ref{df=theta} to obtain that
\begin{align}\label{KKKK}
\int_{\CC} |v(\xi)|^2 e^{-2\widetilde{\varphi}(\xi)}\omega(\xi) \, dV(\xi)\lesssim  \int_{\CC} |K_z^{\widetilde{\varphi}}(\xi)|^2|\overline{\partial}\chi_w(\xi)|_{i\partial\bar{\partial}\widetilde{\varphi}}^2 e^{-2\widetilde{\varphi}(\xi)}\omega(\xi) \, dV(\xi).
\end{align}
We claim that
\begin{equation}\label{ade9}
\omega(\xi) \lesssim e^{- \varepsilon d_\rho(z, w)},\quad \xi \in B^r(w).
\end{equation}
In fact, by the triangle inequality and \eqref{drhoB}, for  
\(\zeta \in B^r(w)\) we have 
$$d_\rho(z, \zeta) \geq d_\rho(z, w) - d_\rho(w, \zeta)
 \geq d_\rho(z, w) - c_r.$$ 
This and \eqref{ade5} imply \eqref{ade9}.
Plugging \eqref{ade9} into \eqref{KKKK}, then combining \eqref{KKKK} and \eqref{ade8}, we deduce that
\begin{align}\label{Kzwrho}
|K_z^{\widetilde{\varphi}}(w) e^{-\widetilde{\varphi}(w)}|^2 \leq C \frac{e^{\widetilde{\varphi}(z)}e^{- \varepsilon d_\rho(z, w)/2}}{\rho^{2n}(w)\rho^{n}(z)}\left(\int_{\CC} |K_z^{\widetilde{\varphi}}(\xi)|^2 e^{-2\widetilde{\varphi}(\xi)} \, dV(\xi)\right)^{1/2}= C \frac{e^{\widetilde{\varphi}(z)}e^{- \varepsilon d_\rho(z, w)/2}}{\rho^{2n}(w)\rho^{n}(z)} \sqrt{K_z^{\widetilde{\varphi}}(z)}.
\end{align}
This together with \eqref{ade10} gives us the off-diagonal upper bound when  $d_{\rho}(z,w) > R$.

\medskip
\noindent
\textbf{Step 2: passage from $\widetilde{\varphi}$ to $\varphi$.}
Set $b:=e^{-2(\varphi-\widetilde{\varphi})}.$
By \eqref{varphi}, we have $c:=e^{-2D}\le b\le e^{2D}=:C$ and 
$F^2_\varphi=F^2_{\widetilde{\varphi}}$.
On $F^2_{\widetilde{\varphi}}$, define $T_b:=P_{\widetilde{\varphi}}M_b$.
For $f,g\in F^2_{\widetilde{\varphi}}$,
\begin{align*}
    \langle T_bf,g\rangle_{\widetilde{\varphi}}
    &=\int_{\CC}b f\overline g
       e^{-2\widetilde{\varphi}}\,dV
      =\langle f,T_bg\rangle_{\widetilde{\varphi}},
\end{align*}
so $T_b$ is self-adjoint.  Moreover,
\begin{equation}
    \langle T_bf,f\rangle_{\widetilde{\varphi}}
    =\|f\|_\varphi^2
    \ge c\|f\|_{\widetilde{\varphi}}^2.
    \label{eq:Tb-coercive}
\end{equation}
It follows that
\[
    \|T_bf\|_{\widetilde{\varphi}}
    \ge c\|f\|_{\widetilde{\varphi}}.
\]
Hence $T_b$ is injective and has closed range.  Since it is
self-adjoint,
\[
    \overline{\operatorname{Ran} T_b}
    =(\ker T_b^*)^\perp
    =(\ker T_b)^\perp
    =F^2_{\widetilde{\varphi}}.
\]
Thus $T_b$ is onto and $\|T_b^{-1}\|\le c^{-1}$. For every $f\in F^2_{\widetilde{\varphi}}=F^2_\varphi$, the reproducing
properties give
\begin{align*}
    f(z)
    =\langle f,K_z^\varphi\rangle_\varphi
      =\langle bf,K_z^\varphi\rangle_{\widetilde{\varphi}}
    =\langle T_bf,K_z^\varphi\rangle_{\widetilde{\varphi}}
      =\langle f,T_bK_z^\varphi\rangle_{\widetilde{\varphi}},
\end{align*}
whereas
$f(z)=\langle f,K_z^{\widetilde{\varphi}}\rangle_{\widetilde{\varphi}}.$
The uniqueness of the Riesz representing vector therefore implies
\begin{equation}
    T_bK_z^\varphi=K_z^{\widetilde{\varphi}},
    \qquad
    K_z^\varphi=T_b^{-1}K_z^{\widetilde{\varphi}}.
    \label{eq:kernel-identity}
\end{equation}

We now prove the localization needed to use
\eqref{eq:kernel-identity}.  Define
\[
    \,d\nu(x):=\rho(x)^{-2n}\,dV(x)
\]
and the unitary map
\[
    U:L^2_{\widetilde{\varphi}}\longrightarrow L^2(\nu),
    \qquad
    Uf(x)=\rho(x)^ne^{-\widetilde{\varphi}(x)}f(x).
\]
Let $\Pi:=UP_{\widetilde{\varphi}}U^{-1}.$
The integral kernel of $\Pi$ with respect to $\nu$ is
\begin{equation}
    p(x,y):=
    \rho(x)^n\rho(y)^n
    e^{-\widetilde{\varphi}(x)-\widetilde{\varphi}(y)}
    K_y^{\widetilde{\varphi}}(x).
    \label{eq:normalized-kernel}
\end{equation}
By Step 1, there exists $C_0,a>0$ such that
\begin{equation}
    |p(x,y)|\le C_0e^{-a d_\rho(x,y)}.
    \label{eq:p-decay}
\end{equation}
Fix $0<\alpha<a$.  Lemma \ref{z-wk2n}(A) and \eqref{eq:p-decay} give
\begin{equation}
\begin{split}
    S_\alpha(p):=
    \sup_x\int_{\CC}|p(x,y)|e^{\alpha d_\rho(x,y)}\,d\nu(y)
    +\sup_y\int_{\CC}|p(x,y)|e^{\alpha d_\rho(x,y)}\,d\nu(x)
    <\infty.
\end{split}
    \label{eq:p-schur}
\end{equation}
Set $\mathcal H:=U(F_{\widetilde\varphi}^2)\subset L^2(\nu),$
and let $\Pi$ be the orthogonal projection from $L^2(\nu)$ onto
$\mathcal H$.  The composition $\Pi M_b\Pi$ is represented by
\[
L^2(\nu)
\xrightarrow{\ \Pi\ }
\mathcal H
\xrightarrow{\ M_b\ }
L^2(\nu)
\xrightarrow{\ \Pi\ }
\mathcal H.
\]
Its restriction to $\mathcal H$ is
\[
\left.\Pi M_b\Pi\right|_{\mathcal H}
=UT_bU^{-1}.
\]
We extend this operator to all of $L^2(\nu)$ by letting it act as the
identity on $\mathcal H^\perp$. More precisely, define
\begin{equation}
A:=(I-\Pi)+\Pi M_b\Pi
  =I+\Pi M_{b-1}\Pi
\label{eq:def-A}
\end{equation}
on $L^2(\nu)$. Then
\[
A|_{\mathcal H}=UT_bU^{-1},
\qquad
A|_{\mathcal H^\perp}=I.
\]
 Indeed, write
\[
h=h_0+h_1,
\qquad
h_0\in\mathcal H,\quad h_1\in\mathcal H^\perp.
\]
Since $\Pi h_0=h_0$ and $\Pi h_1=0$, we have
\[
Ah
=
\Pi M_bh_0+h_1.
\]
Moreover, since $\Pi=UP_{\widetilde\varphi}U^{-1}$ and multiplication
by $b$ commutes with $U$,
\[
\Pi M_bh_0
=
UT_bU^{-1}h_0.
\]
Consequently, with respect to the orthogonal decomposition
$L^2(\nu)=\mathcal H\oplus\mathcal H^\perp$, the operator $A$ has the
block diagonal form
\[
A=
\begin{pmatrix}
UT_bU^{-1} & 0\\
0 & I
\end{pmatrix}.
\]
It follows that
\begin{equation}
    A\ge m_0I,
    \qquad
    m_0=\min\{1,c\}>0,
    \qquad
    \|A^{-1}\|_{L^2(\nu)\rightarrow L^2(\nu)}\le m_0^{-1}.
    \label{eq:A-coercive}
\end{equation}
Set $Q:=A-I=\Pi M_{b-1}\Pi$.  Its kernel is
\begin{equation}
    q(x,y)=\int_{\CC}
       p(x,u)(b(u)-1)p(u,y)\,d\nu(u).
    \label{eq:q-kernel}
\end{equation}
Using the triangle inequality $e^{\alpha d_\rho(x,y)}\le e^{\alpha d_\rho(x,u)}e^{\alpha d_\rho(u,y)}$
and Fubini's theorem, \eqref{eq:p-schur} gives
\begin{equation}
\begin{split}
    S_\alpha(q):=
    \sup_x\int_{\CC}|q(x,y)|e^{\alpha d_\rho(x,y)}\,d\nu(y)
    +\sup_y\int_{\CC}|q(x,y)|e^{\alpha d_\rho(x,y)}\,d\nu(x)
    <\infty.
\end{split}
    \label{eq:q-schur}
\end{equation}
Fix $z\in\CC$ and, for $R>0$, set $h_{z,R}(x):=\min\{d_\rho(z,x),R\}.$
The function $h_{z,R}$ is bounded and $1$-Lipschitz with respect to
$d_\rho$.  For $\delta>0$, define the bounded invertible multiplication
operator
\[
    W_{\delta,z,R}F(x):=e^{\delta h_{z,R}(x)}F(x).
\]
Set
\[
E_{\delta,z,R}
:=
W_{\delta,z,R}AW_{\delta,z,R}^{-1}-A.
\]
Then the kernel of $E_{\delta,z,R}$ is $\left( e^{\delta(h_{z,R}(x)-h_{z,R}(y))}-1\right)q(x,y).$
Since $|h_{z,R}(x)-h_{z,R}(y)|\le d_\rho(x,y)$
and $|e^t-1|\le |t|e^{|t|}$, the kernel satisfies
\[
 |\left(
      e^{\delta(h_{z,R}(x)-h_{z,R}(y))}-1
    \right)q(x,y)|\leq   \delta d_\rho(x,y)e^{\delta d_\rho(x,y)}|q(x,y)|.
\]
If $0<\delta\le\alpha/2$, then
\[
    se^{\delta s}
    \le \frac{2}{e\alpha}e^{\alpha s},
    \qquad s\ge0.
\]
By \eqref{eq:q-schur} and Schur's test,
\begin{align}\label{EzR}
\|E_{\delta,z,R}\|_{L^2(\nu)\to L^2(\nu)}
\leq C_1\delta,
\end{align}
where $C_1$ is independent of $z$ and $R$. Choose
\begin{align}\label{323}
0<\delta<
\min\left\{
\frac{\alpha}{2},
\frac{m_0}{2C_1}
\right\}.
\end{align}
It follows from \eqref{eq:A-coercive} and \eqref{EzR} that
\[
\|A^{-1}E_{\delta,z,R}\|_{L^2(\nu)\to L^2(\nu)}
\leq
\frac{C_1\delta}{m_0}
<\frac12.
\]
Since $W_{\delta,z,R}AW_{\delta,z,R}^{-1}=A\bigl(I+A^{-1}E_{\delta,z,R}\bigr),$
the Neumann series converges in $\mathcal B(L^2(\nu))$ and gives
\[
\begin{aligned}
W_{\delta,z,R}A^{-1}W_{\delta,z,R}^{-1}
&=
\bigl(
W_{\delta,z,R}AW_{\delta,z,R}^{-1}
\bigr)^{-1} 
=
\sum_{k=0}^{\infty}
\bigl(-A^{-1}E_{\delta,z,R}\bigr)^kA^{-1}.
\end{aligned}
\]
Consequently,
\begin{align}\label{WAW}
\left\|
W_{\delta,z,R}A^{-1}W_{\delta,z,R}^{-1}
\right\|_{L^2(\nu)\to L^2(\nu)}
&\leq
\frac{
\|A^{-1}\|_{L^2(\nu)\to L^2(\nu)}
}{
1-
\|A^{-1}E_{\delta,z,R}\|_{L^2(\nu)\to L^2(\nu)}
} 
\leq
\frac{2}{m_0},
\end{align}
uniformly in $z$ and $R$.

Let $F_z:=UK_z^{\widetilde{\varphi}}$ and $G_z:=UK_z^\varphi.$ By \eqref{eq:kernel-identity} and \eqref{eq:def-A}, $G_z=A^{-1}F_z.$ Moreover, \eqref{eq:normalized-kernel} gives
\[F_z(x)=\frac{e^{\widetilde{\varphi}(z)}}{\rho(z)^n}p(x,z).\]
Therefore, after decreasing $\delta$ if necessary so that $\delta<a$,
Lemma \ref{z-wk2n}(A) and \eqref{eq:p-decay} give
\begin{equation}
    \|W_{\delta,z,R}F_z\|_{L^2(\nu)}
    \lesssim
    \frac{e^{\widetilde{\varphi}(z)}}{\rho(z)^n},
    \label{eq:weighted-F}
\end{equation}
uniformly in $R$.  Combining \eqref{WAW} and \eqref{eq:weighted-F}, we obtain
\[
    \|W_{\delta,z,R}G_z\|_{L^2(\nu)}\leq \left\|
W_{\delta,z,R}A^{-1}W_{\delta,z,R}^{-1}
\right\|_{L^2(\nu)\to L^2(\nu)} \|W_{\delta,z,R}F_z\|_{L^2(\nu)}
    \lesssim
    \frac{e^{\widetilde{\varphi}(z)}}{\rho(z)^n}.
\]
Letting $R\to\infty$ and applying the monotone convergence theorem
gives
\begin{equation}
    \int_{\CC}|K_z^\varphi(x)|^2
       e^{-2\widetilde{\varphi}(x)}
       e^{2\delta d_\rho(z,x)}\,dV(x)
    \lesssim
    \frac{e^{2\widetilde{\varphi}(z)}}{\rho(z)^{2n}}.
    \label{eq:weighted-Kphi}
\end{equation}
Finally, applying Lemma \ref{meaninequality} to $K_z^\varphi$
gives
\begin{align*}
    |K_z^\varphi(w)|^2e^{-2\widetilde{\varphi}(w)}
    &\lesssim
    \frac{1}{\rho(w)^{2n}}
    \int_{B^r(w)}|K_z^\varphi(x)|^2
       e^{-2\widetilde{\varphi}(x)}\,dV(x).
\end{align*}
For $x\in B^r(w)$, we have $d_\rho(z,x)\ge d_\rho(z,w)-L_r.$
It follows from \eqref{eq:weighted-Kphi} that
\[
    |K_z^\varphi(w)|^2e^{-2\widetilde{\varphi}(w)}
    \lesssim
    \frac{e^{2\widetilde{\varphi}(z)}}
         {\rho(z)^{2n}\rho(w)^{2n}}
    e^{-2\delta d_\rho(z,w)}.
\]
This proves \eqref{eq:general-upper}.
\end{proof}
 
To derive the lower estimate, we shall need the following theorem on one-point interpolation with uniform $L^2$ estimates.
\begin{lem}\label{peak}
Let $\varphi \in \mathcal W^*(\CC)$ with $\rho \in \mathcal S$. Then there exists a constant $C>0$ such that for every $x \in \mathbb{C}^n$, one can find a holomorphic function $f \in H(\mathbb{C}^n)$ satisfying
$$
f(x)=\frac{e^{\varphi(x)}}{\rho(x)^n}, \qquad \|f\|_{F^2_\varphi} \leq C.
$$
\end{lem}
\begin{proof}
  Fix $x \in \mathbb{C}^n$. Let $F_x$ be a holomorphic function on $B(x)$ such that $F_x(x)=0$ and $\operatorname{Re} F_x = f_x$, as in Lemma \ref{fx}.
On $B(x)$, define the $(0,1)$ form
\[
\theta(w)
:= \frac{e^{F_x(w)+\varphi(x)}}{\rho(x)^n}\,\bar{\partial}\chi_x(w),
\]
where
$
\chi_x(w)
=
\eta\!\left(\frac{w-x}{\rho(x)}\right)
$
with $\eta$ as defined in \eqref{eta}. Since $F_x$ is holomorphic, the $(0,1)$-form $\theta$ is $\bar\partial$-closed. Because $\operatorname{supp}\chi_x$ is compactly contained in
$B(x)$, we extend $\theta$ by zero to $\mathbb C^n$. This extension
is still $\bar\partial$ closed.
Then $\theta$ is supported in $B(x)\backslash B^{1/2}(x)$, and$
|\bar{\partial}\chi_x| \leq C/\rho(x).$ By \eqref{ade1}, we have
\begin{equation}\label{ade3} |\theta|_{i\partial\bar{\partial}\widetilde\varphi}^2=\frac{e^{2f_x+2\varphi(x)}}{\rho(x)^{2n}}
	|\bar{\partial}\chi_x|^2_{i\partial\bar{\partial}\widetilde\varphi}\simeq
	\frac{e^{2f_x+2\varphi(x)}}{\rho(x)^{2n}}|\bar{\partial}\chi_x|^2\rho^2.
\end{equation}
For every fixed $ x\in\CC$ and $\delta\in(0,1]$, define 
$$\psi_{x,\delta}(w):=\widetilde\varphi(w)+\frac{n}{2}\ln \left(\frac{|w-x|^{2}}{\rho(x)^{2}}+\delta\right),\quad w\in\CC.$$
Since the function
\[
w\longmapsto
\log\left(
\frac{|w-x|^2}{\rho(x)^2}+\delta
\right)
\]
is smooth and plurisubharmonic, we have
\begin{equation}\label{ade2}
	i\partial\bar\partial\psi_{x,\delta}\geq i\partial\bar\partial\widetilde\varphi\simeq \omega_\rho.
\end{equation}
In particular, $\psi_{x,\delta}$ is a \(C^2\) strictly
plurisubharmonic function. For $\varepsilon>0$, we let
$\omega_x:=e^{\varepsilon g_x}$, where $g_x$ is as in
\eqref{ade5}. Moreover, by \eqref{ade5},
\begin{align}\label{gx}
c d_\rho(\xi,x)-1
\leq g_x(\xi)
\leq C d_\rho(\xi,x)+1,
\qquad x,\xi\in\CC.
\end{align}
The estimate following \eqref{ade5} also gives $|\partial g_x|_{i\partial\bar\partial\widetilde\varphi}
\lesssim1.$
Hence we can choose $\varepsilon>0$ so that
\[
|\partial\omega_x|_{i\partial\bar\partial\widetilde\varphi}
=
\varepsilon\omega_x
|\partial g_x|_{i\partial\bar\partial\widetilde\varphi}
\leq\varepsilon_1\omega_x
\]
for some $\varepsilon_1\in(0,\sqrt{2})$.  From this and Theorem~\ref{df=theta}, for every $0<\delta\leq1$ we obtain a function $u_\delta$ such that $\bar\partial u_\delta=\theta$ and
\begin{align}\label{ade4}
\int_{\CC}|u_\delta(w)|^2e^{-2\psi_{x,\delta}(w)}\omega_x(w)
\,dV(w)\leq C\int_{\CC}|\theta(w)|_{i\partial \overline{\partial}\psi_{x,\delta}}^2e^{-2\psi_{x,\delta}(w)}\omega_x(w)
\,dV(w),
\end{align}
where the constant $C>0$ is independent of $x$ and $\delta$. Since $\theta$ is supported
on $B(x)\backslash B^{1/2}(x)$, equations \eqref{ade3}, \eqref{ade2}, \eqref{gx}, and \eqref{varphi}, together with Lemma~\ref{fx}, give
\begin{align*}
\int_{\CC}|\theta(w)|_{i\partial \overline{\partial}\psi_{x,\delta}}^2e^{-2\psi_{x,\delta}(w)}\omega_x(w)
\,dV(w) &\lesssim\int_{\CC}\frac{|\theta(w)|_{i\partial \overline{\partial}\widetilde\varphi}^2e^{-2\varphi(w)+\varepsilon g_x(w)}}
 {\left(\frac{|w-x|^{2}}{\rho(x)^{2}}+\delta\right)^n}\,dV(w)\\
 &\lesssim \int_{B(x)\backslash B^{1/2}(x)} e^{-2(\varphi(w)-\varphi(x)-f_x(w))}\rho(w)^{-2n}\,dV(w)\lesssim 1,
\end{align*}
where the implicit constants are independent of $x$ and $\delta$. For $0<\delta\leq 1$, set
\[
W_\delta(w):=e^{-2\psi_{x,\delta}(w)}\omega_x(w).
\]
Thus, by \eqref{ade4}, we have
\[
\int_{\mathbb C^n}|u_\delta(w)|^2W_\delta(w)\,dV(w)\leq C,
\]
where $C$ is independent of $x$ and $\delta$. 

For $0<\delta\leq1$, we have $W_\delta\geq W_1$. Since  $W_1$ has a positive lower bound on every compact subset of $\mathbb C^n$, the family $\{u_\delta: 0<\delta\leq 1\}$ is bounded
in $L^2_{\mathrm{loc}}(\mathbb C^n)$. Choose a sequence $\delta_j\downarrow0$. By applying weak compactness on an exhaustion of $\CC$ by compact
balls and then using a diagonal argument, we may assume that, for
every $m\in\mathbb N$,  \[
u_{\delta_j}\rightharpoonup u
\quad\text{weakly in }L^2(\overline{B(0,m)}).
\]
Passing to the limit against compactly supported test forms gives
$\bar\partial u=\theta$ in the sense of distributions.

Fix $\tau\in(0,1)$ and put $K_m=\overline{B(0,m)}$. For all
sufficiently large $j$, $\delta_j\leq\tau$, and hence
$W_\tau\leq W_{\delta_j}$. Since $W_\tau$ is bounded above and below
by positive constants on $K_m$, weak lower semicontinuity yields
\begin{align*}
\int_{K_m}|u(w)|^2W_\tau(w)\,dV(w)
&\leq
\liminf_{j\to\infty}
\int_{K_m}|u_{\delta_j}(w)|^2W_\tau(w)\,dV(w)\\
&\leq
\liminf_{j\to\infty}
\int_{\mathbb C^n}
|u_{\delta_j}(w)|^2W_{\delta_j}(w)\,dV(w)\leq C.
\end{align*}
First letting $m\to\infty$ and then letting $\tau\downarrow0$,
the monotone convergence theorem gives
\begin{align}\label{u=0}
\int_{\mathbb C^n}
|u(w)|^2e^{-2\widetilde\varphi(w)}
\frac{\rho(x)^{2n}}{|w-x|^{2n}}
\omega_x(w)\,dV(w)
\leq C.
\end{align}

By \eqref{gx}, for $w\in \CC$ we have 
$$\omega_x(w)\geq e^{\varepsilon(c d_\rho(w,x)-1)}\geq e^{-\varepsilon}.$$ 
Combining this with \eqref{u=0} and \eqref{varphi},  we obtain
\begin{align}\label{u=00}
\int_{B^{1/2}(x)}
\frac{|u(w)|^2}{|w-x|^{2n}}
e^{-2\widetilde\varphi(w)} \rho(x)^{2n}\, dV(w)\lesssim 1.
\end{align}

Since $\theta=0$ on $B^{1/2}(x)$, we have
$\bar\partial u=0$ there in the sense of distributions. Since $u\in L^2_{\mathrm{loc}}(\CC)$, it defines a distribution,
and hence a $(0,0)$-current, on $B^{1/2}(x)$. By the
Dolbeault--Grothendieck lemma for currents
\cite[Chapter I, Section 3.E, Theorem 3.29(a)]{DJ12},  $u$ agrees almost
everywhere on $B^{1/2}(x)$ with a holomorphic function, which we
continue to denote by $u$. If $u(x)\neq0$, continuity gives
$|u(w)|\geq c>0$ near $x$, whereas $|w-x|^{-2n}$ is not locally
integrable in $\mathbb R^{2n}$; this contradicts \eqref{u=00}. Hence
$u(x)=0$.

Let $a_x$ denote the extension by zero to $\mathbb C^n$ of the
function $e^{F_x+\varphi(x)}\chi_x/\rho(x)^n$ on $B(x)$, and set
$$g:=a_x-u.$$
Then $\bar{\partial}g=0$ in the sense of distributions on $\CC$. Another application of the Dolbeault--Grothendieck lemma shows that
$g$ agrees almost everywhere with an entire function, denoted by
$f$.

 It remains to estimate its norm. By Lemma \ref{dzw} and \eqref{gx}, there exist $c_1,\alpha,C >0$ such that, with $t=|w-x|/\rho(x)$,
$$\omega_x(w)^{-1}\frac{|w-x|^{2n}}{\rho(x)^{2n}}
\lesssim t^{2n}e^{-c_1\min\{t,t^\alpha\}}\lesssim 1$$
for every $x,w\in\CC$.
Thus in view of \eqref{u=0}, we have
\[
\begin{aligned}
\|u\|_{L_\varphi^2}^2
&\lesssim
\int_{\CC}
|u(w)|^2e^{-2\widetilde\varphi(w)}
\frac{\rho(x)^{2n}}{|w-x|^{2n}}\omega_x(w)  \\
&\hspace{35mm}\times
\omega_x(w)^{-1}
\frac{|w-x|^{2n}}{\rho(x)^{2n}}
\,dV(w)
\lesssim1.
\end{aligned}
\]
Together with \eqref{rhozw}, this implies that
\begin{align*}
  \|f\|_{L^2_\varphi}
  & \lesssim \left(\int_{\CC}\left|\frac{e^{F_x(w)+\varphi(x)}}{\rho(x)^{n}}\chi_x(w)\right|^2
  e^{-2\varphi(w)}\,dV(w)\right)^{1/2}+\|u\|_{L^2_\varphi}\lesssim 1.
\end{align*} 
Finally, $\chi_x\equiv1$ on $B^{1/2}(x)$, and hence
\[
f(w)
=
\frac{e^{F_x(w)+\varphi(x)}}{\rho(x)^n}-u(w)
\]
almost everywhere on $B^{1/2}(x)$. Both sides are holomorphic there, so the identity holds everywhere. Since $F_x(x)=0$ and $u(x)=0$, we
conclude that
\[
f(x)=
\frac{e^{F_x(x)+\varphi(x)}}{\rho(x)^n}=\frac{e^{\varphi(x)}}{\rho(x)^n}.
\]
This completes the proof of this lemma.
\end{proof}
 
\begin{prop}[Near-diagonal lower bounds]
Let $\varphi \in \mathcal W^*(\CC)$ with $\rho \in \mathcal S$. Then there exist constants $C,r>0$ such that
\begin{equation}
    |K_z^\varphi(w)|\geq C\frac{e^{\varphi(z)+\varphi(w)}}{\rho(z)^n\rho(w)^n}
    \label{eq:general-upper1}
\end{equation}
whenever $d_\rho(z,w)<r$.
\end{prop}
\begin{proof}
By Lemma~\ref{rhoB}\textup{(C)}, it is enough, after changing
the value of $r$, to prove the estimate for $w\in B^r(z)$ with $r$
sufficiently small. From the fact that
\[
K_z^\varphi(z) = \sup\{|f(z)|^2 : f \in F_{\varphi}^2, \|f\|_{F_{\varphi}^2} \leq 1\}
\]
and Lemma \ref{peak}, we see that \(K_z^{\varphi}(z) \geq  \frac{C_1e^{2\widetilde{\varphi}(z)}}{\rho(z)^{2n}}\). 
 Set
\[
G_z(\xi):=|K_z^\varphi(\xi)|e^{-\widetilde\varphi(\xi)}.
\]
The function $G_z$ is locally Lipschitz.  If $w\in B^r(z)$ and
$\gamma(t)=z+t(w-z)$, then
$t\mapsto K_z^\varphi(\gamma(t))$ is not identically zero because
$K_z^\varphi(z)>0$. Consequently,
\[
|G_z(w)-G_z(z)|
\le
2|w-z|\int_0^1
|\nabla G_z(\gamma(t))|\,dt.
\]
From Lemma \ref{meaninequality} (B) and (C), Cauchy--Schwarz inequality and \eqref{eq:general-upper}, we deduce that there exists $C_2, C_3>0$ such that for $ w \in B^r(z)$,
\begin{align*}
|G_z(w)-G_z(z)|
&\le
C_2\frac{|z-w|}{\rho(z)^{2n+1}}
\int_{B^R(z)}
|K_z^\varphi(\xi)|e^{-\varphi(\xi)}\,dV(\xi)
\\
&\le
\frac{C_2r}{\rho(z)^n}
\left(
\int_{B^R(z)}
|K_z^\varphi(\xi)|^2e^{-2\varphi(\xi)}\,dV(\xi)
\right)^{1/2}
\\
&\le
\frac{C_3r e^{\widetilde\varphi(z)}}{\rho(z)^{2n}}.
\end{align*} 
If we pick $r$ sufficiently small such that $C_1-C_3r>0$, then 
$$|K^\varphi_z(w)|e^{-\widetilde\varphi(w)}\geq |K^\varphi_z(z)|e^{-\widetilde\varphi(z)}-\frac{C_3re^{\widetilde\varphi(z)}}{\rho(z)^{2n}}\geq (C_1-C_3 r)\frac{e^{\widetilde\varphi(z)}}{\rho(z)^{2n}}.$$
This together with \eqref{rhozw} and \eqref{varphi} implies \eqref{eq:general-upper1}. The proof is complete.
\end{proof}

 Next we show that the distance $d_\rho$ is comparable to the Bergman distance $\varrho$.  It is well known, see \cite[Page 1]{N03} that the Bergman metric $\beta$ admits the extremal characterization
\[\beta(z;w)=\frac{\sup\left\{|df_z(w)|:f\in F_\varphi^2,\ f(z)=0,\ \|f\|_{F_\varphi^2}=1
\right\}}{\sqrt{K^\varphi_z(z)}},\quad  df_z(w)=\sum_{k=1}^n w_k\partial_k f(z).\]
  We denote by $\varrho$ the distance induced by the Bergman metric $\beta$.
 \begin{lem}\label{bergmetric}
Let $\varphi \in \mathcal W^*(\CC)$ with $\rho \in \mathcal S$. Then there exists a constant $C>0$, independent of $z, w$, such
that
$$C^{-1}d_{\rho}(z,w)\leq \varrho(z,w)\leq Cd_{\rho}(z,w).$$
\end{lem}

 \begin{proof}
We first compare the corresponding infinitesimal metrics. We claim that
there exists a constant \(C\geq1\) such that
\begin{equation}\label{bergman-infinitesimal}
C^{-1}\frac{|v|}{\rho(z)}
\leq
\beta(z;v)
\leq
C\frac{|v|}{\rho(z)},
\qquad z,v\in\mathbb C^n.
\end{equation}
By Theorem~\ref{main},
\begin{equation}\label{diagonal-kernel-equivalence}
\sqrt{K_z^\varphi(z)}
\simeq
\frac{e^{\varphi(z)}}{\rho(z)^n}.
\end{equation}
Let \(f\in F_\varphi^2\) satisfy \(f(z)=0\) and
\(\|f\|_{F_\varphi^2}=1\). By Lemma~\ref{meaninequality}(C),
\[
|df_z(v)|
\lesssim
\frac{e^{\varphi(z)}}{\rho(z)^{n+1}}|v|.
\]
 Using
\eqref{diagonal-kernel-equivalence}, we obtain
\[
\beta(z;v)
\lesssim
\frac{|v|}{\rho(z)}.
\]
It remains to prove the reverse inequality. The case \(v=0\) is
immediate, so assume that \(v\neq0\). Define
\[
L_{z,v}(\xi)
:=
\sum_{j=1}^n\overline{v_j}(\xi_j-z_j),
\qquad \xi\in\mathbb C^n,
\]
and set $g_{z,v}(\xi):=L_{z,v}(\xi)K_z^\varphi(\xi).$
Then \(g_{z,v}\in F_\varphi^2\), \(g_{z,v}(z)=0\), and
\[
d(g_{z,v})_z(v)
=
|v|^2K_z^\varphi(z).
\]
We next estimate the norm of \(g_{z,v}\). By \eqref{eq:general-upper} and Lemma~\ref{z-wk2n}(A),
\[
\begin{aligned}
\|g_{z,v}\|_{F_\varphi^2}^2
&\leq
|v|^2
\int_{\mathbb C^n}
|\xi-z|^2
|K_z^\varphi(\xi)|^2e^{-2\varphi(\xi)}
\,dV(\xi)
\\
&\lesssim
|v|^2
\frac{e^{2\varphi(z)}}{\rho(z)^{2n}}
\int_{\mathbb C^n}
|\xi-z|^2
\rho(\xi)^{-2n}
e^{-2\varepsilon d_\rho(z,\xi)}
\,dV(\xi)\\
&\lesssim
|v|^2\rho(z)^2
\frac{e^{2\varphi(z)}}{\rho(z)^{2n}}
\lesssim
|v|^2\rho(z)^2K_z^\varphi(z).
\end{aligned}
\]
Let
\[
h_{z,v}:=
\frac{g_{z,v}}{\|g_{z,v}\|_{F_\varphi^2}}.
\]
Then \(h_{z,v}(z)=0\) and
\(\|h_{z,v}\|_{F_\varphi^2}=1\). Hence, by the extremal
characterization of the Bergman metric,
\[
\begin{aligned}
\beta(z;v)
&\geq
\frac{|d(h_{z,v})_z(v)|}
{\sqrt{K_z^\varphi(z)}}
=
\frac{|v|^2K_z^\varphi(z)}
{\|g_{z,v}\|_{F_\varphi^2}\sqrt{K_z^\varphi(z)}}
\gtrsim
\frac{|v|}{\rho(z)}.
\end{aligned}
\]
This proves \eqref{bergman-infinitesimal}.

Now let \(\gamma:[0,1]\to\mathbb C^n\) be a piecewise
\(C^1\) curve joining \(z\) to \(w\). By
\eqref{bergman-infinitesimal},
\[
C^{-1}
\int_0^1
\frac{|\gamma'(t)|}{\rho(\gamma(t))}
\,dt
\leq
\int_0^1
\beta(\gamma(t);\gamma'(t))
\,dt
\leq
C
\int_0^1
\frac{|\gamma'(t)|}{\rho(\gamma(t))}
\,dt.
\]
Taking the infimum over all such curves \(\gamma\) yields
\[
C^{-1}d_\rho(z,w)
\leq
\varrho(z,w)
\leq
Cd_\rho(z,w).
\]
This completes the proof.
\end{proof}

We next derive $L^p$ estimates for the Bergman kernel $K^\varphi_z$ and $z\in\CC$.
\begin{prop} \label{prop:normkz} 
Let $\varphi \in \mathcal W^*(\CC)$ with $\rho \in \mathcal S$. For $0< p \leq \infty$, we have 
\begin{align}\label{Kzrho}
\|K^\varphi_z\|_{L^p_\varphi} \simeq e^{\varphi(z)} \rho(z)^{\frac{2n}{p}-2n}
\end{align}
for every $z\in \CC$.
\end{prop}

\begin{proof}
\textbf{Case 1}: $0< p < \infty$. From Theorem \ref{main}, we know that
\begin{align*}
  \|K^\varphi_z\|_{L^p_\varphi}^p \geq \int_{B^r(z)}|K^\varphi_z(w)|^pe^{-p\varphi(w)}\,dV(w)
  \gtrsim e^{p\varphi(z)} \rho(z)^{2n-2pn},
\end{align*}
where $r$ is from Theorem \ref{main}. By Theorem \ref{main}  and  Lemma \ref{z-wk2n}, we deduce that
\begin{align*}
  \|K^\varphi_z\|_{L^p_\varphi}^p \lesssim \frac{e^{p\varphi(z)}}{\rho(z)^{np}}\int_{\CC}\frac{e^{- p\varepsilon d_\rho(z, w)}}{\rho(w)^{np}}\,dV(w)\lesssim e^{p\varphi(z)} \rho(z)^{2n-2pn}.
\end{align*}
This gives us \eqref{Kzrho} when $0< p < \infty$. 

\textbf{Case 2}: $p=\infty$. By Lemma \ref{meaninequality} and \eqref{Kzrho} with $p=1$, we obtain
\begin{align*}
	|K^\varphi_z(w)e^{-\varphi(w)}| =e^{-\varphi(w)+\varphi(z)} |K_w^\varphi(z)e^{-\varphi(z)}| \lesssim \frac{e^{-\varphi(w)+\varphi(z)}}{\rho(z)^{2n}} \int_{B(z)} |K_w^\varphi(\xi)|e^{-\varphi(\xi)}\,d V(\xi) \lesssim  \frac{e^{\varphi(z)}}{\rho(z)^{2n}},
\end{align*}
which implies that $\|K^\varphi_z\|_{L^\infty_\varphi}\lesssim e^{\varphi(z)}\rho(z)^{-2n}$. On the other hand, by  Theorem \ref{main},
$$\|K^\varphi_z\|_{L^\infty_\varphi}\geq |K^\varphi_z(z)e^{-\varphi(z)}|\gtrsim e^{\varphi(z)}\rho(z)^{-2n}.$$
This completes the proof.  
\end{proof}

\section{Bergman Projection, Duality and Complex Interpolation} \label{sec:proj}
In this section, we study the Bergman projection on the weighted Fock spaces $F^p_\varphi$, together with its duality and complex interpolation properties, both of which follow from the boundedness of the Bergman projection.

It is straightforward to see from the reproducing property of the Bergman kernel for $F^2_\varphi$ that the Bergman projection $P_\varphi \colon L^2_\varphi \to F^2_\varphi$ is the integral operator given by 
$$P_\varphi f(z) := \int_\CC f(w) \overline{K^\varphi_z(w)} e^{-2\varphi(w)} \,dV(w), \qquad f\in L^2_\varphi, z\in\CC.$$ 
The following theorem proves that the Bergman projection $P_{\varphi}$ is bounded. 
\begin{thm} \label{thm:projbnd} 
Let $\varphi \in \mathcal W^*(\CC)$ with $\rho \in \mathcal S$. Then the Bergman projection $P_\varphi$ is a bounded linear operator from $L^p_\varphi$ to $F^p_\varphi$ for any $1\leq p \leq \infty$. Moreover, $P_\varphi$ is a bounded linear operator on $F^p_\varphi$ for any $0< p \leq \infty$.  
\end{thm} 
 
\begin{proof} 
For $f \in L^1_\varphi$, Fubini's theorem and Proposition \ref{prop:normkz} yield,  
\begin{align*} 
\|P_\varphi f\|_{L^1_\varphi} &\le \int_{\mathbb{C}^n} e^{-\varphi(z)} \, dV(z) \int_{\CC} |K^\varphi_z(w)f(w)| e^{-2\varphi(w)} \, dV(w) \\ 
&= \int_{\CC} |f(w)| e^{-2\varphi(w)} \, dV(w) \int_{\mathbb{C}^n} |K^\varphi_z(w)| e^{-\varphi(z)} \, dV(z) \\ 
&\le C \|f\|_{L^1_\varphi}. 
\end{align*} 
Similarly,  for $f \in L^\infty_\varphi$, 
\begin{align*} 
\|P_\varphi f\|_{L^\infty_\varphi} &\le \sup_{z \in \CC} e^{-\varphi(z)} \int_{\CC} |K^\varphi_z(w)f(w)| e^{-2\varphi(w)} \, dV(w) \\ 
&\le \|f\|_{L^\infty_\varphi} \sup_{z \in \CC} e^{-\varphi(z)} \int_{\CC} |K^\varphi_z(w)| e^{-\varphi(w)} \, dV(w) \\ 
&\le C \|f\|_{L^\infty_\varphi}. 
\end{align*} 
Thus, $P_\varphi$ is bounded from $L^p_\varphi$ to $F^p_\varphi$ when $p=1$ and $p=\infty$. 
 
For $f\in L_\varphi^p$ with $1<p<\infty$, H\"older's inequality and 
Fubini's theorem give 
\begin{align*} 
\|P_\varphi f\|_{L_\varphi^p}^p 
&\leq 
\int_{\mathbb C^n} e^{-p\varphi(z)} 
\left( 
\int_{\mathbb C^n} 
|K^\varphi_z(w)f(w)|e^{-2\varphi(w)}\,dV(w) 
\right)^p dV(z) 
\\ 
&\leq 
\int_{\mathbb C^n}\int_{\mathbb C^n} 
|f(w)|^p e^{-p\varphi(w)} 
|K^\varphi_z(w)|e^{-\varphi(w)}\,dV(w)\, 
\|K^\varphi_z\|_{L_\varphi^1}^{p-1} 
e^{-p\varphi(z)}\,dV(z) 
\\ 
&\leq 
C\int_{\mathbb C^n}e^{-\varphi(z)}\,dV(z) 
\int_{\mathbb C^n} 
|f(w)|^p e^{-p\varphi(w)} 
|K^\varphi_z(w)|e^{-\varphi(w)}\,dV(w) 
\\ 
&\leq 
C\int_{\mathbb C^n} 
|f(w)|^p e^{-p\varphi(w)}e^{-\varphi(w)}\,dV(w) 
\int_{\mathbb C^n} 
|K^\varphi_z(w)|e^{-\varphi(z)}\,dV(z) 
\\ 
&\leq 
C\|f\|_{L_\varphi^p}^p. 
\end{align*} 
Thus, $P_\varphi$ is bounded from $L_\varphi^p$ to $F_\varphi^p$ 
for every $1\leq p\leq\infty$. 
  
 Next, we treat the case $0<p<1$. First, we claim that $P_\varphi$ is well defined on $F^p_\varphi$. In fact, given any $f \in F^p_\varphi$, by Lemma \ref{meaninequality}, Theorem \ref{main} and Lemma \ref{z-wk2n}, we obtain for every fixed $z\in\CC$ , 
\begin{align*} 
   \int_{\CC} |K^\varphi_z(w)f(w)|e^{-2\varphi(w)} \, dV(w) 
    &\le C \|f\|_{F^p_\varphi} \int_{\CC} \rho(w)^{-\frac{2n}{p}} |K^\varphi_z(w)| e^{-\varphi(w)} \, dV(w) \\ 
    &\le C \|f\|_{F^p_\varphi} e^{\varphi(z)} \rho(z)^{-n} \int_{\CC} \rho(w)^{-\frac{2n}{p}-n} e^{-\varepsilon d_\rho(z,w)} \, dV(w) \\ 
    &\le C \|f\|_{F^p_\varphi} e^{\varphi(z)} \rho(z)^{-\frac{2n}{p}} <\infty. 
\end{align*} 
 
Now, we prove the boundedness of $P_\varphi$. Let $\{a_k\}_{k=1}^\infty$ be the $1$-lattice in $\CC$. For any  $z\in\CC$ and $f \in F^p_\varphi$ with $0<p<1$, we get 
\begin{align*} 
    |P_\varphi f(z)|^p &\le \left( \sum_{k=1}^\infty \int_{B(a_k)} |f(w)K^\varphi_z(w)|e^{-2\varphi(w)} \, dV(w) \right)^p \\ 
    &\le \sum_{k=1}^\infty \left(\int_{B(a_k)} |f(w)K^\varphi_z(w)| e^{-2\varphi(w)} \, dV(w)\right)^p \\ 
    &\le \sum_{k=1}^\infty \rho(a_k)^{2np} \left( \sup_{w \in B(a_k)} |f(w)K^\varphi_z(w)|e^{-2\varphi(w)} \right)^p. 
\end{align*} 
 
From the definition of $\varphi$, if we replace the weight $\varphi$ by $2\varphi$, then there exists $\rho_{2\varphi}\in \mathcal{S}$ such that $i\partial\bar{\partial}(2\widetilde\varphi)\simeq \omega_{\rho_{2\varphi}}$ with $\rho_{2\varphi}\simeq \rho$. Applying  Lemma \ref{meaninequality} with weight $2\varphi$ instead of $\varphi$, there is some constant $C > 0$ such that $|P_\varphi f(z)|^p$ is bounded by $C$ times
$$ \sum_{k=1}^\infty \rho(a_k)^{2np-2n} \sup_{w \in B(a_k)} \int_{B(w)} |f(u)|^p |K^\varphi_z(u)|^p e^{-2p\varphi(u)} \, dV(u). $$ 
Combining this with \eqref{Bzwz} and \eqref{NNN}, we obtain 
\begin{align*} 
    |P_\varphi f(z)|^p &\le C \sum_{k=1}^\infty \int_{B^{m_1}(a_k)} \rho(u)^{2np-2n} |f(u)|^p |K^\varphi_z(u)|^p e^{-2p\varphi(u)} \, dV(u) \\ 
    &\le CN \int_{\mathbb{C}^n} \rho(u)^{2np-2n} |f(u)|^p |K^\varphi_z(u)|^p e^{-2p\varphi(u)} \, dV(u). 
\end{align*} 
Therefore, it follows from Fubini's theorem and Proposition \ref{prop:normkz} that  
\begin{align*} 
    &\int_{\CC} |P_\varphi f(z)|^p e^{-p\varphi(z)} \, dV(z)  \\ 
    &\le C \int_{\CC} \int_{\CC} |K^\varphi_z(u)|^p e^{-p\varphi(z)} \, dV(z) \rho(u)^{2np-2n} |f(u)|^p e^{-2p\varphi(u)} \, dV(u) \\ 
    &\le C \int_{\CC} |f(u)|^p e^{-p\varphi(u)} \, dV(u). 
\end{align*} 
This means that $P_\varphi$ is bounded on $F^p_\varphi$ for $0<p<1$. The proof is complete. 
\end{proof} 

 We equip $L_\varphi^2$ with the inner product
$$\langle f, g \rangle_\varphi:= \int_{\mathbb C^n} f(w) \overline{g(w)} e^{-2\varphi(w)} \, dV(w),\quad f,g\in L^2_\varphi.$$
\begin{cor} \label{fggf}
If $1 \le p \le \infty$ and $p'$ is the conjugate exponent of $p$, then
$$ \langle P_\varphi f, g \rangle_\varphi = \langle f, P_\varphi g \rangle_\varphi $$
 for every  $f \in L^p_\varphi$  and  $g \in L^{p'}_\varphi$.
 \end{cor}
\begin{proof} 
It follows from Fubini's theorem. Note that the hypothesis of Fubini's theorem holds due to Hölder's inequality and the $L^p$-boundedness of the Bergman projection $P_\varphi$.
\end{proof}

\begin{thm} \label{Pf=fp}
Let $1 \le p \le \infty$. Then $f = P_\varphi f$, for every $f \in F^p_\varphi$.
\end{thm}

\begin{proof}
Choose \(R>1\) and \(h_0\in C_0^\infty(\mathbb C^n)\) such that \(h_0=1\) on \(B(0,R)\), \(h_0=0\) outside \(B(0,R+1)\), and \(|\bar\partial h_0|\lesssim1\). Set
\[
h_j(w):=h_0(w/j),\qquad A_j:=\{w:jR<|w|<j(R+1)\}.
\]
Then \(\operatorname{supp}\bar\partial h_j\subset A_j\) and \(|\bar\partial h_j|\lesssim j^{-1}\). Since \(fh_j\in L_\varphi^2\), we have
\[
P_\varphi(fh_j)=fh_j-u_j,\qquad u_j:=(I-P_\varphi)(fh_j),
\]
where \(u_j\) is the \(L_\varphi^2\)-minimal solution of
\[
\bar\partial u_j=\theta_j,\qquad \theta_j:=f\bar\partial h_j.
\]

Let \(\varepsilon_0>0\) be the constant in \eqref{eq:general-upper}. For each \(z\in\mathbb C^n\), take \(g_z\) as in \eqref{ade5} and set
\[
\omega_z(w):=e^{-\delta g_z(w)},
\]
where \(\delta>0\) is sufficiently small. Since
\[
|g_z(x)-g_z(w)|\lesssim d_\rho(x,w),\qquad
|\partial g_z|_{i\partial\bar\partial\widetilde\varphi}\lesssim1,
\]
we may choose \(\delta\) so that
\[
|\partial\omega_z|_{i\partial\bar\partial\widetilde\varphi}
\leq\varepsilon_1\omega_z,\qquad 0<\varepsilon_1<\sqrt2,
\quad
\text{and}
\quad
\frac{\omega_z(x)^{1/2}}{\omega_z(w)^{1/2}}
\leq e^{\varepsilon_0d_\rho(x,w)/2}.
\]
Consequently, \eqref{eq:general-upper} gives
\[
|K_x^\varphi(w)|e^{-\varphi(x)-\varphi(w)}
\frac{\omega_z(x)^{1/2}}{\omega_z(w)^{1/2}}
\lesssim
\frac{e^{-\varepsilon_0d_\rho(x,w)/2}}
{\rho(x)^n\rho(w)^n}.
\]
Lemma~\ref{z-wk2n}(A) and Schur's test therefore imply
\begin{equation}\label{weighted-projection-estimate}
\int_{\mathbb C^n}|P_\varphi q|^2e^{-2\varphi}\omega_z\,dV
\lesssim
\int_{\mathbb C^n}|q|^2e^{-2\varphi}\omega_z\,dV,
\end{equation}
with a constant independent of \(z\).

Let \(v_j\) be the \(L_{\widetilde\varphi}^2\)-minimal solution of \(\bar\partial v_j=\theta_j\). Theorem~\ref{df=theta}, applied with \(\omega_z\), gives
\[
\int_{\mathbb C^n}|v_j|^2e^{-2\widetilde\varphi}\omega_z\,dV
\lesssim
\int_{\mathbb C^n}
|\theta_j|_{i\partial\bar\partial\widetilde\varphi}^2
e^{-2\widetilde\varphi}\omega_z\,dV.
\]
Applying the same theorem with \(\omega\equiv1\) shows that \(v_j\in L_\varphi^2\). Since \(fh_j-v_j\in F_\varphi^2\), we have
\[
u_j=(I-P_\varphi)(fh_j)=(I-P_\varphi)v_j.
\]
Combining this identity with \eqref{weighted-projection-estimate}, \(\|\varphi-\widetilde\varphi\|_\infty<\infty\), and \eqref{ade1}, we obtain
\begin{equation}\label{weighted-uj-estimate}
\int_{\mathbb C^n}|u_j(w)|^2e^{-2\varphi(w)}\omega_z(w)\,dV(w)
\lesssim
\int_{\mathbb C^n}\rho(w)^2|\theta_j(w)|^2e^{-2\varphi(w)}
\omega_z(w)\,dV(w).
\end{equation}

Fix a compact set \(E\subset\mathbb C^n\). For all sufficiently large \(j\), \(h_j=1\) on \(E\) and \(B(z)\cap A_j=\varnothing\) for every \(z\in E\). Hence \(u_j\) is holomorphic on \(B(z)\). Moreover, \(\omega_z\simeq1\) on \(B(z)\) and \(\omega_z(w)\lesssim e^{-c\delta d_\rho(z,w)}\). Lemma~\ref{meaninequality}(A) and \eqref{weighted-uj-estimate} therefore yield
\[
|u_j(z)|^2e^{-2\varphi(z)}
\lesssim
\frac{1}{\rho(z)^{2n}}
\int_{\mathbb C^n}\rho(w)^2|\theta_j(w)|^2e^{-2\varphi(w)}
e^{-c\delta d_\rho(z,w)}\,dV(w).
\]
Since \(|\theta_j(w)|\lesssim j^{-1}|f(w)|\chi_{A_j}(w)\), this becomes
\[
|u_j(z)|^2e^{-2\varphi(z)}
\lesssim
\frac{1}{j^2\rho(z)^{2n}}
\int_{\mathbb C^n}|f(w)|^2e^{-2\varphi(w)}
\rho(w)^2e^{-c\delta d_\rho(z,w)}\,dV(w).
\]
For \(1\leq p\leq\infty\), with the convention \(1/\infty=0\), Lemma~\ref{meaninequality}(A) gives
\[
|f(w)|^2e^{-2\varphi(w)}
\lesssim
\|f\|_{L_\varphi^p}^2\rho(w)^{-4n/p}.
\]
Using Lemma~\ref{z-wk2n}(A) with \(l=2-4n/p\), we conclude that
\begin{equation}\label{uj-decay-estimate}
|u_j(z)|^2e^{-2\varphi(z)}
\lesssim
\frac{\|f\|_{L_\varphi^p}^2}{j^2}
\rho(z)^{2-4n/p}.
\end{equation}
Thus \(u_j\to0\) locally uniformly on \(\mathbb C^n\).

Suppose first that \(1\leq p<\infty\). The boundedness of \(P_\varphi\) gives
\[
\|P_\varphi f-P_\varphi(fh_j)\|_{L_\varphi^p}
\lesssim
\|f(1-h_j)\|_{L_\varphi^p}\longrightarrow0.
\]
By Lemma~\ref{meaninequality}(A), \(P_\varphi(fh_j)(z)\to P_\varphi f(z)\) locally uniformly. On the other hand, for every fixed \(z\) and all sufficiently large \(j\), \(h_j(z)=1\), so
\[
P_\varphi(fh_j)(z)=f(z)-u_j(z)\longrightarrow f(z)
\]
by \eqref{uj-decay-estimate}. Hence \(P_\varphi f=f\).

If \(p=\infty\), then for every fixed \(z\),
\[
P_\varphi(fh_j)(z)
=\int_{\mathbb C^n}f(w)h_j(w)\overline{K_z^\varphi(w)}
e^{-2\varphi(w)}\,dV(w).
\]
Since
\[
|f(w)h_j(w)\overline{K_z^\varphi(w)}|e^{-2\varphi(w)}
\leq
\|f\|_{L_\varphi^\infty}|K_z^\varphi(w)|e^{-\varphi(w)}
\]
and the right-hand side is integrable, dominated convergence gives \(P_\varphi(fh_j)(z)\to P_\varphi f(z)\). At the same time, \eqref{uj-decay-estimate} gives \(P_\varphi(fh_j)(z)=f(z)-u_j(z)\to f(z)\). Therefore \(P_\varphi f=f\). This completes the proof.
\end{proof}

As a consequence of the boundedness of the Bergman projection, we obtain the following two theorems on the complex interpolation and duality of $F^p_\varphi$. 

\begin{thm}\label{cominter}
If $1 \leq p_1 \leq p_2 \leq \infty$ and $0 < \theta < 1$, then $L_{\varphi}^{p_{\theta}} = [L_{\varphi}^{p_1}, L_{\varphi}^{p_2}]_{\theta}$ and $F_{\varphi}^{p_{\theta}} = [F_{\varphi}^{p_1}, F_{\varphi}^{p_2}]_{\theta}$ with equivalent norms, where
\[
\frac{1}{p_{\theta}} = \frac{1 - \theta}{p_1} + \frac{\theta}{p_2}.
\]
\end{thm}

\begin{proof}
Let $1\leq p_{\theta}\leq \infty$. We define two bounded linear operators $T : L_{\varphi}^{p_{\theta}} \to L^{p_{\theta}}(\CC)$ by $Tf = f e^{-\varphi}$ and $S : L^{p_{\theta}}(\CC) \to L_{\varphi}^{p_{\theta}}$ by $Sg = g e^{\varphi}$. It is obvious that $ST = I$, where $I$ is the identity on $L_{\varphi}^{p_{\theta}}$.
It follows from \cite[Theorem 2.5]{Zhu07} that $L^{p_{\theta}}(\CC) = [L^{p_1}(\CC), L^{p_2}(\CC)]_{\theta}$.  Thus,
\[
[L_{\varphi}^{p_1}, L_{\varphi}^{p_2}]_{\theta} = S\bigl([L^{p_1}(\CC), L^{p_2}(\CC)]_{\theta}\bigr) = S\bigl(L^{p_{\theta}}(\CC)\bigr) = L_{\varphi}^{p_{\theta}},
\]
with equivalent norms. Since $P_\varphi$ is a bounded projection from $L_{\varphi}^p$ onto $F_{\varphi}^p$ when $1 \leq p \leq \infty$, the standard retraction argument for complex interpolation gives
\[
[F_{\varphi}^{p_1}, F_{\varphi}^{p_2}]_{\theta} = P_\varphi\bigl([L_{\varphi}^{p_1}, L_{\varphi}^{p_2}]_{\theta}\bigr) = P_\varphi\bigl(L_{\varphi}^{p_{\theta}}\bigr) = F_{\varphi}^{p_{\theta}},
\]
with equivalent norms. This completes the proof.
\end{proof}

\begin{thm} \label{dual}
Let $1\leq p<\infty$, and let $p'$ be determined by
$1/p+1/p'=1,$
with the convention that $p'=\infty$ when $p=1$. Then the dual of $F^p_\varphi$ is identified with $F^{p'}_\varphi$ under the integral pairing $\langle \cdot, \cdot \rangle_\varphi$.
\end{thm}
\begin{proof} For any $g \in F^{p'}_\varphi$, we define an operator $A$ by $A_g(\cdot) = \langle \cdot, g \rangle_\varphi$. Then for any $f \in F^p_\varphi$, by Hölder's inequality, we have $|A_g(f)| \le \|g\|_{L^{p'}_\varphi} \|f\|_{L^p_\varphi}$. It follows that $A_g$ is a bounded linear functional on $F^p_\varphi$ with norm at most $\|g\|_{L^{p'}_\varphi}$.

On the other hand, let $A$ be a bounded linear functional on $F^p_\varphi$. Since $F^p_\varphi$ is a closed subspace of $L^p_\varphi$, the Hahn-Banach extension theorem shows that $A$ can be extended to a bounded linear functional $\tilde{A}$ on $L^p_\varphi$. The usual weighted $L^p$ duality identifies the dual of $L^p_\varphi$ with $L^{p'}_\varphi$, so one can find $h\in L_\varphi^{p'}$ such that
$$\widetilde A(f)=\langle f,h\rangle_\varphi
\quad\text{and}\quad
\|h\|_{L_\varphi^{p'}}=\|\widetilde A\|.$$
Set $g=P_\varphi h$. Then by Theorem \ref{thm:projbnd}, we have $\|g\|_{L^{p'}_\varphi} \lesssim \|h\|_{L^{p'}_\varphi} = \|A\|$.  Therefore, by Corollary \ref{fggf}, for $f \in F^p_\varphi$,
$$ A(f) = \tilde{A}(f) = \langle f, h \rangle_\varphi = \langle P_\varphi f, h \rangle_\varphi = \langle f, P_\varphi h \rangle_\varphi = \langle f, g \rangle_\varphi = A_g(f). $$
This completes the proof. 
\end{proof}
\begin{cor} \label{dense}
For any $1 \le p < \infty$, the linear span $E$ of all the Bergman kernels $\{K_z^\varphi,z \in \CC\}$, is dense in $F^p_\varphi$.
\end{cor}
\begin{proof} By Theorem \ref{dual} and the Hahn-Banach theorem, we only need to prove that if $f \in F^{p'}_\varphi$ satisfies $\langle f, g \rangle_\varphi = 0$ for every $g \in E$, then $f = 0$. From Theorem \ref{Pf=fp}, it follows that $f(z) = P_\varphi f(z) = \langle f, K_z^\varphi \rangle_\varphi = 0$ for every $z \in \CC$ and the proof is complete.
\end{proof}

\section{\texorpdfstring{$L^p$ Estimates for the $\bar{\partial}$ Equation and
 Applications to Hankel Operators}{Lp Estimates for the d bar Equation and Applications to Hankel Operators}}\label{canonical}
 In this section, we study the globally convexifiable subclass
$\mathcal W_{\mathrm{gc}}^*(\mathbb C^n)$ of
$\mathcal W^*(\mathbb C^n)$. We first establish $L^p$ estimates for
the $\bar\partial$ equation and then apply them to characterize the
boundedness and compactness of Hankel operators with general, possibly
unbounded, symbols.

More precisely, we identify $\mathbb C^n$ with $\mathbb R^{2n}$ by
writing $z=(z_1,\ldots,z_n)$ with $z_j=x_{2j-1}+ix_{2j},$
and associating \(z\) with
$x=(x_1,\ldots,x_{2n})\in\mathbb R^{2n}.$
For a function \(u\in\mathcal C^2(\mathbb R^{2n})\), let
\[
D_{\mathbb R}^2u(z)
:=
\left(
\frac{\partial^2u}
{\partial x_j\partial x_k}(x)
\right)_{j,k=1}^{2n}
\]
denote its real Hessian at \(z\). Let $\varphi\in\mathcal W^*(\CC)$ with $\rho\in\mathcal S$. We say that $\varphi$ is globally $\rho$-convexifiable if there exist a $\rho$-regularization $\widetilde\varphi$ with doubling Monge--Amp\`ere measure, an entire function $F\in H(\mathbb C^n)$ and constants $0<m\leq M<\infty$ such that
\begin{align}\label{111}
m\rho(z)^{-2}I_{2n}
\leq
D_{\mathbb R}^2
\bigl(\widetilde\varphi-\operatorname{Re}F\bigr)(z)
\leq
M\rho(z)^{-2}I_{2n},
\qquad z\in\mathbb C^n.
\end{align}
Here $I_{2n}$ is the $2n\times2n$ identity matrix. For symmetric matrices
$A$ and $B$, we use the convention that $A\leq B$ if $B-A$ is positive
semidefinite. Let $\mathcal W_{\mathrm{gc}}^*(\CC)$ denote the class of all weights $\varphi\in\mathcal W^*(\CC)$ satisfying \eqref{111}. This class contains the Gaussian weights $\varphi(z)=\alpha|z|^2/2, \alpha>0,$
as well as plurisubharmonic weights at bounded distance from $\varphi(z)=|z|^2-\frac{1}{2}\log(1+|z|^2),$
which induce norms equivalent to those of the Fock--Sobolev spaces, see \cite{Zhu12,Cho12}. More importantly, this class contains weights whose curvature does not
satisfy uniform upper and lower bounds. For example, let $\alpha>1$
with $\alpha\neq2$, and set
$\varphi_\alpha(z)=|z|^\alpha+\operatorname{Re}(z_1^2)$ with
$ z\in\CC.$

For $\varphi\in \mathcal W_{\mathrm{gc}}^*(\CC)$, we let $\psi=\widetilde{\varphi}-\operatorname{Re}F$.
For $t=(t_{1},t_{2},\ldots,t_{2n})\in\mathbb{R}^{2n}$, write
$\xi_{j}=t_{2j-1}+it_{2j}$ and
$\xi=(\xi_{1},\xi_{2},\ldots,\xi_{n})$. An elementary calculation
similar to that on page 125 of \cite{K92} shows that
\[
\operatorname{Re}
\sum_{j,k=1}^{n}
\frac{\partial^{2}\psi}
     {\partial z_{j}\partial z_{k}}(z)
\xi_{j}\xi_{k}
+
\sum_{j,k=1}^{n}
\frac{\partial^{2}\psi}
     {\partial z_{j}\partial\overline{z}_{k}}(z)
\xi_{j}\overline{\xi}_{k}
=
\frac{1}{2}
\sum_{j,k=1}^{2n}
\frac{\partial^{2}\psi}
     {\partial x_{j}\partial x_{k}}(x)
t_{j}t_{k}
\geq
\frac{1}{2}m\rho(z)^{-2}|\xi|^{2}.
\]
Replacing $\xi$ with $i\xi$ in the above inequality gives
\[
-\operatorname{Re}
\sum_{j,k=1}^{n}
\frac{\partial^{2}\psi}
     {\partial z_{j}\partial z_{k}}(z)
\xi_{j}\xi_{k}
+
\sum_{j,k=1}^{n}
\frac{\partial^{2}\psi}
     {\partial z_{j}\partial\overline{z}_{k}}(z)
\xi_{j}\overline{\xi}_{k}
\geq
\frac{1}{2}m\rho(z)^{-2}|\xi|^{2}.
\]
Thus,
\[
\sum_{j,k=1}^{n}
\frac{\partial^{2}\psi}
     {\partial z_{j}\partial\overline{z}_{k}}(z)
\xi_{j}\overline{\xi}_{k}
\geq
\frac{1}{2}m\rho(z)^{-2}|\xi|^{2}.
\]
Similarly, we have an upper bound for the complex Hessian of
$\psi$. Therefore,
$
\omega_\rho\simeq i\partial\bar{\partial}\psi,
$
where
$
\omega_\rho(z)=\rho(z)^{-2}i\partial\bar\partial|z|^2$ for $z\in\CC$.

Given two nonnegative integers $s,t\leq n$, we write
\begin{align}\label{2-10}
\omega
=
\sum_{|\alpha|=s,\,|\beta|=t}
\omega_{\alpha,\beta}\,
dz^{\alpha}\wedge d\overline{z}^{\beta}
\end{align}
for a differential form of type $(s,t)$ as in \cite{K92}.
We denote by $L_{s,t}$ the family of all $(s,t)$-forms $\omega$
as in \eqref{2-10} with coefficients $\omega_{\alpha,\beta}$
measurable on $\mathbb{C}^n$, and set
\begin{equation}
|\omega|
=
\sum_{|\alpha|=s,\,|\beta|=t}
|\omega_{\alpha,\beta}|,
\qquad
\|\omega\|_{L^p_\psi}
=
\left\||\omega|\right\|_{L^p_\psi}.
\end{equation} 

Following \cite[Proposition~10]{BA82}, for $\omega\in L_{0,1}$, we define the normalized Berndtsson--Andersson operator by
\begin{align}\label{A}
A_{\psi}(\omega)(z)
:={}&
c_n\int_{\mathbb{C}^{n}}
e^{\langle 2\partial\psi(\xi),z-\xi\rangle}
\sum_{j=0}^{n-1}
\omega(\xi)\wedge
\frac{
\partial|\xi-z|^{2}
\wedge
\bigl(2\bar{\partial}\partial\psi(\xi)\bigr)^{j}
\wedge
\bigl(\bar{\partial}\partial|\xi-z|^{2}\bigr)^{n-1-j}
}{
j!\,|\xi-z|^{2n-2j}
},
\end{align}
 where
\[
\left\langle \partial\psi(\xi),z-\xi\right\rangle
=
\sum_{j=1}^{n}
\frac{\partial\psi}{\partial\xi_{j}}(\xi)
(z_{j}-\xi_{j}),
\]
and $c_n\neq0$ is the dimensional normalization constant. Under the
integrability hypothesis in \cite[Proposition~10]{BA82}, the corresponding
representation formula gives $\bar\partial A_\psi(\omega)=\omega$ for every
$\bar\partial$-closed $(0,1)$ form $\omega$.

Set
\[
\Gamma:=\left\{\sum_{j=1}^{N}a_jK^\varphi_{z_j}:N\in\mathbb{N},\;a_j\in\mathbb{C},\;z_j\in\CC
\text{ for }1\leq j\leq N\right\}.
\]
Corollary \ref{dense} tells us that $\Gamma$ is dense in
$F^p_\varphi$ for all $1\leq p<\infty$.

\begin{lem}\label{solve1}
Suppose $1\leq p\leq\infty$ and $\varphi\in \mathcal W_{\mathrm{gc}}^*(\CC)$. Let  $\psi=\widetilde{\varphi}-\operatorname{Re}F$.  For $\omega\in L_{0,1}$, set $$T_{\varphi}(\omega):=e^FA_{\psi}(e^{-F}\omega).$$
\textup{(A)}
For $\omega\in L_{0,1}$ with $\rho\omega\in L_\varphi^p$, there is a constant $C>0$ such that
\[
\|T_{\varphi}(\omega)\|_{L^p_\varphi}
\leq
C\|\omega\rho\|_{L^p_\varphi}.
\]
\textup{(B)}
Let $\kappa\in\mathbb R$. For every $g\in\Gamma$ and every
$f\in\mathcal C^2(\mathbb C^n)$ satisfying
$\rho^\kappa\bar\partial f\in L^p,$
we have
\[
\bar{\partial}T_{\varphi}
\bigl(g\bar{\partial}f\bigr)
=
g\bar{\partial}f.
\]
\end{lem}

\begin{proof}
Let \(0<\alpha\leq1\leq\beta\) be as in
Lemma~\ref{dzw}. Put \(x=z-\xi\) and
\(t=|x|/\rho(\xi)\). Since
\(i\partial\bar\partial\psi\simeq\omega_\rho\), we have
\(|\bar\partial\partial\psi(\xi)|\lesssim\rho(\xi)^{-2}\).
Moreover,
\[
|\partial|\xi-z|^2|\lesssim|\xi-z|,
\qquad
|\bar\partial\partial|\xi-z|^2|\lesssim1.
\]
Consequently, for \(0\leq j\leq n-1\),
\begin{align}\label{AA}
\left|\omega(\xi)\wedge\frac{\partial|\xi-z|^2\wedge
\bigl(2\bar\partial\partial\psi(\xi)\bigr)^j\wedge
\bigl(\bar\partial\partial|\xi-z|^2\bigr)^{n-1-j}}{j!\,|\xi-z|^{2n-2j}}
\right|
\lesssim
\frac{
|\omega(\xi)|\rho(\xi)^{-2j}
}{
|\xi-z|^{2n-2j-1}
}.
\end{align}
Define
\[
R_\psi(z,\xi)
:=
\psi(z)-\psi(\xi)
-2\operatorname{Re}
\langle\partial\psi(\xi),z-\xi\rangle.
\]
Taylor's formula and \eqref{111} give
\[
R_\psi(z,\xi)
=
\int_0^1(1-s)
D_{\mathbb R}^2\psi(\xi+sx)[x,x]\,ds
\gtrsim
|x|^2\int_0^1
\frac{1-s}{\rho(\xi+sx)^2}\,ds.
\]
If \(t\leq1\), Lemma~\ref{rhoB}\textup{(A)} gives
\(R_\psi(z,\xi)\gtrsim t^2\). If \(t>1\), applying
\eqref{z-w} along the segment from \(\xi\) to \(z\) gives
\(R_\psi(z,\xi)\gtrsim t^{2\alpha}\). Thus
\begin{align}\label{B(z)^c}
R_\psi(z,\xi)
\gtrsim
\min\{t^2,t^{2\alpha}\}.
\end{align}
It follows that
\[
\left|
e^{\langle2\partial\psi(\xi),z-\xi\rangle}
\right|e^{-\psi(z)}
=
e^{-\psi(\xi)}e^{-R_\psi(z,\xi)}
\lesssim
e^{-\psi(\xi)}
\left(e^{-ct^2}+e^{-ct^{2\alpha}}\right).
\]
Since
\[
\sum_{j=0}^{n-1}
t^{2j}
\left(e^{-ct^2}+e^{-ct^{2\alpha}}\right)
\lesssim
e^{-c_0t^2}+e^{-c_0t^{2\alpha}}
\]
for some \(c_0>0\), all the terms in the finite sum in
\eqref{A} can be estimated simultaneously.

For \(\sigma\in\{\alpha,1\}\), set
\[
\mathcal H_\sigma(z,\xi)
:=
\frac{1}{|\xi-z|^{2n-1}}
\exp\left\{
-c_0
\left(\frac{|z-\xi|}{\rho(\xi)}\right)^{2\sigma}
\right\}.
\]
Combining \eqref{A}, \eqref{AA}, and \eqref{B(z)^c}, we obtain
\begin{align}\label{A1}
|A_\psi(\omega)(z)|e^{-\psi(z)}
&\lesssim
\int_{\CC}
\frac{|\omega(\xi)|e^{-\psi(\xi)}}{|\xi-z|^{2n-1}}
\sum_{j=0}^{n-1}t^{2j}
\left(e^{-ct^2}+e^{-ct^{2\alpha}}\right)dV(\xi)\\
&\lesssim
\sum_{\sigma\in\{\alpha,1\}}
\int_{\CC}
|\omega(\xi)|e^{-\psi(\xi)}
\mathcal H_\sigma(z,\xi)\,dV(\xi).
\end{align}
We next verify, for both \(\sigma=\alpha\) and \(\sigma=1\),
\begin{equation}
\sup_{\xi\in\CC}
\frac{1}{\rho(\xi)}
\int_{\CC}\mathcal H_\sigma(z,\xi)\,dV(z)
<\infty,
\qquad
\sup_{z\in\CC}
\int_{\CC}
\frac{\mathcal H_\sigma(z,\xi)}{\rho(\xi)}\,dV(\xi)
<\infty.
\end{equation}
For the first integral, the substitution
\(u=(z-\xi)/\rho(\xi)\) gives
\begin{align}
\frac{1}{\rho(\xi)}
\int_{\CC}\mathcal H_\sigma(z,\xi)\,dV(z)
=
\int_{\CC}
\frac{e^{-c_0|u|^{2\sigma}}}{|u|^{2n-1}}\,dV(u)
<\infty.
\end{align}
For the second integral, write
\begin{align}
\int_{\CC}
\frac{\mathcal H_\sigma(z,\xi)}{\rho(\xi)}\,dV(\xi)
=
\left[
\int_{B(z)}
+
\int_{\CC\setminus B(z)}
\right]
\frac{\mathcal H_\sigma(z,\xi)}{\rho(\xi)}\,dV(\xi).
\end{align}
If \(\xi\in B(z)\), then
\(\rho(\xi)\simeq\rho(z)\) by
Lemma~\ref{rhoB}\textup{(A)}, and hence
\[
\int_{B(z)}
\frac{\mathcal H_\sigma(z,\xi)}{\rho(\xi)}\,dV(\xi)
\lesssim
\frac{1}{\rho(z)}
\int_{|\xi-z|<\rho(z)}
\frac{dV(\xi)}{|\xi-z|^{2n-1}}
\lesssim1.
\]
Now suppose that \(\xi\notin B(z)\), and set
\(s=|z-\xi|/\rho(z)\geq1\). Applying
Lemma~\ref{dzw} with the pairs \((z,\xi)\) and
\((\xi,z)\), we obtain
\[
t\gtrsim s^{\alpha/\beta},
\qquad
t\lesssim s^{\beta/\alpha},
\qquad
\frac{\rho(z)}{\rho(\xi)}
=\frac{t}{s}
\lesssim s^{\beta/\alpha-1}.
\]
Therefore,
\begin{align}
\int_{\CC\setminus B(z)}
\frac{\mathcal H_\sigma(z,\xi)}{\rho(\xi)}\,dV(\xi)
&\lesssim
\frac{1}{\rho(z)^{\beta/\alpha}}
\int_{\CC\setminus B(z)}
\frac{
\exp\left\{
-c_1
\left(\frac{|z-\xi|}{\rho(z)}\right)^{2\sigma\alpha/\beta}
\right\}
}{
|z-\xi|^{2n-\beta/\alpha}
}\,dV(\xi)\\
&=
\int_{\CC\setminus B(0,1)}
\frac{
e^{-c_1|u|^{2\sigma\alpha/\beta}}
}{
|u|^{2n-\beta/\alpha}
}\,dV(u)
<\infty.
\end{align}
This proves both integral estimates for each
\(\sigma\in\{\alpha,1\}\).

Let \(v(\xi)=\rho(\xi)|\omega(\xi)|e^{-\psi(\xi)}\).
Fubini's theorem and the first integral estimate give
\[
\|A_\psi(\omega)\|_{L_\psi^1}
\lesssim
\sum_{\sigma\in\{\alpha,1\}}
\int_{\CC}
v(\xi)
\left[
\frac{1}{\rho(\xi)}
\int_{\CC}\mathcal H_\sigma(z,\xi)\,dV(z)
\right]dV(\xi)
\lesssim
\|\rho\omega\|_{L_\psi^1}.
\]
Similarly, the second integral estimate gives
\[
\|A_\psi(\omega)\|_{L_\psi^\infty}
\lesssim
\|\rho\omega\|_{L_\psi^\infty}.
\]
Interpolation therefore yields
\(\|A_\psi(\omega)\|_{L_\psi^p}
\lesssim\|\rho\omega\|_{L_\psi^p}\) for
\(1\leq p\leq\infty\). Since
\(\|\varphi-\widetilde\varphi\|_{L^\infty}<\infty\), we have
\[
\|T_\varphi(\omega)\|_{L_\varphi^p}
\simeq
\|A_\psi(e^{-F}\omega)\|_{L_\psi^p}
\lesssim
\|e^{-F}\rho\omega\|_{L_\psi^p}
\simeq
\|\rho\omega\|_{L_\varphi^p}.
\]
This proves \textup{(A)}.

We now prove \textup{(B)}. Since \(g\in\Gamma\), it suffices by
linearity to consider \(g=K_{z_0}^{\varphi}\). Put
\(h=\rho^\kappa|\bar\partial f|\), and let \(p'\) be the conjugate
exponent of \(p\). Since
\(e^{-\operatorname{Re}F}e^{-\psi}
=e^{-\widetilde\varphi}\simeq e^{-\varphi}\), the curvature
estimate for \(\psi\) gives
\begin{align*}
I_{z,z_0}
&:=
\int_{\CC}
|K_{z_0}^{\varphi}(\xi)\bar\partial f(\xi)e^{-F(\xi)}|
e^{-\psi(\xi)}
\sum_{j=0}^{n-1}
\frac{|\bar\partial\partial\psi(\xi)|^j}
{j!\,|\xi-z|^{2n-2j-1}}\,dV(\xi)\\
&\lesssim
\int_{\CC}
|K_{z_0}^{\varphi}(\xi)\bar\partial f(\xi)|
e^{-\varphi(\xi)}
\sum_{j=0}^{n-1}
\frac{\rho(\xi)^{-2j}}
{|\xi-z|^{2n-2j-1}}\,dV(\xi)\\
&\lesssim
\left[
\int_{B(z)}+\int_{\CC\setminus B(z)}
\right]
|K_{z_0}^{\varphi}(\xi)\bar\partial f(\xi)|
e^{-\varphi(\xi)}
\left[
\frac{1}{|\xi-z|^{2n-1}}
+
\frac{\rho(\xi)^{-2n+2}}{|\xi-z|}
\right]dV(\xi)=: I_1+I_2.
\end{align*}

By Lemma~\ref{rhoB}\textup{(A)},
\begin{align*}
I_1
&\lesssim
\sup_{\xi\in B(z)}
|K_{z_0}^{\varphi}(\xi)\rho(\xi)\bar\partial f(\xi)|
e^{-\varphi(\xi)}
\int_{B(z)}
\left[
\frac{1}{\rho(\xi)|\xi-z|^{2n-1}}
+
\frac{1}{\rho(\xi)^{2n-1}|\xi-z|}
\right]dV(\xi)\\
&\lesssim
\sup_{\xi\in B(z)}
|K_{z_0}^{\varphi}(\xi)\rho(\xi)\bar\partial f(\xi)|
e^{-\varphi(\xi)}.
\end{align*}
Theorem~\ref{main} and the local comparability of \(\rho\) then
imply
\[
I_1
\lesssim
\frac{e^{\varphi(z_0)}}
{\rho(z_0)^n\rho(z)^{n+\kappa-1}}
\sup_{\xi\in B(z)}
\rho(\xi)^\kappa|\bar\partial f(\xi)|
<\infty.
\]

For \(\xi\notin B(z)\), we have
\(|\xi-z|\geq\rho(z)\). Hence
\begin{align*}
I_2
&\lesssim
\frac{1}{\rho(z)}
\int_{\CC\setminus B(z)}
|K_{z_0}^{\varphi}(\xi)\bar\partial f(\xi)|
e^{-\varphi(\xi)}
\left[
\rho(z)^{-2n+2}
+
\rho(\xi)^{-2n+2}
\right]dV(\xi).
\end{align*}
By Theorem~\ref{main} and the triangle inequality for \(d_\rho\),
\[
|K_{z_0}^{\varphi}(\xi)|e^{-\varphi(\xi)}
\lesssim
C(z,z_0)\rho(\xi)^{-n}
e^{-\varepsilon d_\rho(z,\xi)},
\]
where
\(C(z,z_0)
=Ce^{\varphi(z_0)}
\rho(z_0)^{-n}e^{\varepsilon d_\rho(z,z_0)}\).
Therefore,
\begin{align*}
I_2
&\lesssim
\frac{C(z,z_0)}{\rho(z)}
\int_{\CC}
h(\xi)e^{-\varepsilon d_\rho(z,\xi)}
\left[
\rho(z)^{-2n+2}\rho(\xi)^{-n-\kappa}
+
\rho(\xi)^{-3n+2-\kappa}
\right]dV(\xi).
\end{align*}
If $1<p\leq\infty$, H\"older's inequality and
Lemma~\ref{z-wk2n}\textup{(A)} give
\begin{align*}
I_2
\lesssim
\frac{C(z,z_0)}{\rho(z)}
\|h\|_{L^p}
\left[
\rho(z)^{-2n+2}
\rho(z)^{\frac{2n}{p'}-n-\kappa}
+
\rho(z)^{\frac{2n}{p'}-3n+2-\kappa}
\right]
\lesssim
\frac{C(z,z_0)}
{\rho(z)^{n+\kappa-1+\frac{2n}{p}}}
\|\rho^\kappa\bar\partial f\|_{L^p}
<\infty.
\end{align*}
When \(p=1\), Lemma~\ref{dzw} gives the corresponding supremum estimate, and hence
\[
I_2\lesssim
\frac{C(z,z_0)}{\rho(z)}\|h\|_{L^1}
\left[
\rho(z)^{-2n+2}\rho(z)^{-n-\kappa}
+\rho(z)^{-3n+2-\kappa}
\right]
\lesssim
\frac{C(z,z_0)}{\rho(z)^{3n+\kappa-1}}
\|\rho^\kappa\bar\partial f\|_{L^1}<\infty.
\]

Combining the estimates for \(I_1\) and \(I_2\), and then using
linearity, we obtain for every \(g\in\Gamma\) and \(z\in\CC\)
\begin{align}\label{gf111}
\int_{\CC}
|g(\xi)\bar\partial f(\xi)e^{-F(\xi)}|e^{-\psi(\xi)}
\sum_{j=0}^{n-1}
\frac{|\bar\partial\partial\psi(\xi)|^j}
{j!\,|\xi-z|^{2n-2j-1}}\,dV(\xi)
<\infty.
\end{align}
Since \(e^{-F}g\) is holomorphic,
\(e^{-F}g\bar\partial f\) is \(\bar\partial\) closed. Thus
\eqref{gf111} and \cite[Proposition~10]{BA82} imply
\[
\bar\partial A_\psi
\bigl(e^{-F}g\bar\partial f\bigr)
=
e^{-F}g\bar\partial f.
\]
Since \(e^F\) is holomorphic, we conclude that
\[
\bar\partial T_\varphi
\bigl(g\bar\partial f\bigr)
=
e^F\bar\partial A_\psi
\bigl(e^{-F}g\bar\partial f\bigr)
=
g\bar\partial f.
\]
This proves \textup{(B)}.
\end{proof}

For $1\leq p<\infty$, define the symbol class
\begin{align}\label{Qp}
\mathcal{Q}_p:=\left\{ f \text{ measurable on } \CC : fg \in  L^p_\varphi \text{ for } g \in \Gamma \right\}.
\end{align}
For a fixed $1\leq p<\infty$ and $f\in\mathcal Q_p$, the formula
$H_f=(\mathrm I-P_\varphi)M_f$ defines $H_f$ from $\Gamma$ into
$L_\varphi^p$. Since $\Gamma$ is dense in $F_\varphi^p$ by
Corollary~\ref{dense}, this gives a densely defined operator on
$F_\varphi^p$.

The following theorem connects the preceding $\bar\partial$ estimates with the analysis of Hankel operators. 

\begin{thm}\label{HfNf1} Let $1\leq p <q\leq \infty$, $\kappa\in\mathbb R$, and $\varphi\in\mathcal W_{\mathrm{gc}}^*(\CC)$. Suppose that $f \in \mathcal{Q}_p\cap \mathcal{C}^2(\CC)$ with $|\rho^\kappa\bar{\partial} f| \in L^q$. Then, for every $g\in \Gamma$,
\begin{align*}
H_f(g) = T_{\varphi}(g\bar{\partial} f) - P_\varphi(T_{\varphi}( g\bar{\partial} f)),
\end{align*}
where $P_\varphi$ denotes the Bergman projection introduced in Section \ref{sec:proj}.
\end{thm}
\begin{proof}
 Let $u:=T_{\varphi}(g\bar{\partial} f)$. Lemma~\ref{solve1}\textup{(B)}, applied with exponent $q$, gives
$\bar\partial u=g\bar\partial f$. Set $s$ by $1/p=1/q+1/s$, namely,
\[
s=\begin{cases}
\dfrac{pq}{q-p},&q<\infty,\\[4pt]
p,&q=\infty.
\end{cases}
\]
Lemma~\ref{solve1}\textup{(A)} with exponent $p$, followed by H\"older's inequality, gives
\[
\|u\|_{L^p_\varphi}
\lesssim\|\rho g\bar\partial f\|_{L^p_\varphi}
\leq \|\rho^\kappa\bar\partial f\|_{L^q}
\|g\rho^{1-\kappa}\|_{L^s_\varphi}.
\]
Proposition~\ref{Gue} and Lemma~\ref{z-wk2n}\textup{(A)} imply that
$g\rho^{1-\kappa}\in L^s_\varphi$ for every $g\in\Gamma$. Thus
$u\in L^p_\varphi$. Combining this with $f\in \mathcal{Q}_p$, we deduce that  $fg - u \in L^p_\varphi$ and 
$$
\bar\partial(fg-u)=g\bar\partial f-\bar\partial u=0.
$$
This means that $fg - u \in F^p_\varphi$. By Theorem \ref{Pf=fp}, we have
$$
P_\varphi(fg - u) = fg - u.
$$
Therefore,
$$
H_f(g) - (u - P_\varphi(u)) = fg - P_\varphi(fg) - (u - P_\varphi(u)) = (fg - u) - P_\varphi(fg - u) = 0,
$$
and the proof is complete.
\end{proof}

\begin{lem}\label{partition}
Given $r>0$, there exist a sequence $\{z_j\}_{j=1}^\infty$ and a smooth partition
of unity $\{\psi_j\}_{j=1}^\infty$ with the following properties.\\
\textup{(A)} $\CC=\bigcup_{j=1}^\infty B^r(z_j)$, and there is an $m>0$
such that $B^{mr}(z_j)\cap B^{mr}(z_k)=\varnothing$ whenever $j\ne k$.\\
\textup{(B)} $\operatorname{supp}\psi_j\subseteq B^{2r}(z_j)$,
$0\leq\psi_j\leq1$, $\sum_{j=1}^\infty\psi_j=1$, and
$|\bar\partial\psi_j(w)|\lesssim\rho(w)^{-1}$ for $w\in\CC$.
\end{lem}

\begin{proof}
Let $\{z_j\}_{j=1}^\infty$ be an $r$ lattice as in Lemma~\ref{cover}, and set
\[
\chi_{z_k}=\eta\left(\frac{\cdot-z_k}{2r\rho(z_k)}\right),
\qquad
\psi_j=\frac{\chi_{z_j}}{\sum_{k=1}^\infty\chi_{z_k}},
\]
where $\eta$ is as in \eqref{eta}. Then
$0\leq\chi_{z_k}\leq1$, $\operatorname{supp}\chi_{z_k}\subseteq
B^{2r}(z_k)$, and $\sum_k\chi_{z_k}\geq1$. Thus
$\{\psi_j\}$ is a partition of unity with the stated support properties.
For $1\leq i\leq n$, Lemma~\ref{cover}\textup{(C)} and the local
comparability of $\rho$ give
\begin{align}\label{3e15}
\left|\frac{\partial\psi_j}{\partial w_i}(w)\right|
\leq 2\sum_{\{k:w\in B^{2r}(z_k)\}}
\left|\frac{\partial\chi_{z_k}}{\partial w_i}(w)\right|
\lesssim \rho(w)^{-1}
\sum_{k=1}^\infty\chi_{B^{2r}(z_k)}(w)
\lesssim N(r,2)\rho(w)^{-1}.
\end{align}
Similarly,
\begin{align}\label{3e16}
\left|\frac{\partial\psi_j}{\partial\overline w_i}(w)\right|
\lesssim\rho(w)^{-1}.
\end{align}
This proves the lemma.
\end{proof}
 To this end, we introduce the associated IDA spaces and  establish a suitable
decomposition theorem. Let $q \ge 1$ and $r > 0$. For $f \in L_{\mathrm{loc}}^q(\CC)$, define the function $G_{q,r}(f)$ to be
$$
G_{q,r}(f)(z) = \inf \left\{ \left( \frac{1}{|B^r(z)|} \int_{B^r(z)} |f - h|^q \, dV \right)^{1/q} : h \in H(B^r(z)) \right\}, \quad z \in \CC.
$$
 For $0 < s \le \infty$, $\alpha \in \mathbb{R}$, the space $\mathrm{IDA}_r^{s,q,\alpha}$ consists of all $f \in L^q_{\mathrm{loc}}(\CC)$ such that
$$
\|f\|_{\mathrm{IDA}_r^{s,q,\alpha}} = \|\rho^\alpha G_{q,r}(f)\|_{L^s} < \infty.
$$
The space $\mathrm{IDA}_r^{\infty,q,\alpha}$ is also denoted by $\mathrm{BDA}_r^{q,\alpha}$. The space $\mathrm{VDA}_r^{q,\alpha}$ consists of all $f \in \mathrm{BDA}_r^{q,\alpha}$ such that
$$
\lim_{z \to \infty} \rho(z)^\alpha G_{q,r}(f)(z) = 0.
$$
Note that the space $\mathrm{IDA}_r^{s,q,\alpha}$ (integral distance to holomorphic functions) is closely related to the space of bounded mean oscillation, which has played an important role in operator theory and related areas (see \cite{Zhu07}). 

Next, we discuss a decomposition theorem for the spaces $\mathrm{IDA}_r^{s,q,\alpha}$.
Let $0 < q < \infty$. For $z \in \mathbb{C}^n$, $f \in L^q_{\mathrm{loc}}(\CC)$ and $r > 0$, we define the $q$-th mean of $|f|$ over $B^r(z)$ by setting
$$
M_{q,r}(f)(z) = \left( \frac{1}{|B^r(z)|} \int_{B^r(z)} |f|^q \, dV \right)^{1/q}.
$$
For $\omega \in L_{0,1}$, we set $M_{q,r}(\omega)(z) = M_{q,r}(|\omega|)(z)$.
Fix $r>0$ and a sufficiently large constant $R_0>1$. Choose
$\delta>0$ so small that
\[
z\in B^\delta(z_j)
\quad\Longrightarrow\quad
B^{R_0\delta}(z_j)\subset B^r(z).
\]
This is possible by Lemma~\ref{rhoB}\textup{(A)}. By
Lemma~\ref{partition}, choose a partition of unity
$\{\psi_j\}_{j=1}^\infty$ subordinate to the $\delta/2$-lattice
$\{B^{\delta/2}(z_j)\}_{j=1}^\infty$ such that
\[
\operatorname{supp}\psi_j\subset B^\delta(z_j),\qquad
0\leq\psi_j\leq1,\qquad
\sum_{j=1}^\infty\psi_j=1,\qquad
|\bar\partial\psi_j(w)|\lesssim\rho(w)^{-1}.
\]
By the choice of the infimum, we may pick
$h_j\in H(B^{R_0\delta}(z_j))$ so that
\begin{align}\label{MhG}
M_{q,R_0\delta}(f-h_j)(z_j)
\leq2G_{q,R_0\delta}(f)(z_j).
\end{align}
 Moreover, we decompose $f=f_1+f_2$ with
\begin{align}\label{f_1f_2}
f_1 = \sum_{j=1}^\infty h_j \psi_j \quad \text{and} \quad f_2 = f - f_1.
\end{align}
If $z\in\operatorname{supp}\psi_j$, then
$\rho(z)\simeq\rho(z_j)$ and
$B^{R_0\delta}(z_j)\subset B^r(z)$. Hence ball inclusion and volume
comparison give
\begin{align*}
G_{q,R_0\delta}(f)(z_j)
\leq C\inf_{h\in H(B^r(z))}
\left(\frac{1}{|B^r(z)|}
\int_{B^r(z)}|f-h|^q\,dV\right)^{1/q}
=C G_{q,r}(f)(z).
\end{align*}
If $\operatorname{supp}\psi_j\cap
\operatorname{supp}\psi_k\ne\varnothing$, Lemmas~\ref{rhoB}\textup{(A)}
and \ref{cover}\textup{(C)} give
$\rho(z_j)\simeq\rho(z_k)$ and place the corresponding balls in a
fixed enlargement of either one. Thus \eqref{MhG} and
the preceding estimate control $h_j-h_k$ on every overlap by
$G_{q,r}(f)(z)$. Applying \eqref{3e16} and the argument in
\cite[Lemma~3.6]{HV23}, after rescaling by $\rho(z_j)$, gives constants
$t,C>0$ such that
\begin{align}\label{local-decompo}
\rho(z)|\bar\partial f_1(z)|
+M_{q,t}(\rho\bar\partial f_1)(z)
+M_{q,t}(f_2)(z)
\leq C G_{q,r}(f)(z),
\qquad z\in\CC.
\end{align}
For the corresponding results on doubling Fock spaces in one complex
variable, see \cite[Theorem~3.6]{LW24} and
\cite[Theorem~1.1]{AHV24}.

\begin{thm}\label{decompo} 
Suppose $1 \le q < \infty$, $0 < s \le \infty$, and $\alpha \in \mathbb{R}$. A function $f \in L^q_{\mathrm{loc}}(\CC)$ belongs to $\mathrm{IDA}_r^{s,q,\alpha}$ if and only if $f$ admits a decomposition $f = f_1 + f_2$ such that $f_1 \in \mathcal{C}^2(\CC)$ and
\begin{align}\label{ff12ff12}
\rho^{1+\alpha} \left| \bar{\partial} f_1 \right| \in L^s, \, \rho^{\alpha} \, M_{q,r}(\rho\bar{\partial} f_1) + \rho^\alpha \, M_{q,r}(f_2) \in L^s 
\end{align}
for some $r > 0$. Furthermore,
$$
\|f\|_{\mathrm{IDA}_r^{s,q,\alpha}} \simeq \inf \left\{\left\|\rho^{1+\alpha} | \bar{\partial} f_1|\right\|_{L^s}+\left\| \rho^{\alpha} \, M_{q,r}(\rho\bar{\partial} f_1) \right\|_{L^s} + \left\| \rho^\alpha \, M_{q,r}(f_2) \right\|_{L^s} \right\},
$$
where the infimum is taken over all decompositions $f = f_1 + f_2$ that satisfy \eqref{ff12ff12}.
\end{thm}

\begin{proof}
Lemmas~\ref{rhoB} and \ref{partition} give, for any \(a,b>0\),
\begin{align}\label{mean-radius}
\|\rho^\alpha M_{q,a}(u)\|_{L^s}
\simeq
\|\rho^\alpha M_{q,b}(u)\|_{L^s}.
\end{align}
If \(f\in\mathrm{IDA}_r^{s,q,\alpha}\), multiplying
\eqref{local-decompo} by \(\rho^\alpha\), taking \(L^s\) norms, and
using \eqref{mean-radius}, we obtain
\[
\|\rho^{1+\alpha}|\bar\partial f_1|\|_{L^s}
+\|\rho^\alpha M_{q,r}(\rho\bar\partial f_1)\|_{L^s}
+\|\rho^\alpha M_{q,r}(f_2)\|_{L^s}
\lesssim
\|f\|_{\mathrm{IDA}_r^{s,q,\alpha}}.
\]

Conversely, for each $z\in\CC$, set
$F_z(\zeta)=f_1(z+\rho(z)\zeta)$. By the local
$\bar\partial$ estimate \cite[(3.22) and (3.25)]{HV23},
for some $R>1$ we have
\begin{align*}
G_{q,r}(f_1)(z)
&=\inf_{h\in H(B(0,r))}
\left(\frac{1}{|B(0,r)|}\int_{B(0,r)}|F_z-h|^q\,dV\right)^{1/q}\\
&\lesssim
\left(\frac{1}{|B(0,Rr)|}\int_{B(0,Rr)}
|\bar\partial_\zeta F_z(\zeta)|^q\,dV(\zeta)\right)^{1/q}\\
&=\left(\frac{1}{|B^{Rr}(z)|}\int_{B^{Rr}(z)}
|\rho(z)\bar\partial f_1(w)|^q\,dV(w)\right)^{1/q}\\
&\lesssim M_{q,Rr}(\rho\bar\partial f_1)(z).
\end{align*}
The last inequality follows from
$\rho(w)\simeq\rho(z)$ on $B^{Rr}(z)$ by
Lemma~\ref{rhoB}\textup{(A)}.
Hence
\[
G_{q,r}(f)
\leq
G_{q,r}(f_1)+M_{q,r}(f_2)
\lesssim
M_{q,Rr}(\rho\bar\partial f_1)+M_{q,r}(f_2).
\]
Taking \(L^s\) norms and applying \eqref{mean-radius} proves the
reverse estimate. The norm equivalence and the independence of the
radius follow at once.
\end{proof}

We shall also need the corresponding characterization of Carleson measures. Before proving the theorem, we establish several key auxiliary results. Write \( k^\varphi_{z,p}:= K^\varphi_z/\|K^\varphi_z\|_{L^p_\varphi} \) to denote the normalized Bergman kernels in \( F^p_\varphi \).

\begin{lem}\label{kz0} 
For every $0<p\leq\infty$, the set \(\{k^\varphi_{z,p} : z \in \CC\}\) is bounded in \(F^p_\varphi\), and \(k^\varphi_{z,p} \to 0\) uniformly on compact subsets of \(\CC\) as \(|z| \to \infty\). 
\end{lem} 
 
\begin{proof} 
Let $d\mu=(i\partial\bar\partial\widetilde\varphi)^n$. The boundedness assertion follows from the normalization. Applying Lemma \ref{estimate:doubling} to $B(0,\rho(0))$ and $B(0,|z|)$, we have, for $|z|>\rho(0)$, there exists $\delta\in(0,1)$ such that 
$$|z|^\delta\lesssim\mu(B(0,|z|))\lesssim |z|^{1/\delta}.$$ 
Applying Lemma \ref{estimate:doubling} again to $B(0,2|z|)$ and $B(z,\rho(z))$, we have, for $|z|>\rho(z)$, there exists $k\in(0,1)$ such that 
$$\left(\frac{|z|}{\rho(z)}\right)^k\lesssim\mu(B(0,|z|))\lesssim \left(\frac{|z|}{\rho(z)}\right)^{1/k}.$$ 
If $|z|<\rho(z)$, then $0\in B(z)$. By \eqref{rhozw}, we have $\rho(0)\simeq \rho(z)$ and so $|z|<C$ for some $C>0$. This implies that there exists $\sigma>1$ such that for $|z|>\sigma$, 
\begin{align*} 
|z|^{1-\frac{1}{k\delta}} \lesssim \rho(z)\lesssim |z|^{1-k\delta}. 
\end{align*} 
From Theorem \ref{main}, Proposition \ref{prop:normkz} and Lemma \ref{dzw}, we deduce that for $|z|>|w|$ 
$$|k^\varphi_{z,p}(w)|\leq Ce^{\varphi(w)}\rho(w)^{-n}|z|^{cn}e^{-\left(\frac{|z|-|w|}{\rho(w)}\right)^\varepsilon}$$ 
where 
$$c =  
\begin{cases}  
\left(1-\frac{1}{k\delta}\right)\left(1 - \dfrac{2}{p}\right) & \text{if } p < 2, \\[6pt] 
(1-k\delta)\left(1 - \dfrac{2}{p}\right) & \text{if } p \geq 2. 
\end{cases}$$ 
Hence, $k^\varphi_{z,p}\rightarrow 0$ uniformly on any compact subset of $\CC$ as $|z|\rightarrow\infty$. This completes the proof. 
\end{proof} 
 
 \begin{lem}\label{atomic}
Let \(\varphi\in\mathcal W^*(\CC)\) with \(\rho\in\mathcal S\),
and let \(\{z_j\}_{j=1}^{\infty}\) be an \(r\)-lattice.
For \(0<p\leq\infty\) and
\(\lambda=\{\lambda_j\}_{j=1}^{\infty}\in\ell^p\), the series
\[
f(z)=\sum_{j=1}^{\infty}
\lambda_j k_{z_j,p}^{\varphi}(z)
\]
converges locally uniformly in \(\CC\) and defines a function
\(f\in F_\varphi^p\). Moreover,
\[
\|f\|_{F_\varphi^p}
\lesssim
\|\lambda\|_{\ell^p}.
\]
\end{lem}

\begin{proof}
By Lemma~\ref{meaninequality}\textup{(A)} and \eqref{NNN},
\[
\sum_{j=1}^{\infty}
|h(z_j)|^t e^{-t\varphi(z_j)}
\rho(z_j)^{2n}
\lesssim
\|h\|_{F_\varphi^t}^t
\]
for every \(0<t<\infty\) and \(h\in F_\varphi^t\).

Set
\[
S_N=\sum_{j=1}^{N}\lambda_jk_{z_j,p}^{\varphi}.
\]
If \(0<p\leq1\), then
\[
\|S_N-S_M\|_{F_\varphi^p}^p
\leq
\sum_{j=M+1}^{N}|\lambda_j|^p.
\]
Hence \(S_N\) converges in \(F_\varphi^p\), and
$\|f\|_{F_\varphi^p}\leq\|\lambda\|_{\ell^p}.$
Lemma~\ref{meaninequality}\textup{(A)} also gives uniform
convergence on compact subsets.

Now suppose that \(1<p\leq\infty\), and let \(p'\) be the
conjugate exponent of \(p\), where \(p'=1\) if \(p=\infty\).
For \(g\in F_\varphi^{p'}\), Proposition~\ref{prop:normkz}
and H\"older's inequality give
\begin{align*}
\bigl|\langle S_N,g\rangle_\varphi\bigr|
\leq
\sum_{j=1}^{N}
\frac{|\lambda_j|\,|g(z_j)|}
{\|K_{z_j}^{\varphi}\|_{L_\varphi^p}}
\lesssim
\|\{\lambda_j\}_{j=1}^{N}\|_{\ell^p}
\left(
\sum_{j=1}^{N}
|g(z_j)|^{p'}e^{-p'\varphi(z_j)}
\rho(z_j)^{2n}
\right)^{1/p'}
\lesssim
\|\{\lambda_j\}_{j=1}^{N}\|_{\ell^p}
\|g\|_{F_\varphi^{p'}}.
\end{align*}
Theorem~\ref{dual}, applied with exponent \(p'\), therefore yields
\[
\|S_N\|_{F_\varphi^p}
\lesssim
\|\{\lambda_j\}_{j=1}^{N}\|_{\ell^p}.
\]

If \(1<p<\infty\), the same estimate applied to \(S_N-S_M\)
shows that \(S_N\) converges in \(F_\varphi^p\). The desired norm
estimate and uniform convergence on compact subsets follow from
Theorem~\ref{dual} and Lemma~\ref{meaninequality}\textup{(A)}.

Finally, let \(p=\infty\). Proposition~\ref{prop:normkz} and the
preceding sampling estimate with \(t=1\) imply that
\[
\sum_{j=1}^{\infty}
|\lambda_jk_{z_j,\infty}^{\varphi}(z)|
\lesssim
\|\lambda\|_{\ell^\infty}
\sum_{j=1}^{\infty}
|K_z^\varphi(z_j)|e^{-\varphi(z_j)}
\rho(z_j)^{2n}
\lesssim
\|\lambda\|_{\ell^\infty}
\|K_z^\varphi\|_{F_\varphi^1}.
\]
Thus the series converges pointwise. The preceding uniform bound
for \(S_N\), together with Montel's theorem, gives uniform
convergence on compact subsets. Taking the limit yields
$
\|f\|_{F_\varphi^\infty}
\lesssim
\|\lambda\|_{\ell^\infty}.
$
This completes the proof.
\end{proof}

We first recall the definition of Fock Carleson measures.
\begin{defn}
Suppose \( \mu \) is a positive Borel measure on \( \CC \) and \( 0 < p, q < \infty \). If the embedding \( \mathrm{Id} : F_{\varphi}^p \to L^q(\CC, e^{-q\varphi} \, d\mu) \) is continuous (or compact) then \( \mu \) is said to be a \( q \)-Carleson measure (or a vanishing \( q \)-Carleson measure) for \( F_{\varphi}^p \).
\end{defn}
For every $r>0$, the  $r$-averaging transform of $\mu$ is defined by
 $$
 \widehat{\mu}_r (z) := \frac{\mu(B^r(z))}{ |B^r(z)|}\simeq \frac{\mu(B^r(z))}{\rho(z)^{2n}} \qquad z\in \CC.
 $$

The following proposition follows from the arguments in
\cite[Theorems~3.1, 3.2, and~3.3]{HL11}. See also
\cite{WTH,AC19}. We record the estimates involving the local scale
\(\rho\).
\begin{prop}\label{car}
Let $\varphi\in \mathcal{W}^*(\CC)$ with $\rho\in \mathcal{S}$ and let \(\mu\) be a positive Borel measure on \(\CC\).\\
\textup{(A)} For \(1\leq p \leq q < \infty\), \(\mu\) is a \(q\)-Carleson measure for \(F_{\varphi}^p\) if and only if
\[
\sup_{z \in \CC} \widehat{\mu}_r(z) \, \rho(z)^{2n(1-\frac{q}{p})} < \infty
\]
for some (or any) \(r >0\). In addition, \(\mu\) is a vanishing \(q\)-Carleson measure for \(F_{\varphi}^p\) if and only if
\[
\lim_{|z| \to \infty} \widehat{\mu}_r(z) \, \rho(z)^{2n(1-\frac{q}{p})} = 0
\]
for some (or any) \(r>0\). Moreover, 
\[
\| \mathrm{Id} \|_{F_{\varphi}^p \to L^q(\CC, e^{-q\varphi} d\mu)} \simeq 
\bigl\| (\widehat{\mu}_r)^{\frac{1}{q}} \rho(z)^{2n(\frac{1}{q}-\frac{1}{p})} \bigr\|_{L^\infty}.
\]
\textup{(B)} For \(1\leq q < p < \infty\), \(\mu\) is a \(q\)-Carleson measure for \(F_{\varphi}^p\) if and only if \(\mu\) is a vanishing \(q\)-Carleson measure for \(F_{\varphi}^p\) if and only if
\[
\widehat{\mu}_r \in L^{\frac{p}{p-q}}
\]
for some (or any) \(r>0\). Moreover,
\[
\| \mathrm{Id} \|_{F_{\varphi}^p \to L^q(\CC, e^{-q\varphi} d\mu)} \simeq 
\bigl\| (\widehat{\mu}_r)^{\frac{1}{q}} \bigr\|_{L^{\frac{pq}{p-q}}}.
\]
\end{prop}

\begin{proof}
By Lemma~\ref{rhoB}\textup{(A)} and \eqref{NNN}, changing \(r\)
only changes the quantities below by a constant. We therefore fix
\(r>0\) sufficiently small. Lemma~\ref{meaninequality}\textup{(A)}
and Fubini's theorem give
\begin{align}\label{car-mean}
\int_{\CC}|g|^q e^{-q\varphi}\,d\mu
\lesssim
\int_{\CC}|g(z)|^q e^{-q\varphi(z)}
\widehat{\mu}_r(z)\,dV(z).
\end{align}

Suppose \(p\leq q\) and set
\[
A_r(z)=\widehat{\mu}_r(z)^{1/q}
\rho(z)^{2n(\frac1q-\frac1p)}.
\]
Theorem~\ref{main} and Proposition~\ref{prop:normkz} yield
\[
A_r(z)^q
\lesssim
\int_{B^r(z)}
|k_{z,p}^\varphi(w)|^q e^{-q\varphi(w)}\,d\mu(w)
\leq
\|\mathrm{Id}(k_{z,p}^\varphi)\|_{L^q(\mu)}^q.
\]
Moreover, Lemma~\ref{meaninequality}\textup{(A)} and \eqref{NNN}
give
\begin{align}\label{FpFq-weight}
\int_{\CC}|g|^q e^{-q\varphi}
\rho^{2n(\frac qp-1)}\,dV
\lesssim
\|g\|_{F_\varphi^p}^q.
\end{align}
Consequently, by \eqref{car-mean} and \eqref{FpFq-weight},
\[
\|\mathrm{Id}(g)\|_{L^q(\mu)}^q
\lesssim
\|A_r\|_{L^\infty}^q\|g\|_{F_\varphi^p}^q.
\]
Suppose $p=1$ and the embedding $\mathrm{Id}$ is compact. If
$|z_j|\to\infty$, then a subsequence of
$\mathrm{Id}(k_{z_j,1}^\varphi)$ converges in $L^q(\mu)$. By
Lemma~\ref{kz0}, it converges pointwise to zero. Hence
a further subsequence converges almost everywhere to its norm
limit, and this limit is zero. Consequently,
\[
\|\mathrm{Id}(k_{z,1}^\varphi)\|_{L^q(\mu)}\longrightarrow0,
\qquad |z|\longrightarrow\infty.
\]
The preceding lower estimate gives $A_r(z)\to0$. The remaining
compactness implications follow from Lemma~\ref{kz0} and
\cite[Theorem~3.2]{HL11}. Thus the estimates above prove
\textup{(A)}.

Now suppose \(q<p\). By \eqref{car-mean} and H\"older's inequality,
\[
\|\mathrm{Id}(g)\|_{L^q(\mu)}^q
\lesssim
\|g\|_{F_\varphi^p}^q
\|\widehat{\mu}_r\|_{L^{p/(p-q)}}.
\]
Conversely, let \(\{z_j\}\) be an \(r\) lattice. Lemma~\ref{atomic},
Khinchine's inequality, and the lower estimate in
Theorem~\ref{main} give, for every finitely supported
\(\lambda\in\ell^p\),
\[
\sum_j|\lambda_j|^q A_r(z_j)^q
\lesssim
\int_0^1
\left\|
\mathrm{Id}\left(
\sum_j\lambda_j\varepsilon_j(\tau)k_{z_j,p}^\varphi
\right)
\right\|_{L^q(\mu)}^q\,d\tau
\lesssim
\|\mathrm{Id}\|^q\|\lambda\|_{\ell^p}^q.
\]
Setting \(\sigma=pq/(p-q)\), the diagonal operator criterion and
Lemma~\ref{rhoB} imply
\[
\bigl\|(\widehat{\mu}_r)^{1/q}\bigr\|_{L^\sigma}^\sigma
\simeq
\sum_j A_r(z_j)^\sigma
\lesssim
\|\mathrm{Id}\|^\sigma.
\]
The compactness assertion follows by truncating \(\mu\), as in
\cite[Theorem~3.3]{HL11}. The preceding estimates also give the
stated norm equivalences.
\end{proof}

\begin{re}\label{FpFq}
Fock spaces associated with different exponents need not be nested.
Indeed, let \(\beta>2\) and
\[
\varphi_\beta(z)=|z|^\beta+\operatorname{Re}z_1,
\qquad z\in\mathbb C^n.
\]
Then \(\varphi_\beta\in\mathcal W^*(\mathbb C^n)\), and since
\(\operatorname{Re}z_1\) is pluriharmonic, one may take
\[
\rho_\beta(z)\simeq(1+|z|)^{1-\frac{\beta}{2}}
\longrightarrow0
\qquad\text{as }|z|\to\infty.
\]
Fix \(1<p<q<\infty\). For \(d\mu=dV\), we have
\(\widehat\mu_r\equiv1\). Proposition~\ref{car} gives
\[
\sup_{z\in\mathbb C^n}
\rho_\beta(z)^{2n(1-q/p)}=\infty
\quad\text{and}\quad
1\notin L^{\frac{q}{q-p}}(\mathbb C^n).
\]
Hence neither identity map
\[
F_{\varphi_\beta}^p\longrightarrow F_{\varphi_\beta}^q,
\qquad
F_{\varphi_\beta}^q\longrightarrow F_{\varphi_\beta}^p
\]
is bounded. Since convergence in each Fock space implies local
uniform convergence, the closed graph theorem yields
\[
F_{\varphi_\beta}^p\not\subset F_{\varphi_\beta}^q
\qquad\text{and}\qquad
F_{\varphi_\beta}^q\not\subset F_{\varphi_\beta}^p.
\]
\end{re}

Our analysis on the Hankel operators \( H_f \) from \( F_\varphi^p \) to \( L_\varphi^q \) will be carried out in two cases: \( 1 \leq p \leq q < \infty \) and \( 1 \leq q < p < \infty \). We always set \( \gamma = n(1/q - 1/p) \).

\begin{thm}\label{Hfpq1}
Let \( 1 \leq p \leq q < \infty \). Suppose \( f \in\mathcal{Q}_p \) and $\varphi\in \mathcal W_{\mathrm{gc}}^*(\CC)$. Then \\
 \textup{(i)} \( H_f : F_\varphi^p \to L_\varphi^q \) is bounded if and only if \( f \in \mathrm{BDA}_r^{q, 2\gamma} \) for some (or any) $r>0$. Furthermore, 
\[
\|H_f\|_{F_\varphi^p \to L_\varphi^q} \simeq \|f\|_{\mathrm{BDA}_r^{q, 2\gamma }}.
\]
 \textup{(ii)}  \( H_f : F_\varphi^p \to L_\varphi^q \) is compact if and only if \( f \in \mathrm{VDA}_r^{q, 2\gamma} \) for some (or any) $r>0$. 
\end{thm}

\begin{proof}
(i): Assume that \( H_f : F_\varphi^p \to L_\varphi^q \) is bounded. Fix \(r>0\) sufficiently small. It follows from Theorem~\ref{main} and Proposition~\ref{prop:normkz} that
\[
\inf_{w \in B^r(z)} |k^\varphi_{z,p}(w)| e^{-\varphi(w)} \approx \rho(z)^{-2n/p} > 0,
\]
which means that $\frac{1}{k^\varphi_{z,p}} P_\varphi(fk^\varphi_{z,p}) \in H(B^r(z))$. This implies that
\begin{align*}
\|H_f(k^\varphi_{z,p})\|_{L_{\varphi}^q}^q &= \int_{\CC} |fk^\varphi_{z,p}(\xi) - P_\varphi(fk^\varphi_{z,p})(\xi)|^q e^{-q\varphi(\xi)}\, dV(\xi) \\
&\geq \int_{B^r(z)} |k^\varphi_{z,p}(\xi)|^q \left| f(\xi) - \frac{1}{k^\varphi_{z,p}(\xi)} P_\varphi(fk^\varphi_{z,p})(\xi) \right|^q e^{-q\varphi(\xi)}\, dV(\xi) \\
&\geq C\rho(z)^{-2n\frac{q}{p}} \int_{B^r(z)} \left| f(\xi) - \frac{1}{k^\varphi_{z,p}(\xi)} P_\varphi(fk^\varphi_{z,p})(\xi) \right|^q\, dV(\xi) \\
&\geq C\{\rho(z)^{2\gamma} G_{q,r}(f)(z)\}^q.
\end{align*}
Therefore,
$$\|f\|_{\mathrm{BDA}_r^{q, 2\gamma}}=\left\|\rho^{2\gamma} G_{q,r}(f)\right\|_{L^\infty}\leq C\|H_f\|_{F_\varphi^p \to L_\varphi^q}.$$

Conversely, if $f\in \mathrm{BDA}_r^{q, 2\gamma}$, then applying Theorem \ref{decompo} to $s=\infty$, we conclude that $f$ admits a decomposition $f = f_1 + f_2$ such that $f_1 \in \mathcal{C}^2(\CC)$ and
\begin{align}\label{ff12}
\rho^{1+2\gamma}| \bar{\partial} f_1 | \in L^\infty, \, \rho^{2\gamma} \, M_{q,r}(f_2) \in L^\infty 
\end{align}
for some $r > 0$. Set 
\begin{align}\label{munu}
d\mu:=\rho^q| \bar{\partial} f_1 |^q\,dV,\,d\nu:=|f_2|^q\,dV.
\end{align}
By Proposition \ref{car}, we know that both $\mu$ and $\nu$ are $q$-Fock Carleson measures for $F^p_\varphi$ and the corresponding embedding operators satisfy 
\begin{align}\label{IDf1f2}
\| \mathrm{Id} \|_{F_{\varphi}^p \to L^q(\CC, e^{-q\varphi} d\mu)} \lesssim 
\bigl\| \rho^{1+2\gamma}| \bar{\partial} f_1 | \bigr\|_{L^\infty},\, \| \mathrm{Id} \|_{F_{\varphi}^p \to L^q(\CC, e^{-q\varphi} d\nu)} \lesssim \bigl\| \rho^{2\gamma} \, M_{q,r}(f_2) \bigr\|_{L^\infty}.
\end{align}
We claim that both $f_1,f_2\in \mathcal{Q}_p$. Indeed, apply \eqref{car-mean} to $d\lambda=|f_2|^p\,dV$. Since $p\leq q$, for any $g\in \Gamma$
\begin{align}\label{957}
\int_{\CC}|g(w)f_2(w)|^pe^{-p\varphi(w)}\,dV(w)&\lesssim \int_{\CC}|g(w)|^pe^{-p\varphi(w)}M_{1,r}(|f_2|^p)(w)\,dV(w)\notag\\
&\lesssim \int_{\CC}|g(w)|^pe^{-p\varphi(w)}M_{q,r}(f_2)(w)^p\,dV(w)\notag\\
&\lesssim \|f\|_{\mathrm{BDA}_r^{q, 2\gamma}}^p\|g\rho^{-2\gamma}\|_{L^p_\varphi}<\infty.
\end{align}
Here the last quantity is finite by Proposition~\ref{Gue} and
Lemma~\ref{z-wk2n}\textup{(A)}. Hence $f_2\in \mathcal{Q}_p$ and so also $f_1=f-f_2\in \mathcal{Q}_p$.
Finally, for $g\in \Gamma$, let $u:=T_{\varphi}(g\bar{\partial} f_1)$, where $T_{\varphi}$ is as in Lemma \ref{solve1}. Applying Theorem \ref{HfNf1} with its upper exponent equal to $\infty$ and with $\kappa=1+2\gamma$, we obtain 
$$H_{f_1}(g)=u-P_\varphi(u).$$
Since $\Gamma$ is dense in $F_\varphi^p$, this, together with \eqref{IDf1f2}, yields
\begin{align}\label{HF1111}
\|H_{f_1}(g)\|_{L^q_\varphi}\leq (1+\|P_\varphi\|_{L^q_\varphi\rightarrow L^q_\varphi})\|u\|_{L^q_\varphi}\lesssim \|g\rho\bar{\partial} f_1\|_{L^q_\varphi}\lesssim \| \mathrm{Id} \|_{F_{\varphi}^p \to L^q(\CC, e^{-q\varphi} d\mu)} \|g\|_{L^p_\varphi}
\end{align}
and 
\begin{align}\label{HF2222}
 \|H_{f_2}(g)\|_{L^q_\varphi}\leq (1+\|P_\varphi\|_{L^q_\varphi\rightarrow L^q_\varphi})\|gf_2\|_{L^q_\varphi}\lesssim \| \mathrm{Id} \|_{F_{\varphi}^p \to L^q(\CC, e^{-q\varphi} d\nu)}\|g\|_{L^p_\varphi}.
\end{align}
This, together with \eqref{IDf1f2}, yields
$$\|H_f\|_{F^p_\varphi\rightarrow L^q_\varphi}\lesssim \bigl\| \rho^{1+2\gamma}| \bar{\partial} f_1 | \bigr\|_{L^\infty}+\bigl\| \rho^{2\gamma} \, M_{q,r}(f_2) \bigr\|_{L^\infty}.$$

 (ii): Suppose first that \(H_f:F_\varphi^p\to L_\varphi^q\) is compact.
If $1<p<\infty$, then for every $g\in F_\varphi^{p'}$,
the reproducing formula and Proposition~\ref{prop:normkz} give
\[
|\langle k_{z,p}^{\varphi},g\rangle_\varphi|
\lesssim |g(z)|e^{-\varphi(z)}\rho(z)^{2n/p'}.
\]
Then the local mean value estimate implies that
\[
|g(z)|^{p'}e^{-p'\varphi(z)}\rho(z)^{2n}
\lesssim
\int_{B(z)}|g(w)|^{p'}e^{-p'\varphi(w)}\,dV(w)
\longrightarrow0,
\]
where the last limit follows from the estimates in the proof of
Lemma~\ref{kz0}. Those estimates also imply that, for every compact
$E\subset\CC$, $B(z)\cap E=\varnothing$ when $|z|$ is sufficiently
large. Hence $k_{z,p}^{\varphi}$ converges weakly to zero in
$F_\varphi^p$ as $|z|\to\infty$.

Let $p=1$. For any sequence $|z_j|\to\infty$, compactness gives a
subsequence, still denoted by $z_j$, such that
$H_f(k_{z_j,1}^{\varphi})\to u$ in $L_\varphi^q$.
We first verify that $f\in L_{\mathrm{loc}}^1(\CC)$. For each
$a\in\CC$, the definition of $\mathcal Q_1$ gives
$fK_a^\varphi\in L_\varphi^1$. By Theorem~\ref{main}, for some
$r>0$,
\[
|K_a^\varphi(w)|e^{-\varphi(w)}\gtrsim
e^{\varphi(a)}\rho(a)^{-2n},
\qquad w\in B^r(a).
\]
Thus $f\in L^1(B^r(a))$, and hence $f\in L_{\mathrm{loc}}^1(\CC)$.
For every compact set $E\subset\CC$, Lemma~\ref{kz0} now gives
\[
\int_E|f k_{z_j,1}^\varphi|\,dV
\leq \sup_E|k_{z_j,1}^\varphi|
\int_E|f|\,dV\longrightarrow0.
\]
Since
\[
H_f(k_{z_j,1}^{\varphi})
=f k_{z_j,1}^{\varphi}-P_\varphi(f k_{z_j,1}^{\varphi}),
\]
for every compact $E\subset\CC$, \eqref{varphi} and the continuity
of $\widetilde\varphi$ give, for $v\in L_\varphi^q$,
\[
\|v\|_{L^1(E)}\leq C_E\|v\|_{L_\varphi^q}.
\]
Thus convergence in $L_\varphi^q$ implies distributional convergence,
and consequently $\bar\partial u=0$. Moreover,
$P_\varphi H_f(k_{z_j,1}^{\varphi})=0$, so the boundedness of
$P_\varphi$ on $L_\varphi^q$ gives $P_\varphi u=0$. Hence
$u=P_\varphi u=0$ by Theorem~\ref{Pf=fp}. Every convergent subsequence
therefore has limit zero, and compactness implies
\[
\big\|H_f(k_{z,p}^{\varphi})\big\|_{L_\varphi^q}
\rightarrow 0
\qquad\text{as } |z|\rightarrow\infty.
\]
Combining this with the estimate obtained in the necessity part of
\textup{(i)},
\[
\rho(z)^{2\gamma}G_{q,r}(f)(z)
\lesssim
\big\|H_f(k_{z,p}^{\varphi})\big\|_{L_\varphi^q}\rightarrow 0,
\]
as $|z|\rightarrow\infty$. Thus \(f\in\operatorname{VDA}_{r}^{q,2\gamma}\).

Conversely, suppose that
\(f\in\mathrm{VDA}_r^{q,2\gamma}\). Choose the decomposition
\(f=f_1+f_2\) furnished by the construction leading to
\eqref{local-decompo}, and define \(\mu\) and \(\nu\) by
\eqref{munu}. Then \eqref{local-decompo} and
Proposition~\ref{car}\textup{(A)} show that both measures are
vanishing \(q\)-Carleson measures for \(F_\varphi^p\). As in the proof of
\textup{(i)}, $f_1,f_2\in\mathcal Q_p$. The estimates
\eqref{HF1111} and \eqref{HF2222} show that $H_{f_1}$ and $H_{f_2}$
factor through the corresponding compact embeddings. Thus both
operators are compact from $F_\varphi^p$ into $L_\varphi^q$. Hence
$H_f=H_{f_1}+H_{f_2}$
is compact. This completes the proof.
\end{proof}

\begin{thm}\label{Hfqp}
Let \( 1 \leq q < p < \infty \) and set $s=\dfrac{pq}{p-q}$. Suppose \( f \in\mathcal{Q}_q \) and $\varphi\in \mathcal W_{\mathrm{gc}}^*(\CC)$. Then the following statements are equivalent:\\
 \textup{(i)} \( H_f : F_{\varphi}^p \to L_{\varphi}^q \) is bounded;\\
 \textup{(ii)} \( H_f : F_{\varphi}^p \to L_{\varphi}^q \) is compact;\\
  \textup{(iii)} \( f \in \mathrm{IDA}_{r}^{s, q, 0} \) for some (or any) $r>0$.
Furthermore,
\[
\|H_f\|_{F_{\varphi}^p \to L_{\varphi}^q} \simeq\|f\|_{\mathrm{IDA}_{r}^{s, q, 0}}.
\]
\end{thm}

\begin{proof}
The implication \textup{(ii)}$\Rightarrow$\textup{(i)} is immediate.

Assume \textup{(i)}. Choose \(r>0\) small enough that the
near-diagonal lower estimate holds on \(B^r(a)\). By
Lemma~\ref{rhoB}\textup{(B)}, choose \(0<r_0<r\) such that
\[
B^{r_0}(z)\subset B^r(a_k)
\qquad\text{whenever }z\in B^{r_0}(a_k),
\]
and let \(\{a_k\}\) be an \(r_0\)-lattice. For every finitely
supported sequence \(\lambda=\{\lambda_k\}\in\ell^p\), set
\[
g_t(z)=\sum_k\lambda_k\gamma_k(t)
k_{a_k,p}^{\varphi}(z),
\]
where \(\{\gamma_k\}\) is a sequence of Rademacher functions.
Lemma~\ref{atomic} gives
\[
\|g_t\|_{F_\varphi^p}\lesssim\|\lambda\|_{\ell^p}.
\]
Hence
\[
\int_0^1\|H_f(g_t)\|_{L_\varphi^q}^q\,dt
\lesssim
\|H_f\|_{F_\varphi^p\to L_\varphi^q}^q
\|\lambda\|_{\ell^p}^q.
\]
On the other hand, Fubini's theorem, Khinchine's inequality, the
finite multiplicity of \(\{B^r(a_k)\}\), and the local estimate used
in the necessity part of Theorem~\ref{Hfpq1}\textup{(i)} yield
\[
\int_0^1\|H_f(g_t)\|_{L_\varphi^q}^q\,dt
\gtrsim
\sum_k|\lambda_k|^q
\left[
\rho(a_k)^{2n(\frac1q-\frac1p)}
G_{q,r}(f)(a_k)
\right]^q.
\]
The characterization of diagonal operators from \(\ell^p\) into
\(\ell^q\) therefore gives
\[
\sum_k
\rho(a_k)^{2n}G_{q,r}(f)(a_k)^s
\lesssim
\|H_f\|_{F_\varphi^p\to L_\varphi^q}^s,
\]
where we used
\[
2ns\left(\frac1q-\frac1p\right)=2n.
\]
Moreover, for \(z\in B^{r_0}(a_k)\),
\[
G_{q,r_0}(f)(z)\lesssim G_{q,r}(f)(a_k).
\]
The covering and finite-overlap properties of the lattice now imply
\[
\|G_{q,r_0}(f)\|_{L^s}^s
\lesssim
\sum_k\rho(a_k)^{2n}G_{q,r}(f)(a_k)^s
\lesssim
\|H_f\|_{F_\varphi^p\to L_\varphi^q}^s.
\]
Thus \(f\in\mathrm{IDA}_{r_0}^{s,q,0}\), and
\[
\|f\|_{\mathrm{IDA}_{r_0}^{s,q,0}}
\lesssim
\|H_f\|_{F_\varphi^p\to L_\varphi^q}.
\]

Conversely, suppose that
\(f\in\mathrm{IDA}_r^{s,q,0}\). By
Theorem~\ref{decompo}, we may write \(f=f_1+f_2\), where
\(f_1\in\mathcal C^2(\CC)\) and
\[
\|\rho\bar\partial f_1\|_{L^s}
+
\|M_{q,r}(\rho\bar\partial f_1)\|_{L^s}
+
\|M_{q,r}(f_2)\|_{L^s}
\lesssim
\|f\|_{\mathrm{IDA}_r^{s,q,0}}.
\]
Define \(\mu\) and \(\nu\) by \eqref{munu} for this decomposition.
Then
\[
\bigl(\widehat{\mu}_r\bigr)^{1/q}
=
M_{q,r}(\rho\bar\partial f_1),
\qquad
\bigl(\widehat{\nu}_r\bigr)^{1/q}
=
M_{q,r}(f_2).
\]
Since \(s/q=p/(p-q)\), it follows that
$
\widehat{\mu}_r,\widehat{\nu}_r
\in L^{\frac{p}{p-q}}.
$
Proposition~\ref{car}\textup{(B)} shows that \(\mu\) and \(\nu\)
are vanishing \(q\)-Carleson measures for \(F_\varphi^p\).
Hence their corresponding embedding operators are compact.

Applying \eqref{car-mean} with exponent $q$ and
$d\lambda=|f_2|^q\,dV$, and then using H\"older's inequality with
exponents $p/q$ and $p/(p-q)$, we obtain for $g\in\Gamma$
\[
\|gf_2\|_{L_\varphi^q}^q
\lesssim
\int_{\CC}|g|^qe^{-q\varphi}M_{q,r}(f_2)^q\,dV
\lesssim
\|g\|_{L_\varphi^p}^q\|M_{q,r}(f_2)\|_{L^s}^q<\infty.
\]
Thus $f_2\in\mathcal Q_q$, and $f_1=f-f_2\in\mathcal Q_q$.

By Theorem~\ref{HfNf1}, applied with its lower exponent equal to
$q$, its upper exponent equal to $s$, and \(\kappa=1\), and the
estimates leading to \eqref{HF1111} and \eqref{HF2222},
\(H_{f_1}\) and \(H_{f_2}\) factor through these compact
embeddings. Therefore both operators are compact from
\(F_\varphi^p\) into \(L_\varphi^q\), and hence
$H_f=H_{f_1}+H_{f_2}$
is compact. The same estimates and Proposition~\ref{car}\textup{(B)}
also give
\[
\|H_f\|_{F_\varphi^p\to L_\varphi^q}
\lesssim
\|M_{q,r}(\rho\bar\partial f_1)\|_{L^s}
+
\|M_{q,r}(f_2)\|_{L^s}
\lesssim
\|f\|_{\mathrm{IDA}_r^{s,q,0}}.
\]
Combining this with the necessity estimate proves the asserted norm
equivalence and completes the proof.
\end{proof}

\begin{cor}[Berger--Coburn phenomenon]\label{berger-coburn}
Let $\varphi=\varphi_0+\operatorname{Re}F$, where
$\varphi_0\in\mathcal W_0$ and $F\in H(\CC)$. If
$f\in L^\infty(\CC)$ and $1\leq p,q<\infty$, then
\[
H_f:F_\varphi^p\to L_\varphi^q
\text{ is compact}
\quad\Longleftrightarrow\quad
H_{\overline f}:F_\varphi^p\to L_\varphi^q
\text{ is compact}.
\]
\end{cor}

\begin{proof}
For $1\leq r<\infty$, define $U_Fg:=e^{-F}g.$
Since $\varphi=\varphi_0+\operatorname{Re}F$, we have
\[
\|U_Fg\|_{L_{\varphi_0}^r}
=
\|g\|_{L_\varphi^r}.
\]
Thus $U_F$ is an isometry from $L_\varphi^r$ onto
$L_{\varphi_0}^r$ and from $F_\varphi^r$ onto
$F_{\varphi_0}^r$. In particular, $U_F$ is unitary on the
corresponding $L^2$ spaces and maps $F_\varphi^2$ onto
$F_{\varphi_0}^2$. Therefore, by the uniqueness of orthogonal
projections,
\[
U_FP_\varphi=P_{\varphi_0}U_F
\qquad\text{on }L_\varphi^2.
\]
Since $L_\varphi^2\cap L_\varphi^r$ is dense in $L_\varphi^r$ and
both Bergman projections are bounded on the corresponding $L^r$
spaces by Theorem~\ref{thm:projbnd}, the same identity holds on
$L_\varphi^r$. Moreover, $U_FM_h=M_hU_F$. Hence, for $h=f$ or
$h=\overline f$,
\[
U_FH_h^\varphi
=
U_F(I-P_\varphi)M_h
=
(I-P_{\varphi_0})M_hU_F
=
H_h^{\varphi_0}U_F.
\]
Consequently, $H_h^\varphi:F_\varphi^p\to L_\varphi^q$ is compact
if and only if
$H_h^{\varphi_0}:F_{\varphi_0}^p\to L_{\varphi_0}^q$ is compact.
The conclusion now follows from the Berger--Coburn theorem for
$\mathcal W_0$ in \cite{HV23}.
\end{proof}

As an application we discuss the characterizations of the symbols $f$ for which the induced Hankel operators \( H_f \) and \(H_{\overline{f}} \) are simultaneously bounded (or compact) from \( F_{\varphi}^p \) to \( L_{\varphi}^q \). For this purpose we need mean oscillation functions and some related spaces. For \( f \in L_{\mathrm{loc}}^q(\CC) \) and \( r > 0 \), set
\[
\mathrm{MO}_{q,r}(f)(z) = \left( \frac{1}{|B^r(z)|} \int_{B^r(z)} |f - \widehat{f}_r(z)|^q \, dV \right)^{1/q},
\]
where 
$$\widehat f_{r}(z) = \frac{1}{|B^r(z)|} \int_{B^r(z)} f \, dV.$$
For continuous $f$, also set
\[
\mathrm{MO}_{\infty,r}(f)(z) = \sup_{w \in B^r(z)} |f(w) - f(z)|.
\]
Given $1\leq q<\infty$, $1\leq s\leq\infty$, \( \sigma \in \mathbb{R} \) and \( r > 0 \), we define space \( \mathrm{IMO}_{r}^{q,s,\sigma} \) to be the family of those \( f \in L_{\mathrm{loc}}^q(\CC) \) such that
\[
\|f\|_{q,s,\sigma,r} = \bigl\| \rho^{\sigma} \mathrm{MO}_{q,r}(f) \bigr\|_{L^s} < \infty.\]
Recall that $\gamma = n(1/q-1/p)$. If \( H_f, H_{\overline{f}} : F_{\varphi}^p \to L_{\varphi}^q \) are simultaneously bounded, then for $1\leq p\leq q<\infty$, Theorem \ref{Hfpq1} implies
$$\rho(z)^{2\gamma} G_{q,r}(f)(z) \leq C \|H_f\|_{F_\varphi^p \to L_\varphi^q} \quad \text{and} \quad \rho(z)^{2\gamma} G_{q,r}(\overline{f})(z) \leq C \|H_{\overline{f}}\|_{F_\varphi^p \to L_\varphi^q}.$$
An argument similar to that of \cite[Proposition~3.6]{LWH25} gives
$$\rho(z)^{2\gamma} \mathrm{MO}_{q,r}(f)(z) \leq C \left\{\|H_f\|_{F_\varphi^p \to L_\varphi^q} + \|H_{\overline{f}}\|_{F_\varphi^p \to L_\varphi^q}\right\}.$$
Conversely,  since $\mathrm{MO}_{q,r}(f)(z)=\mathrm{MO}_{q,r}(\overline{f} )(z)$, it is trivial to see that
$$ G_{q,r}(f)(z)+G_{q,r}(\overline{f})(z)\leq 2\mathrm{MO}_{q,r}(f)(z)$$
which, together with Theorem \ref{Hfpq} implies that
$$\|H_f\|_{F_\varphi^p \to L_\varphi^q} + \|H_{\overline{f}}\|_{F_\varphi^p \to L_\varphi^q}\simeq\bigl\| \rho^{2\gamma} \mathrm{MO}_{q,r}(f) \bigr\|_{L^\infty}.$$
Similarly, for $1\leq q<p<\infty$, Theorem \ref{Hfqp} implies that
$$\|H_f\|_{F_\varphi^p \to L_\varphi^q} + \|H_{\overline{f}}\|_{F_\varphi^p \to L_\varphi^q}\simeq \| \mathrm{MO}_{q,r}(f)\|_{L^{n/\gamma}}.$$
For every $1\leq t<\infty$, $f\in\mathcal Q_t$ if and only if
$\overline f\in\mathcal Q_t$, since $|\overline f|=|f|$. Based on the above facts, the following two theorems follow immediately from Theorems \ref{Hfpq1} and \ref{Hfqp}, respectively.

\begin{thm} \label{ffpq}
Let \( 1 \leq p \leq q < \infty \).  Suppose \( f \in\mathcal{Q}_p \) and $\varphi\in \mathcal W_{\mathrm{gc}}^*(\CC)$. Then \\
 \textup{(A)} \( H_f, H_{\overline{f}} : F_{\varphi}^p \to L_{\varphi}^q \) are simultaneously bounded if and only if \( f \in \mathrm{IMO}_r^{q, \infty, 2\gamma} \)
for some (or any) \( r > 0 \). Furthermore,
\[
\|H_f\|_{F_{\varphi}^p \to L_{\varphi}^q}+ \|H_{\overline f}\|_{F_{\varphi}^p \to L_{\varphi}^q}\approx \bigl\| \rho^{2\gamma} \mathrm{MO}_{q,r}(f) \bigr\|_{L^\infty}.
\]
 \textup{(B)} \( H_f, H_{\overline{f}} : F_{\varphi}^p \to L_{\varphi}^q \) are simultaneously compact if and only if $$\lim_{|z|\rightarrow\infty} \rho^{2\gamma}(z) \mathrm{MO}_{q,r}(f)(z)=0 $$ for some (or any) \( r > 0 \).
\end{thm}

\begin{thm}\label{ffqp}
Let \( 1 \leq q < p < \infty \). Suppose \( f \in\mathcal{Q}_q \) and $\varphi\in \mathcal W_{\mathrm{gc}}^*(\CC)$. Then the following statements are equivalent:\\
 \textup{(A)}\( H_f, H_{\overline{f}} : F_{\varphi}^p \to L_{\varphi}^q \) are simultaneously bounded.\\
 \textup{(B)} \( H_f, H_{\overline{f}} : F_{\varphi}^p \to L_{\varphi}^q \) are simultaneously compact.\\
 \textup{(C)} \( f \in \mathrm{IMO}_r^{q,n/\gamma,0} \) for some (or any) \( r > 0 \).
\end{thm}

When $f$ is holomorphic, it is trivial that $H_f = 0$. Therefore we have the following theorem on Hankel operators with conjugate holomorphic symbols.
\begin{thm}\label{fpq1}
Let \(1 \le p, q < \infty\). Suppose  \(f \in \mathcal{Q}_{\min\{p,q\}} \cap H(\CC)\) and \(\varphi \in \mathcal{W}^*_{\mathrm{gc}}(\CC)\).  Then\\
 \textup{(A)} For \(p \le q\), \(H_{\overline{f}}\) is bounded from \(F_{\varphi}^p\) to \(L_{\varphi}^q\) if and only if \(\rho^{2\gamma+1} |\nabla f| \in L^\infty\); \(H_{\overline{f}}\) is compact from \(F_{\varphi}^p\) to \(L_{\varphi}^q\) if and only if \(\lim_{|z| \to \infty} \rho^{2\gamma+1} |\nabla f(z)| = 0\).\\
 \textup{(B)} For \(p > q\), \(H_{\overline{f}}\) is bounded from \(F_{\varphi}^p\) to \(L_{\varphi}^q\) if and only if \(H_{\overline{f}}\) is compact from \(F_{\varphi}^p\) to \(L_{\varphi}^q\) if and only if \(\rho |\nabla f| \in L^{n/\gamma}\).
\end{thm}

\begin{proof}
By Theorems \ref{ffpq} and \ref{ffqp}, it suffices to prove that for fixed $r>0$ there are two positive constants $C_1$ and $C_2$ such that 
$$C_1 \rho(z) \left|\frac{\partial f}{\partial z_j}(z)\right| \leq \mathrm{MO}_{q,r}(f)(z) \leq C_2 \sup_{\xi \in B^r(z)} \rho(\xi) |\nabla f(\xi)|$$
for any $j=1,\ldots,n$ and every $z\in \CC$. Since $f$ is holomorphic, the mean value property gives $\widehat f_r(z)=f(z)$, and the fundamental theorem of calculus gives
\[
\mathrm{MO}_{q,r}(f)(z) = \left( \frac{1}{|B^r(z)|} \int_{B^r(z)} |f(w) - f(z)|^q \, dV(w) \right)^{1/q}\leq C\sup_{\xi \in B^r(z)} \rho(\xi) |\nabla f(\xi)|.
\]

On the other hand, by Cauchy’s formula,
\begin{align*}
\frac{\partial f}{\partial z_j}(z)= \frac{1}{2\pi} \int_{0}^{2\pi} \frac{f(z_1,\ldots,z_j + se^{i\theta},\ldots,z_n) - f(z)}{se^{i\theta}} \, d\theta,
\end{align*}
and integrating both sides in the variable $s$ from $\frac{r}{4}\rho(z)$ to $\frac{r}{2}\rho(z)$ yields
\begin{align}\label{1625}
\frac{3r^2}{32}\rho(z)\frac{\partial f}{\partial z_j}(z)
&=
\frac{1}{2\pi\rho(z)}
\int_{\frac r4\rho(z)}^{\frac r2\rho(z)}
\int_0^{2\pi}
\frac{
f(z_1,\ldots,z_j+se^{i\theta},\ldots,z_n)-f(z)
}{
se^{i\theta}
}
s\,d\theta\,ds.
\end{align}
 Set $x:=(z_1,\ldots,z_j+se^{i\theta},\ldots,z_n)$. Then $|x-z|< \frac{r}{2}\rho(z)$, and $\rho(x)\simeq \rho(z)$. It follows from the subharmonicity of $|f-f(z)|$ and \eqref{Bzwz} that for \(m>0\) sufficiently small,
\begin{align*}
|f(x)-f(z)|&\leq \frac{1}{|B^{mr}(x)|}
\int_{B^{mr}(x)}|f(w)-f(z)|\,dV(w)\\
&\lesssim\frac{1}{|B^{r}(z)|}\int_{B^{r}(z)}|f(w)-f(z)|\,dV(w)\\
&\lesssim \mathrm{MO}_{q,r}(f)(z)
\end{align*}
which, together with \eqref{1625} further implies that
$$\rho(z)\left|\frac{\partial f}{\partial z_j}(z)\right|\lesssim \mathrm{MO}_{q,r}(f)(z)$$
for all  $j=1,\ldots,n$. Lemma~\ref{rhoB}\textup{(A)} gives
\[
\rho(z)^{2\gamma}\sup_{\xi\in B^r(z)}
\rho(\xi)|\nabla f(\xi)|
\lesssim
\sup_{\xi\in B^r(z)}
\rho(\xi)^{2\gamma+1}|\nabla f(\xi)|.
\]
The growth estimates used in Lemma~\ref{kz0} show that these balls
escape every compact set as $|z|\to\infty$. Thus the left hand side
tends to zero whenever $\rho^{2\gamma+1}|\nabla f|$ does. For every
finite $t\geq1$, the local mean value estimate applied to the
holomorphic functions $\partial_jf$ on slightly larger balls,
followed by Lemma~\ref{rhoB} and Fubini's theorem, gives
\[
\left\|\sup_{\xi\in B^r(\,\cdot\,)}
\rho(\xi)|\nabla f(\xi)|\right\|_{L^t}
\lesssim
\|\rho|\nabla f|\|_{L^t}.
\]
These observations yield the required assertions for the cases
$p\leq q$ and $q<p$. This completes the proof.
\end{proof}


\begin{thebibliography}{99}

\bibitem{AC19}
H.~Arroussi and C.~Tong,
Weighted composition operators between large Fock spaces in several
complex variables,
\emph{J. Funct. Anal.} \textbf{277} (2019), 3436--3466.

\bibitem{AHV24}
G.~Asghari, Z.~Hu, and J.~A.~Virtanen,
Schatten class Hankel operators on doubling Fock spaces and the
Berger--Coburn phenomenon,
\emph{J. Math. Anal. Appl.} \textbf{540} (2024),
Article No.~128596.

\bibitem{A86}
S.~Axler,
The Bergman space, the Bloch space, and commutators of multiplication
operators,
\emph{Duke Math. J.} \textbf{53} (1986), 315--332.

\bibitem{W05}
W.~Bauer,
Mean oscillation and Hankel operators on the Segal--Bargmann space,
\emph{Integral Equations Operator Theory} \textbf{52} (2005), 1--15.

\bibitem{BC87}
C.~A.~Berger and L.~A.~Coburn,
Toeplitz operators on the Segal--Bargmann space,
\emph{Trans. Amer. Math. Soc.} \textbf{301} (1987), 813--829.

\bibitem{BC94}
C.~A.~Berger and L.~A.~Coburn,
Heat flow and Berezin--Toeplitz estimates,
\emph{Amer. J. Math.} \textbf{116} (1994), 563--590.

\bibitem{BA82}
B.~Berndtsson and M.~Andersson,
Henkin--Ramirez formulas with weight factors,
\emph{Ann. Inst. Fourier (Grenoble)} \textbf{32} (1982), 91--110.

\bibitem{Cho12}
H.~R.~Cho and K.~Zhu,
Fock--Sobolev spaces and their Carleson measures,
\emph{J. Funct. Anal.} \textbf{263} (2012), 2483--2506.

\bibitem{C91}
M.~Christ,
On the $\bar\partial$ equation in weighted $L^2$ norms in
$\mathbb C^1$,
\emph{J. Geom. Anal.} \textbf{1} (1991), 193--230.

\bibitem{CH21}
L.~Coburn, M.~Hitrik, J.~Sj\"ostrand, and F.~White,
Weyl symbols and boundedness of Toeplitz operators,
\emph{Math. Res. Lett.} \textbf{28} (2021), 681--696.

\bibitem{CO11}
O.~Constantin and J.~Ortega-Cerd\`a,
Some spectral properties of the canonical solution operator to
$\bar\partial$ on weighted Fock spaces,
\emph{J. Math. Anal. Appl.} \textbf{377} (2011), 353--361.

\bibitem{DA15}
G.~M.~Dall'Ara,
Pointwise estimates of weighted Bergman kernels in several complex
variables,
\emph{Adv. Math.} \textbf{285} (2015), 1706--1740.

\bibitem{D98}
H.~Delin,
Pointwise estimates for the weighted Bergman projection kernel in
$\mathbb C^n$, using a weighted $L^2$ estimate for the
$\bar\partial$ equation,
\emph{Ann. Inst. Fourier (Grenoble)} \textbf{48} (1998), 967--997.

\bibitem{DJ12}
J.-P.~Demailly,
\emph{Complex Analytic and Differential Geometry},
OpenContent Book, Institut Fourier,
Universit\'e Grenoble I, 2012.

\bibitem{GW79}
R.~E.~Greene and H.~Wu,
$\mathcal C^\infty$ approximations of convex, subharmonic, and
plurisubharmonic functions,
\emph{Ann. Sci. \'Ec. Norm. Sup\'er.} \textbf{12} (1979), 47--84.

\bibitem{HV21}
R.~Hagger and J.~A.~Virtanen,
Compact Hankel operators with bounded symbols,
\emph{J. Operator Theory} \textbf{86} (2021), 317--329.

\bibitem{H01}
J.~Heinonen,
\emph{Lectures on Analysis on Metric Spaces},
Universitext, Springer-Verlag, New York, 2001.

\bibitem{HL11}
Z.~Hu and X.~Lv,
Toeplitz operators from one Fock space to another,
\emph{Integral Equations Operator Theory} \textbf{70} (2011), 541--559.

\bibitem{HV22}
Z.~Hu and J.~A.~Virtanen,
Schatten class Hankel operators on the Segal--Bargmann space and the
Berger--Coburn phenomenon,
\emph{Trans. Amer. Math. Soc.} \textbf{375} (2022), 3733--3753.

\bibitem{HV23}
Z.~Hu and J.~A.~Virtanen,
IDA and Hankel operators on Fock spaces,
\emph{Anal. PDE} \textbf{16} (2023), 2041--2077.

\bibitem{HW18}
Z.~Hu and E.~Wang,
Hankel operators between Fock spaces,
\emph{Integral Equations Operator Theory} \textbf{90} (2018),
Paper No.~37.

\bibitem{K92}
S.~G.~Krantz,
\emph{Function Theory of Several Complex Variables},
2nd ed., Wadsworth \& Brooks/Cole, Pacific Grove, CA, 1992.

\bibitem{L01}
N.~Lindholm,
Sampling in weighted $L^p$ spaces of entire functions in
$\mathbb C^n$ and estimates of the Bergman kernel,
\emph{J. Funct. Anal.} \textbf{182} (2001), 390--426.

\bibitem{LWH25}
G.~Liu, X.~Wang, and W.~Huang,
Bounded, compact and Schatten class Hankel operators on large Fock
spaces,
\emph{Banach J. Math. Anal.} \textbf{19} (2025),
Paper No.~40.

\bibitem{L92}
D.~H.~Luecking,
Characterizations of certain classes of Hankel operators on the
Bergman spaces of the unit disk,
\emph{J. Funct. Anal.} \textbf{110} (1992), 247--271.

\bibitem{LW24}
X.~Lv and E.~Wang,
Hankel operators on doubling Fock spaces,
\emph{J. Math. Anal. Appl.} \textbf{531} (2024),
Paper No.~127780.

\bibitem{Mm03}
N.~Marco, X.~Massaneda, and J.~Ortega-Cerd\`a,
Interpolating and sampling sequences for entire functions,
\emph{Geom. Funct. Anal.} \textbf{13} (2003), 862--914.

\bibitem{Mo09}
J.~Marzo and J.~Ortega-Cerd\`a,
Pointwise estimates for the Bergman kernel of the weighted Fock space,
\emph{J. Geom. Anal.} \textbf{19} (2009), 890--910.

\bibitem{MP95}
P.~Mattila,
\emph{Geometry of Sets and Measures in Euclidean Spaces:
Fractals and Rectifiability},
Cambridge Stud. Adv. Math., vol.~44,
Cambridge University Press, Cambridge, 1995.

\bibitem{N03}
N.~Nikolov and P.~Pflug,
Behavior of the Bergman kernel and metric near convex boundary points,
\emph{Proc. Amer. Math. Soc.} \textbf{131} (2003), 2097--2102.

\bibitem{Phung24}
T.~T.~Phung,
$L^p$ estimates for the Bergman projection on generalized Fock spaces,
\emph{J. Geom. Anal.} \textbf{34} (2024),
Paper No.~171.

\bibitem{Sv12}
A.~Schuster and D.~Varolin,
Toeplitz operators and Carleson measures on generalized
Bargmann--Fock spaces,
\emph{Integral Equations Operator Theory} \textbf{72} (2012), 363--392.

\bibitem{SY}
K.~Seip and E.~H.~Youssfi,
Hankel operators on Fock spaces and related Bergman kernel estimates,
\emph{J. Geom. Anal.} \textbf{23} (2013), 170--201.

\bibitem{WTH}
X.~Wang, Z.~Tu, and Z.~Hu,
Bounded and compact Toeplitz operators with positive measure symbol
on Fock-type spaces,
\emph{J. Geom. Anal.} \textbf{30} (2020), 4324--4355.

\bibitem{ZHW}
Z.~Zeng, Z.~Hu, and X.~Wang,
Bounded, compact and Schatten class Hankel operators on Fock-type
spaces,
\emph{Trans. Amer. Math. Soc.} \textbf{378} (2025), 805--849.

\bibitem{Zhu07}
K.~Zhu,
\emph{Operator Theory in Function Spaces},
2nd ed., Math. Surveys Monogr., vol.~138,
Amer. Math. Soc., Providence, RI, 2007.

\bibitem{Zhu12}
K.~Zhu,
\emph{Analysis on Fock Spaces},
Grad. Texts in Math., vol.~263,
Springer, New York, 2012.

\end{thebibliography}
\end{document}